\documentclass[reqno,11pt]{preprint}
\usepackage{geometry}
\usepackage{hyperref}
\usepackage[full]{textcomp}
\usepackage[osf]{newtxtext}

\usepackage[cal=boondoxo]{mathalfa}
\usepackage{microtype}
\usepackage{shuffle}
\usepackage{longtable}
\usepackage{scalerel}
\usepackage{forest}
\usepackage{amsmath, stmaryrd}
\usepackage{amsfonts}
\usepackage{amssymb}
\usepackage{mathrsfs}

\usepackage[all]{xy}
\usepackage[utf8]{inputenc} 
\usepackage{comment}

\usepackage{hyperref}
\usepackage{breakurl}

\usepackage{mathrsfs}
\usepackage{enumitem}
\usepackage{booktabs}

\usepackage{tikz}
\usetikzlibrary{cd}
\usepackage{dsfont}
\usepackage{mhequ} 
\usepackage{mhsymb} 
\usepackage{mhenvs} 
\usepackage{array}

\usepackage{wasysym}
\usepackage{centernot}
\usepackage{mathtools}

\setlist{noitemsep,nolistsep,leftmargin=1.7em}
\usetikzlibrary{snakes}
\usetikzlibrary{decorations}
\usetikzlibrary{decorations.markings}
\usetikzlibrary{positioning}
\usetikzlibrary{shapes}
\usetikzlibrary{external}

\providecommand{\figures}{false}
{ \ifthenelse{\equal{\figures}{false}} {#1}{\[ {\rm Figure \ missing !} \]} }{}
\def\id{\mathrm{id}}

\newcommand{\e}{{\mathfrak{e}}}
\newcommand{\smallM}{\scaleto{M}{4pt}}
\newcommand{\smallF}{\scaleto{F}{4pt}}
\newcommand{\smallMSR}{\scaleto{MSR}{4pt}}
\newcommand{\smallB}{\scaleto{B}{4pt}}
\newcommand{\smallW}{\scaleto{W}{4pt}}
\newcommand{\DEC}{\scaleto{DEC}{4pt}}

\renewcommand{\N}{\mathbb{N}}
\renewcommand{\k}{\mathbf{k}}

\newcommand{\cF}{{\mathcal F}}
\newcommand{\cT}{{\mathcal T}}

\renewcommand{\R}{\mathbb{R}}

\def\${|\!|\!|}

\def\Mren{\mathcal{R}}

\definecolor{Albicolor}{rgb}{0.8, 0.2, 0.2}
\newcommand{\alberto}[1]{\textbf{\textcolor{Albicolor}{#1}}}

\definecolor{cadmiumgreen}{rgb}{0.0, 0.42, 0.24}

\newenvironment{DIFnomarkup}{}{} 

\theorembodyfont{\rmfamily}
\newtheorem{example}[lemma]{Example}

\newfont{\indic}{bbmss12}

\def\un#1{\hbox{{\indic 1}$_{#1}$}}

\colorlet{symbols}{blue!90!black}
\colorlet{testcolor}{green!60!black}
\colorlet{connection}{red!30!black}
\def\symbol#1{\textcolor{symbols}{#1}}
\def\symbol#1{\textcolor{symbols}{#1}}

\tikzset{
root/.style={circle,fill=black!50,inner sep=0pt, minimum size=3mm},
        circ/.style={circle,fill=white,draw=black,very thin,inner sep=.5pt, minimum size=1.2mm},
        round1/.style={fill=white,outer sep = 0,inner sep=2pt,rounded corners=1mm,draw,text=black,thin,minimum size=1.2mm},
          circ1/.style={circle,fill=red!10,draw=red,very thin,inner sep=.5pt, minimum size=1.2mm},
        rect/.style={fill=white,outer sep = 0,inner sep=2pt,rectangle,draw,text=black,thin,minimum size=1.2mm},
        rect1/.style={fill=white,outer sep = 0,inner sep=2pt,rectangle,draw,text=black,thin,minimum size=1.2mm},
        round2/.style={fill=red!10,outer sep = 0,inner sep=2pt,rounded corners=1mm,draw,text=black,thin,minimum size=1.2mm},
       round3/.style={fill=blue!10,outer sep = 0,inner sep=2pt,rounded corners=1mm,draw,text=black,thin,minimum size=1.2mm}, 
       mid arrow/.style={postaction={decorate}, decoration={markings, mark=at position 0.7 with {\arrow{>}}}},
        rect2/.style={fill=black!10,outer sep = 0,inner sep=2pt,rectangle,draw,text=black,thin,minimum size=1.2mm},
        dot/.style={circle,fill=black,inner sep=0pt, minimum size=2mm},
		dotred/.style={circle,fill=red!40,inner sep=0pt, minimum size=2mm, node font=\small},
		dotblue1/.style={circle,fill=blue!30,inner sep=0pt, minimum size=2mm, node font=\small},
		dotblue/.style={circle,fill=blue!30,inner sep=0pt, minimum size=1.5mm},
        var/.style={circle,fill=black!10,draw=black,inner sep=0pt, minimum size=3mm},
        kernel/.style={semithick},
         diag/.style={thin,shorten >=4pt,shorten <=4pt},
        kernel1/.style={thick, mid arrow},
        kernels/.style={snake=zigzag,shorten >=2pt,shorten <=2pt,segment amplitude=1pt,segment length=4pt,line before snake=2pt,line after snake=5pt,},
		kernels1/.style={snake=zigzag,segment amplitude=0.5pt,segment length=2pt},
		rho1/.style={densely dotted,semithick},
        rho/.style={densely dashed,semithick,shorten >=2pt,shorten <=2pt},
           testfcn/.style={dotted,semithick,shorten >=2pt,shorten <=2pt},
           visible/.style={draw, circle, fill, inner sep=0.25ex},
        renorm/.style={shape=circle,fill=white,inner sep=1pt},
        labl/.style={shape=rectangle,fill=white,inner sep=1pt},
        xic/.style={very thin,circle,fill=symbols,draw=black,inner sep=0pt,minimum size=1.2mm},
        xi/.style={very thin,circle,fill=blue!10,draw=black,inner sep=0pt,minimum size=1.2mm},
	xib/.style={very thin,circle,fill=blue!10,draw=black,inner sep=0pt,minimum size=1.6mm},
	xie/.style={very thin,circle,fill=green!50!black,draw=black,inner sep=0pt,minimum size=1mm},
	xid/.style={very thin,circle,fill=symbols,draw=black,inner sep=0pt,minimum size=1.6mm},
	edgetype/.style={very thin,circle,draw=black,inner sep=0pt,minimum size=5mm},
	nodetype/.style={very thick,circle,draw=black,inner sep=0pt,minimum size=5mm},
	kernels2/.style={very thick,draw=connection,segment length=12pt},
clean/.style={thin,circle,fill=black,inner sep=0pt,minimum size=1mm},	not/.style={thin,circle,fill=symbols,draw=connection,fill=connection,inner sep=0pt,minimum size=0.8mm},
	>=stealth,
        }

\tikzset{ individus/.style={scale=0.40,draw,circle,thick,fill=black!10},
 individu/.style={scale=0.40,draw,circle,thick,fill=black!50},       } 
\makeatletter
\def\DeclareSymbol#1#2#3{\expandafter\gdef\csname MH@symb@#1\endcsname{\tikz[baseline=#2,scale=0.15,draw=symbols]{#3}}\expandafter\gdef\csname MH@symb@#1s\endcsname{\scalebox{0.7}{\tikz[baseline=#2,scale=0.15,draw=symbols]{#3}}}}
\def\<#1>{\csname MH@symb@#1\endcsname}
\makeatother

 \def\1{\mathbf{\symbol{1}}}

 \def\un#1{\hbox{{\indic 1}$_{#1}$}}
\def\one{\mathbf{1}}

\def\bulletalpha{
	\begin{tikzpicture}[scale=0.15,baseline=0.1cm]
		\node at (0,1.2)  [dotred] {};
	\end{tikzpicture}
	}
\def\bulletbeta{
		\begin{tikzpicture}[scale=0.15,baseline=0.1cm]
			\node at (0,1.2)  [dotblue] {};
		\end{tikzpicture}
	}
	
\def\bulletbetadec{
		\begin{tikzpicture}[scale=0.15,baseline=0.1cm]
			\node at (0,1.2)  [dotblue] {\mbox{\small $i$}};
		\end{tikzpicture}
	}
	
\def\doublebullet{
		\begin{tikzpicture}[scale=0.2,baseline=0.1cm]
		\node at (0,1) [dotblue] (root) {\mbox{\small $i$}};
		\node at (2,1) [dotblue] (center) {\mbox{\small $i$}};
	\end{tikzpicture}
	}

\DeclareMathAlphabet{\mathpzc}{OT1}{pzc}{m}{it}

\def\id{\mathrm{id}}

\begin{document}

\title{Exotic B-series representation of the Feller semigroup for It\^o diffusions and the MSR path integral}
\author{A. Bonicelli}

\institute{ Laboratoire de Probabilit\'es Statistique et Mod\'elisation,\\ Sorbonne Université, Paris, France \\
\email{bonicelli@lpsm.paris}
}

\maketitle

\begin{abstract}
	In this paper we consider the expansion of the Feller semigroup of a one-dimensional It\^o diffusion as a power series in time. Taking our moves from previous results on expansions labelled by exotic trees, we derive an explicit expression for the combinatorial factors involved, that leads to an exotic Butcher series representation. A key step is the extension of the notion of tree factorial and Connes-Moscovici weight to this richer family of rooted trees. The ensuing expression is suitable for a comparison with the perturbative path integral construction of the statistics of the diffusion, known in the literature as Martin-Siggia-Rose formalism. Resorting to multi-indices to represent pre-Feynman diagrams, we show that the latter coincides with the exotic B-series representation of the semigroup, giving it a solid mathematical foundation. 
\end{abstract}

\setcounter{tocdepth}{2}
\tableofcontents

\section{Introduction}

In its seminal paper \cite{butcher1963coefficients} Butcher introduced a series representation of the Taylor expansion of Runge–Kutta numerical integrators for ordinary differential equations, which is nowadays referred to as Butcher series. It lays its foundations on a one-to-one correspondence between non-planar rooted trees and specific vector fields discovered by Cayley \cite{cayley1857xxviii} and it consists of representing the numerical scheme as a power series in the time step with coefficients completely characterized by the labelling trees. To wit, according to \cite{Hairer2006}, given the ODE
\begin{equation*}
	\dot{y}=g(y)\,,\qquad y(0)=y_0\,,
\end{equation*}
a numerical scheme $Y_{n+1}=\Psi_n[y_0,g,\delta]$ of step $\delta$ is a B-series if it can be represented as
\begin{equation*}
	B(a,g,y_0,t)=\sum_{\tau\in\cT}\frac{a(\tau)}{\sigma(\tau)}\Upsilon_{u_0}^g(\tau)\,t^{|\tau|}\,,
\end{equation*}
where $\cT$ is the space of non-planar rooted trees, \emph{i.e.}, connected and simply connected graphs $\tau\in\cT$ identified with a set of vertices $\mathcal{V}(\tau)$ and a set of edges $\mathcal{E}(\tau)$ such that for each $e\in\mathcal{E}(\tau)$ there exist $e_+,e_-\in\mathcal{V}(\tau)$ such that $e=(e_+,e_-)$. Each tree has a distinguished vertex, the root, having no outgoing edges. All other vertices have a single outgoing edge and an arbitrary number of incoming edges, the latter being oriented towards the root. Here $\sigma(\tau)$ accounts for the symmetries of the tree, $|\tau|$ denotes the number of vertices of $\tau$ and $\Upsilon_{u_0}^g(\tau)$, called the elementary differential, realizes the aforementioned isomorphism between trees and vector fields. The functional $a:\mathcal{T}\to\R$ is determined by the coefficients of the numerical scheme. This groundbreaking idea paved the way for the algebraic study of numerical integrators of ordinary differential equations. 
We refer the reader interested in a comprehensive introduction to B-series to the monograph \cite{butcher2021b}, while a detailed account of the algebraic structures involved can be found in \cite{chartier2010algebraic}.

Since Butcher's seminal paper there has been an increasing interest on generalizations of B-series to accommodate different aspects of the numerical analysis of differential equations, featuring larger classes of combinatorial trees. In this paper we are mainly interested in the generalization of Butcher series to the realm of stochastic differential equations, where they proved efficient as a tool for strong and weak error analysis \cite{komori1997rooted, rossler2006rooted} and the derivation of order conditions for sampling the invariant measure of an ergodic SDE \cite{laurent2020exotic, laurent2022order, bronasco2024hopf}. The latter problem prompted the introduction of exotic S-series as a generalization of B-series over exotic forests, \emph{i.e.}, forests of trees endowed with a decoration that specifies the deterministic or stochastic origin of a vertex. In addition, such decoration codifies the information that pairs of stochastic vertices are contracted when taking expectations. 
If we consider the exact solution of a generic stochastic differential equation, it was shown in \cite[\S 5]{platen2013} that an iterative application of It\^o's lemma provides us with a so call It\^o-Taylor expansion of $f(u_t)$ over multiple stochastic integrals labelled by hierarchical sets. Yet, if we are interested in expectations with respect to the law of the Wiener process, such generalization of Taylor's expansion greatly simplifies, leaving us with the sole deterministic iterated integrals. In \cite{rossler2004stochastic}, this observation led to the reformulation of the It\^o-Taylor expansion of $\mathbb{E}^{u_0}[f(u_t)]$ as an exotic S-series over labelled, exotic forests. 

In this paper we search for an explicit form of the functional on trees that identifies an exotic B-series representation of the Feller semigroup of the solution to the Cauchy problem on $[0,T]\subset\R_+$
\begin{align}\label{Eq: SDE}
	du=\alpha(u)\,dt+\beta(u)\,dW\,, \qquad	u(0)=u_0\,,
\end{align}
characterized by smooth and homogeneous drift and diffusion $\alpha,\beta\in C^\infty(\R)$ and fixed initial condition $u_0\in\R$. Here $\{W_t\}_{t\geq0}$ is a Wiener process on the underlying probability space $(\Omega,\mathcal{F},\mathbb{P})$. 
Assuming that a unique solution exists on $[0,T]$, then the statistics of the solution is described by the action of the semigroup $T_t:C_0(\R)\to C_0(\R)$ generated by the operator $L:=\alpha\partial_x+\frac{1}{2}\beta^2\partial_x^2$, 
\begin{equation*}
	T_tf(u_0):=\mathbb{E}^{u_0}[f(u_t)]\,,\hspace{2cm}f\in C_0(\R)\,.
\end{equation*} 
where $C_0(\R)$ is the space of continuous functions vanishing at infinity. 
Hereafter, $\mathbb{E}^{u_0}$ will denote the conditional expectation with respect to the constraint $u(0)=u_0$.
To derive a closed form of the combinatorial factor that characterizes the B-serie expansion of $\mathbb{E}^{u_0}[f(u_t)]$ some work is required. 
To wit, the iterative application of It\^o's lemma at the heart of the It\^o-Taylor expansion suggests that the contributions to the expansion are obtained by an operation that we call natural growth, first considered in \cite{kreimer1999chen} in the context of renormalization of quantum field theories. It consists of iteratively grafting single vertices of colour $\alpha$ and double vertices of colour $\beta$ in all possible ways, starting from the root. This operation translates at the level of trees the action of the two contribution to the differential operator $L$. Therefore we must find a way of counting, for each exotic tree, the number of ways of obtaining it via natural growth. Even if in the non-decorated setting this number is well-known in the literature as the Connes-Moscovici weight, see \cite{brouder2000runge}, to our knowledge no analogous notion has been proposed for the class of exotic trees that we use. The task of finding its suitable extension is addressed in Section \ref{Sec: exotic s-series} and leads to the following B-series representation of the Feller semigroup. 
\begin{proposition}\label{Prop: key result}
	Let $\alpha,\beta,f\in C^\infty(\R)$ be real analytic functions such that  \eqref{Eq: SDE} admits a unique solution $u$  with initial condition $u_0\in\R$ on $[0,T]$. Then the Feller semigroup associated to the solution of \eqref{Eq: SDE} admits the formal power series representation
	\begin{equation}\label{Eq: convergent exact expansion expectation intro}
		\mathbb{E}^{u_0}[f(u_t)]=\sum_{\tau\in\cT_{\alpha,\beta}^\mathfrak{e}}\frac{|\tau|}{\sigma_\mathfrak{e}(\tau)\tau!}\Upsilon_{u_0}^{\alpha,\beta,f}(\tau)\,t^{|\tau|-1}\,,
	\end{equation}
	where the sum runs over all exotic coloured trees lying in $\cT_{\alpha,\beta}^\e$.
\end{proposition}
This result is of independent interest and could find application as a refinement of the core ideas of \cite{laurent2022order, laurent2020exotic}. In addition, the study of sophisticated algebraic structures characterizing the space of exotic aromatic trees \cite{laurent2022order} could translate in useful operations on the semigroup, such as the change of integration prescription.

This representation of the expectations for the solution of an SDE resonates with the perturbative study of their representation in terms of functional integrals. In the early 70s Martin, Siggia and Rose \cite{martin1973statistical} proposed for the first time a representation of the random observables of a disordered system as operators endowed with suitable commutation relations with auxiliary, conjugate fields. The Martin-Siggia-Rose formalism was further elaborated by Janssen, De Dominicis and Peliti \cite{janssen1976lagrangean, de1978field}, who derived an alternative representation of the generating functional for the probability distribution of the stochastic dynamics as an infinite-dimensional integral over a suitable space of paths with constrained initial condition. In analogy with the path integral formulation of Euclidean quantum field theories, the equations of motion of the (random) fields identify a class of actions used to define the associated Gibbs measure. This fruitful approach opened up the path to the plethora of techniques adopted to study perturbative quantum field theories, as \emph{e.g.} Feynman diagrams expansions of expectations \cite{de2006random} and renormalization group analysis \cite{cardy1996renormalisation}. 

Despite its effectiveness as a tool to compute quantities of interest for the aforementioned models, this approach suffers from the problem that its derivation is not entirely rigorous, since it often involves a discretization procedure in time that, when talking the continuum limit, would yield functional integrals with respect to a non-existing infinite dimensional Lebesgue measure. Efforts to circumvent this issue have been made in \cite{butko2016feynman}, where the Feller semigroup for a large class of stochastic processes is constructed via a finite-dimensional approximation based on the Chernoff's theorem. 
Complementarily, in \cite{bonicelli2023algebraic} the authors addressed the problem of showing that, at every perturbative order, expectations of functions of the solution labelled by suitable Feynman diagrams coincide with the ones obtained by directly expanding the solution as a formal power series of the perturbative parameter. Their approach relies on techniques typical of algebraic quantum field theory \cite{rejzner2016perturbative} and from their applications to singular stochastic PDEs \cite{bonicelli2023schrodinger, bonicelli2023stochastic, bonicelli2023microlocal}
to systematically compute the contractions forming Feynman diagrams via a deformation of the algebraic structures involved. 

The second part of this paper aims at extending the results of \cite{bonicelli2023algebraic} to the case of a generic It\^o diffusion by showing that the formal series expansion of the path integral expectations coincides with the exotic B-series representation introduced above. In the following we give a concise overview of the main idea at the heart of the Martin-Siggia-Rose (MSR) construction. The MSR expectation of the composition between $f\in C^\infty(\R)$ and the solution of \eqref{Eq: SDE} at times $t\in[0,T]$ is identified with the infinite-dimensional integral
\begin{equation}\label{Eq: path integral expectation}
	\mathbb{E}^{\smallMSR}[f(u_t)]:=\int_{\mathcal{C}\times \tilde{\mathcal{C}}}e^{-S[\psi,\tilde{\psi}]}f(\psi_t)\,\mathcal{D}\psi\mathcal{D}\tilde{\psi}\,,
\end{equation}
over a space of paths $(\psi,\tilde{\psi})\in{\mathcal{C}\times \tilde{\mathcal{C}}}$ such that $\psi(0)=u_0$ and whose regularity depends on the explicit form of the coefficients. The action functional has the form 
\begin{align}\label{Eq: action}
	S[\psi&,\tilde{\psi}]:=\int_0^T\tilde{\psi}_s\dot{\psi}_s\,ds\nonumber\\
	&-\underbrace{\int_0^T\tilde{\psi}_s(\alpha(\psi_s)-\theta_0\beta^{(1)}(\psi_s)\beta(\psi_s))\,ds-\frac{1}{2}\int_0^T\tilde{\psi}_s^2\beta^2(\psi_s)\,ds+\theta_0\int_0^T\alpha^{(1)}(\psi_s)\,ds}_{S_{int}}\,.
\end{align}
where $\theta_0\in[0,1]$ is a residual freedom linked to the choice of stochastic integration when solving \eqref{Eq: SDE}. The path $\tilde{\psi}$ is an auxiliary field with no physical meaning that arises from the discrete approximation of \eqref{Eq: path integral expectation} as the dual variable of $\psi$. This representation led to the interpretation of $e^{-\int_0^T\tilde{\psi}_s\dot{\psi}_s\,ds}\mathcal{D}\psi\mathcal{D}\tilde{\psi}$ as a Gaussian measure over the Hilbert space $H^1([0,T])\oplus H^1([0,T])$ whose quadratic form reads
\begin{equation}\label{Eq: free part MSR}
	\int_0^T\tilde{\psi}_s\dot{\psi}_s\,ds=\langle(\psi,\tilde{\psi}),T(\psi,\tilde{\psi})\rangle_{L^2([0,T])\oplus L^2([0,T])}\,,\qquad T:=\begin{pmatrix}
		0 & d_t^{\ast} \\
		d_t & 0
	\end{pmatrix}\,.
\end{equation}
where $d_t$ denotes the first derivative and $d_t^\ast$ its adjoint on $H^1([0,T])$.
Namely, one would interpret $\{(\psi_t,\tilde{\psi}_t)\}_{t\in[0,T]}$ as a Gaussian vector on $(\bar{\Omega},\bar{\mathcal{F}},\bar{\mathbb{P}})$ such that
\begin{align}
	&\overline{\mathbb{E}}[\psi_t]=u_0,\qquad\overline{\mathbb{E}}[\tilde{\psi}_t]=0\qquad\forall t\in[0,T]\,,\nonumber\\
	&\overline{\mathbb{E}}[\psi_t\psi_s]=\overline{\mathbb{E}}[\tilde{\psi}_t\tilde{\psi}_s]=0\,,\qquad\overline{\mathbb{E}}[\psi_t\tilde{\psi}_s]=\Theta(t-s)\,,\qquad\forall t,s\in[0,T]\,.\label{Eq: properties random fields}
\end{align}
where $\Theta\in\CD'(\R)$ is the Heaviside theta function. 
We denote by $\overline{\mathbb{E}}$ the expectation with respect to the law of the formal free fields to distinguish it from the Gaussian measure of the noise. However, such Gaussian fields do not exists in the continuum setting since the form $T$ is not positive definite, invalidating the use of Wick's theorem to perform explicit computations. Core of this paper is to explain why the erroneous application of this theorem leads to correct results. 
 Hence, let us expand the path integral around the free action via a contraction map $\Gamma_\Theta$ that identifies pairs $(\psi,\tilde{\psi})$ with the integral kernel of $\Theta$ to express the MSR expectations as  
\begin{align}\label{Eq: perturbative expectation action}
		\mathbb{E}^{\smallMSR}[f(u_t)]&=\Gamma_\Theta\left[e^{-S_{int}[\psi,\tilde{\psi}]}f(\psi_t)\right]\big\vert_{\substack{\psi=u_0\\\tilde{\psi}=0}}\,.
\end{align} 
A rigorous definition of this expression will be given in Section \ref{Sec: deformation maps}.
We are borrowing techniques from algebraic quantum field theory in its formulation based on distribution-valued functionals \cite{rejzner2016perturbative}. The idea is to implement the information that fields $(\psi,\tilde{\psi})$ should contract into the kernel $\Theta$ by enriching the algebraic structure of suitable classes of distribution-valued functionals via a deformation. As first proposed in \cite{bonicelli2023algebraic}, this allows to implement a form of Wick's theorem suitable to compute the perturbative expectation of expectation values as per \eqref{Eq: perturbative expectation action}.

The expression in \eqref{Eq: perturbative expectation action} can be made explicit by Taylor expanding $\alpha$, $\beta$ and $f$ in powers of the argument around a fixed initial condition $u_0\in\R$.
As a result, the MSR expectation in \eqref{Eq: perturbative expectation action} boils down to a series of contributions of the form
\begin{align}\label{Eq: expectation products}
	\Gamma_{\Theta}\prod_{\substack{n,k\in\N\\
			(k_1,k_2)\in\N^2}}\Bigl(\int_0^T \tilde{\psi}_s(\psi_s-u_0)^{n}\,ds\Bigr)^{\gamma_\alpha(n)}\Bigl(\int_0^T \tilde{\psi_s}^2(\psi_s-u_0)^{k_1+k_2}\,ds\Bigr)^{\gamma_\beta(k_1,k_2)}(\psi_t-u_0)^{\gamma_\bullet(k)}\big\vert_{\substack{\psi=u_0\\ \tilde{\psi}=0}}\,,
\end{align}
for integers $\gamma_\alpha(n),\gamma_\bullet(k), \gamma_{\beta}(k_1,k_2)\in\N$. Then, a far reaching generalization of the Wick's theorem in Gaussian integration theory, known as the Linked Cluster Theorem \cite{brueckner1955many, brouder2009quantum}, guarantees that this kind of expectations give rise to connected Feynman diagrams obtained via contraction of the interaction vertices:
	\begin{align}\label{Eq: interaction vertices}
	\alpha^{(1)}(u_0)\,\tilde{\psi}\psi\quad\leftrightarrow\quad \begin{tikzpicture}[scale=0.2,baseline=0.2cm]
		\node at (-4,5)  [] (leftlr) {};
		\node at (-4,2)  [dotred] (left) {};
		\node at (-4,-1)  [] (leftr) {};
		\draw[kernel1] (leftlr) to     node [sloped,below] {\small }     (left);
		\draw[kernel1, red] (left) to     node [sloped,below] {\small }     (leftr);  
	\end{tikzpicture}
	\hspace{2cm}
	\beta(u_0)\beta^{(1)}(u_0)\,\tilde{\psi}^2\psi\quad\leftrightarrow\quad
	\begin{tikzpicture}[scale=0.2,baseline=0.2cm]
		\node at (-2,2)  [dotblue1] (left) {};
		\node at (-2.5,-1)  [] (leftr) {};
		\draw[kernel1, red] (leftr) to     node [sloped,below] {\small }     (left);
		\node at (0.5,2)  [dotblue1] (right) {};
		\node at (1,-1)  [] (rightr) {};
		\node at (0.5,5)  [] (rightl) {};
		\draw[kernel1, red] (right) to     node [sloped,below] {\small }     (rightr);    
		\draw[kernel1] (rightl) to     node [sloped,below] {\small }     (right);    
		\draw[rho1] (left) to     node [sloped,below] {\small }     (right); 
	\end{tikzpicture}
\end{align}
The resulting graph is what we call a Feynman diagram, the edges representing the propagator $\Theta$ and coloured vertices being the integration variables. The root plays a special role as the only non-integrated point.  
 We stress once again that this graphical representation of the contributions to \eqref{Eq: perturbative expectation action}, which offers an efficient way of accounting for the contributions to expectations, rests upon the heuristics that the free part of the action induces a Gaussian measure. Goal of this paper is to show that the resulting series coincide with the ones derived rigorously via a suitable representation of the semigroup of the It\^o diffusion \eqref{Eq: SDE}.
Departing from the diagram-based strategy of \cite{bonicelli2023algebraic}, we adopt the point of view of \cite{bruned2025renormalising} on the construction and renormalization of cumulants of a scalar quantum field theory. Building on ideas of \cite{berglund2022perturbation}, the authors realized how the information needed to perform the BPHZ renormalization of the theory is already stored at the level of \emph{pre-Feynman diagrams}, \emph{i.e.}, the collections of vertices with free legs identified in \eqref{Eq: interaction vertices}. The major advantage is a substantial simplification in the combinatorics of the problem, which allows for a reformulation of the Hopf-algebraic construction of \cite{berglund2022perturbation} in terms of dual structures.
In the series of papers \cite{linares2024diagram, otto2024priori} the authors introduced the idea of replacing combinatorial trees with \emph{multi-indices} in an alternative formulation of the theory of regularity structures for solving a large class of singular stochastic PDEs. This time a multi-index represents the number of vertices of the pre-Feynman diagrams with a given number of free legs. On account of the peculiar form of the interaction in the MSR action, we identify a suitable space of \emph{Feynman multi-indices} $\mathcal{M}_{\smallF}$.
We derive a result analogous to \cite[Lem. 4.12]{bruned2025renormalising} on a formal multi-index decomposition of \eqref{Eq: perturbative expectation action} as
\begin{equation}\label{Eq: MSR Butcher}
	\mathbb{E}^{\smallMSR}[f(u_t)]=\sum_{\gamma\in\mathcal{M}_{\smallF}}\frac{\Upsilon_{u_0}^{\alpha,\beta,f}(z^\gamma)}{\sigma_{\smallF}(z^\gamma)}\Pi_{t}(z^\gamma)\,,
\end{equation}
where $\Upsilon_{u_0}^{\alpha,\beta,f}$ accounts for the derivatives of $\alpha,\beta$ and $f$ coming from their Taylor expansion, while $\sigma_{\smallF}(z^\gamma)$ is an integer that encodes the combinatorial aspects of the pre-Feynman diagram linked to symmetries. The \emph{realization map} $\Pi_t(z^\gamma)$ plays a crucial role as it coincides with the combination of integrals associate to all Feynman diagrams that we can obtain form the pre-Feynman diagram $z^\gamma$. 

Returning to the path integral expectations and given the interpretation of multi-indices as pre-Feynman diagrams, we observe that, as highlighted in \cite{bruned2025renormalising}, one can always switch to Feynman diagrams by expressing \eqref{Eq: expectation products} as a combination of integrals labelled by exotic trees. Denoting by $\tau$ any of the exotic rooted trees represented by a multi-index $z^\gamma\in\mathcal{M}_{\smallF}$, it is rather natural to introduce a realization map on exotic coloured trees of the form
\begin{align}\label{Eq: realization map trees}
	\Pi_t^{\mathfrak{e}}&(\tau):=\un{(|\mathcal{V}_\beta(w)|=2k)}\int_{[0,T]^{|\tau|-1}}\prod_{\substack{e\in\mathcal{E}(\tau_i)\\ e_+=\rho_{\tau_i}}}\Theta(t-t_{e_-})\prod_{\substack{e\in\mathcal{E}(w)\\ e_+\neq\rho_\tau}}\Theta(t_{e_+}-t_{e_-})\prod_{\substack{(v,w)\in\mathcal{V}_\beta(\tau)\\\mathfrak{e}(v)=\mathfrak{e}(w)}}\delta(t_{v_1}-t_{v_2})\prod_{u\in\mathcal{V}(\tau)\setminus\{\rho(\tau)\}}dt_u\,.
\end{align}
The a priori ill-defined products of distributions in \eqref{Eq: realization map trees} can be given proper meaning via a regularization procedure which boils down to assigning a value to $\lim_{t\to 0}\Theta(t)$. Being interested in solutions \emph{a la} It\^o, we will stick to the choice $\Theta(0)=0$.
As it stands, the exotic B-series \eqref{Eq: convergent exact expansion expectation intro} is not expressed in terms of the realization map $\Pi^\mathfrak{e}$, hence preventing a direct comparison with the MSR expectations. In addition, the integrals in \eqref{Eq: realization map trees} do not fall under the scope of applicability of the branched rough path theory formulated in \cite{gubinelli2010ramification} due to the nonlocal effect of the exotic decoration, which are responsible for the presence of Dirac deltas in the multiple integral. One of the main results of the second part of this work is to show that 
\begin{equation*}
	\Pi_t^\mathfrak{e}(\tau)=\frac{|\tau|}{\tau!}t^{|\tau|-1}\,,\qquad\forall \tau\in\cT_{\alpha,\beta}^{\mathfrak{e}}\,.
\end{equation*}
where $\tau!$ is a suitable generalization to exotic trees of the tree factorial, ubiquitous in the B-series literature as a generalization of the factorial of integers that accounts for how sparse the tree is. We shall denote by $\mathbb{E}^{\smallF}[f(u_t)]$ the form of \eqref{Eq: convergent exact expansion expectation intro} in terms of the realization map $\Pi^\mathfrak{e}$, which reads
\begin{equation*}
	\mathbb{E}^{\smallF}[f(u_t)]=\sum_{\tau\in\cT_{\alpha,\beta}^\mathfrak{e}}\frac{\Upsilon_{u_0}^{\alpha,\beta,f}(\tau)}{\sigma_\mathfrak{e}(\tau)}\,\Pi_t^\e(\tau)\,.
\end{equation*}
To conclude, the last part of the paper addresses the translation of the tree-based expansion in term of Feynman multi-indices. Following the lines of \cite{bruned2025renormalising} we introduce suitable maps that intertwine between exotic trees and multi-indices:
\begin{equation*}
	\Psi^\e:\cT_{\alpha,\beta}^\mathfrak{e}\to\mathcal{M}_{\smallF}\,,\qquad \Phi^\e:\mathcal{M}_{\smallF}\to\langle\cT_{\alpha,\beta}^\mathfrak{e}\rangle\,.
\end{equation*}
The \emph{counting map} $\Psi^\e$ simply constructs a monomial $z^\gamma$ by adding an abstract variable $z_{\alpha,n}$ for each vertex of colour $\alpha$ with $n$ incoming edges and an abstract variable $z_{\beta, (k_1,k_2)}$ for each pair of vertices with colour $\beta$ with $k_1$ and $k_2$ incoming edges, respectively. 
The action of $\Phi$ on $z^\gamma\in\mathcal{M}_{\smallF}$ is defined implicitly as a weighted combination of trees $\tau_i\in\cT_{\alpha,\beta}^\mathfrak{e}$ such that $\Psi(\tau)=z^\gamma$. The weights involve the symmetry factors of $\tau_i$ and $z^\gamma$, respectively and they are crucial to show that
\begin{equation*}
	\Pi_{t}(\Phi^\e(\tau))=\Pi_t^\mathfrak{e}(\tau)\,,\qquad\forall \tau\in\cT_{\alpha,\beta}^\e\,.
\end{equation*}
With this identity at hand, we obtain the sought after identity between the MSR expectations and the exotic B-series at the heart of this work.
\begin{proposition}\label{Prop: main proposition}
	Under the assumptions of Proposition \ref{Prop: key result}, the path integral expectation $\mathbb{E}^{\smallMSR}[f(u_t)]$ coincides with the exotic B-series representing $\mathbb{E}^{u_0}[f(u_t)]$ as formal power series. 
\end{proposition}
In this work we restrict to scalar-value diffusions since the advantages of using multi-indices are manifest only in the one-dimensional setting. However, the results obtained here should be seen as a starting point for better understanding how the use of Feynman diagrams in the vector-valued case can lead to analogous conclusions. 


The result stated in Proposition \ref{Prop: main proposition} may seem counterintuitive: even if no Gaussian measure that abides by \eqref{Eq: properties random fields} exists, a procedure intimately linked to contractions yields the correct result. Our analysis unveils that this is a fortuitous coincidence. Indeed, the contractions via $\Gamma_\theta$ of the exponential of the interacting action with $f$ coincide with the grafting operations that characterize the expansion of the Feller semigroup. A highly non-trivial step in this identification is that the integral of Heaviside functions created by the contraction map boils down to the right power of the time parameter, weighted by a suitable combinatorial factor, see Proposition \ref{Prop: explicit form Pi}.
This correspondence tells us that the formal derivation of \eqref{Eq: path integral expectation} consisting of a discretization of the SDE and a series of formal manipulations, see \emph{e.g.} \cite{peliti1985field}, can be circumvented by resorting to a B-series representation of the It\^o-Taylor expansion.

\subsection*{Organization of the paper}

In Section \ref{Sec: Exotic B-series} we adopt provide a truncated expansion of the expectation values of smooth functions of an It\^o diffusion in terms of exotic coloured trees. After addressing combinatorial properties of exotic trees and, in particular, a generalization of the Connes-Moscovici weights, we represent them as exotic B-series in which the combinatorial part has an explicit form.
Eventually we find a link between the realization map on exotic coloured trees and the combinatorial information carried by the generalized Connes-Moscovici weights, that passes through a novel characterization of the tree factorial.

In the first part of Section \ref{Sec: MSR expectations via Feynman multi-indices} we introduce a perturbative, algebraic version of the MSR path integral. After motivating the use of Feynman multi-indices to describe the pre-Feynman diagrams that label MSR expectations, we introduce suitable tools to account for the analytic and combinatorial information that they codify. Then we prove one of the main results of the paper, namely the decomposition of the MSR formula for expectations in terms of multi-indices.

Section \ref{Sec: Multi-indices exotic B-series} represents the last step of our bipartite program. After characterizing the link between exotic trees and multi-indices, we draw a connection between $\Pi(z^\gamma)$ and $\Pi^\e(\Phi^\e(z^\gamma))$, a crucial result that leads to the proof of Proposition \ref{Prop: main proposition}.

In Appendix \ref{Sec: reduction}, thanks to the explicit form of the realization map $\Pi^\e$, we unveil a non-trivial combinatorial relation that involves a reduction operation on exotic trees. The latter also allows to resort to a suitable form of Chen's relation in terms of the Butcher-Connes-Kreimer coproduct to make a comparison with the definition of factorial for a generic combinatorial Hopf algebra, proposed in \cite{curry2020planarly}.
Appendix \ref{Sec: applications} is devoted to the explicit computation of the contributions to the convergent B-series representation of moments of two well-known diffusions, the Orstein-Uhlenbeck process and the geometric Brownian motion.
Lastly, in Appendix \ref{Sec: deformation maps} we briefly recall the notions at the heart of the algebraic formulation of the perturbative MSR expectations in terms of distribution-valued functionals.

\subsection{Notation}\label{Sec: notation}
Throughout the paper we denote by $C^\infty(\R)$ the space of infinitely differentiable real function and by $\mathcal{D}(\R)$ the space of test functions, \emph{i.e.}, of functions lying in $ C^\infty(\R)$ with compact support. Accordingly, $\mathcal{D}'(\R)$ denotes the space of distributions.

For any set $A$ we denote by $\langle A\rangle$ its span, \emph{i.e.}, the vector space of finite linear combinations of elements lying in $A$. Given $s\leq T\in\R_+$, we denote by $[s,T]_\leq^n$ the simplex
\begin{equation*}
	[s,T]_\leq^n:=\{(t_1,\ldots, t_n)\in[s,T]^{\times n}\,:\;s\leq t_1\leq\ldots, \leq t_n\leq T\}\,.
\end{equation*}

\subsection*{Acknowledgements}
A.B. is supported by an FSMP postdoctoral fellowship. The author gratefully acknowledges Nicolò Drago, Titus Lupu, Camille Tardif and Lorenzo Agabiti for useful discussions on an early version of the manuscript and Dominique Manchon for suggesting a comparison with the definition of factorial of a combinatorial Hopf algebra.

\section{Exotic B-series expansion of expectations}\label{Sec: Exotic B-series}

The problem of representing the expectation value of the composition of a diffusion process with a test function as a power series is well known in the numerical analysis literature. The seminal works of Burrage and Burrage \cite{burrage1996high, burrage2000order} and Komori, Mitsui, Sugiura \cite{komori1997rooted} on the derivation of order conditions for stochastic Runge-Kutta integrators paved the way for the adoption of suitable generalizations of tree-labelled series to the study of stochastic numerical schemes. Recently, the importance of \emph{exotic rooted trees} for the study of weak convergence for stochastic Runge-Kutta methods \cite{rossler2006rooted} and for sampling the invariant measure of ergodic problems \cite{laurent2020exotic, laurent2022order} became apparent. Taking our moves from their approach,  we derive an exotic B-series representation of expectations of a diffusion process with an explicit characterization of the combinatorial factors involved. 

\subsection{It\^o-Taylor expansion and expectations}\label{Sec: ito-taylor}
In this section we review the ideas at the heart of a generalization of Taylor expansion suitable for representing the composition of a diffusion process solving \eqref{Eq: SDE} with a smooth function $f\in C^\infty(\R)$. This notion, first proposed by Wagner and Platen in \cite{wagner1982taylor}, has proven to be a fundamental tool for a systematic expansion of numerical integrators for stochastic differential equations. We refer the reader to \cite[\S 5]{platen2013} for additional details, while we only provide a brief compendium of the main ides behind the so called It\^o-Taylor expansion. 
We introduce the differential operators 
\begin{equation}\label{Eq: operators}
	L:=\alpha\frac{\partial}{\partial x}+\frac{1}{2}\beta^2\frac{\partial^2}{\partial x^2}\,,\qquad J:=\beta\frac{\partial}{\partial x}\,,
\end{equation}
where $\alpha,\beta\in C^\infty(\R)$ are the drift and diffusion coefficients in the SDE \eqref{Eq: SDE}, considered in the It\^o sense. A direct application of It\^o formula yields the integral equation for $f(u_t)$:
\begin{align}\label{Eq: Ito lemma}
	f(u_t)&=f(u_0)+\int_0^{t}Lf(u_s)\,ds+\int_0^{t}Jf(u_s)\,dW_{s}\,.
\end{align}
Nothing prevents us from substituting $f(u_s)$ in the first integral of  \eqref{Eq: Ito lemma} with an analogous expression, obtaining 
\begin{align}
	f(u_t)&=f(u_0)+\int_0^{t}Lf(u_0)\,ds+\int_0^{t}\int_0^{s}L^2f(u_r)\,drds+\int_0^{t}\int_0^{s}JLf(u_r)\,dW_{r}ds+\int_0^{t}Jf(u_{s})\,dW_{s}\nonumber\\
	&=f(u_0)+\Bigl(\alpha(u_0)f^{(1)}(u_0)+\frac{1}{2}\beta^2(u_0)f^{(2)}(u_0)\Bigr)\int_0^t\,ds+R^f_t\,,\label{Eq: formal expansion E}
\end{align}
for the remainder
\begin{equation*}
	R^f_t=\int_0^{t}\int_0^{s}L^2f(u_r)\,drds+\int_0^{t}\int_0^{s}JLf(u_r)\,dW_{r}ds+\int_0^{t}Jf(u_{s})\,dW_{s}\,.
\end{equation*}
This could be considered the simplest It\^o-Taylor expansion of $f(u_t)$. An analogous substitution could have been performed inside the stochastic integral, yielding an additional term $\beta(u_0)f^{(1)}(u_0)\int_0^t\,dW_s$ and a richer remainder. This procedure can be iterated by plugging \eqref{Eq: Ito lemma} in place of each instance of $f$. Considering operators $A_1,\ldots, A_n\in\{I,J\}$, a generic contribution to  the $n$-th iteration would have the form
\begin{equation}\label{Eq: contributions ito-taylor}
	\int_{[0,t]^n_{\leq}} A_1\ldots A_n f(u_0)\prod_{i=1}^n(ds_i\un{(A_i=L)}+dW_{s_i}\un{(A_i=J)})\,,
\end{equation}
where $[0,t]^n_{\leq}:=\{(t_1,\ldots, t_n)\in[0,t]^n\,:\,0\leq t_1\leq\ldots\leq t_n\leq t\}$. These mixed iterated integrals over a simplex are referred to as multiple It\^o integrals in \cite[\S 5]{platen2013}, where the information on the nature of each integration, as well as of its component in the vector-valued setting, are stored within words called hierarchical sets. Hence, we end up with a series of multiple It\^o integrals, whose weak convergence is subordinated to the Lipschitz continuity of the coefficients  and to the polynomial growth at infinity of both $\alpha$ and $\beta$, see \cite[\S 5.11]{platen2013}, at least for $f(u)=u$. 

Since the aim of the paper is to derive an expansion of the Feller semigroup defining  $\mathbb{E}^{u_0}[f(u_t)]$, we shall focus on an It\^o-Taylor expansion of the expectation of $f(u)$ with respect to the Wiener measure. Following \cite{rossler2004stochastic} we observe that, thanks to the martingale nature of It\^o integrals, the expectation value of any iterated integral containing at least one stochastic component vanishes, leaving us with the sole integrals in time. The reader can convince him/herself that the ensuing expression for the truncated expectation at order $n\in\N$ reads
\begin{align}
	\mathbb{E}^{u_0}[f(u_t)]&=\sum_{k=0}^n\frac{t^k}{k!}L^kf(u_0)+\int_0^t\ldots \int_0^{s_n}\mathbb{E}L^{n+1}f(u_{s_{n+1}})\,ds_{n+1}\ldots ds_1=\sum_{k=0}^n\frac{t^k}{k!}L^kf(u_0)+R^n_t\,.\label{Eq: expansion Ef}
\end{align}
Observe that, in the deterministic setting $\beta=0$, \eqref{Eq: expansion Ef} reduces to the deterministic Taylor expansion of $f(u_t)$.

\begin{remark}
	By definition, the semigroup generated by $L$ yields the solution to the Cauchy problem
	\begin{equation}\label{Eq: Talay}
		\begin{cases}
			\frac{\partial}{\partial t}v=Lv\\
			v(0)=f(u_0)
		\end{cases}
	\end{equation}
	This point of view was adopted in \cite{talay1990expansion} for studying the global error of numerical integrators for stochastic differential equations.  Equation \eqref{Eq: expansion Ef} can indeed be seen as an iterative integration of \eqref{Eq: Talay} obtained by a procedure analogous to the one outlined in \eqref{Eq: formal expansion E}.
\end{remark}

\subsection{Exotic coloured trees}\label{Sec: Graphical representation}

The directed trees introduced for the first time in \cite{burrage1996high, komori1997rooted} to study the Taylor expansion of numerical integrators for SDEs
carry a decoration of the vertices that specifies whether they represent a \emph{deterministic} contribution coming from the drift or a \emph{stochastic} contribution from the diffusion. This distinction is of paramount importance to store within these combinatorial object the information on the different kinds of measures that characterize the multiple stochastic integrals. As a preliminary step we define the space of trees needed to represent the It\^o-Taylor expansion of the solution of \eqref{Eq: SDE}. Notice that, when choosing $f(u)=u$ in \eqref{Eq: Ito lemma}, a generic contribution to the stochastic expansion of the solution has the form \eqref{Eq: contributions ito-taylor} for $A_i\in\{L',J\}$ with $L':=\alpha\frac{\partial}{\partial x}$. Hence, resorting to the isomorphism between rooted trees and vector fields spanned by the coefficients of the equation, at the heart of B-series, it is rather natural to consider decorated versions of rooted trees as advocated at the beginning of this section. Intuitively, such a decoration should codify the nature of each $A_i$.
\begin{definition}\label{Def: graphical notation perturbative}
	We denote by $\cT_{\alpha,\beta}$ the space of non-planar, directed \emph{coloured trees} $(\tau,\mathfrak{t})$, where
	\begin{itemize}
		\item $\tau=(\mathcal{V}(\tau),\mathcal{E}(\tau))$ is completely characterized by the set of vertices $\mathcal{V}(\tau)$ and the set of edges $\mathcal{E}(\tau)\subset\mathcal{V}(\tau)\times \mathcal{V}(\tau)$. There is a distinguished vertex $\rho_\tau$, the root, having no outgoing edges. 
		\item  $\mathfrak{t}:\mathcal{V}\to\{\alpha,\beta\}$ associates a colour to each vertex. 
	\end{itemize}
	We introduce the space of coloured forests $\mathcal{F}_{\alpha,\beta}$ as the symmetric tensor algebra with respect to the concatenation product\footnote{The concatenation product of trees coincides with their juxtaposition and will be denoted by $\cdot:\mathcal{F}_{\alpha,\beta}\times \mathcal{F}_{\alpha,\beta}\to\mathcal{F}_{\alpha,\beta}$. We stress that $(\mathcal{F}_{\alpha,\beta},\cdot)$ forms a unital, commutative algebra, the identity being the empty forest $\emptyset$}. The colour decoration on $\mathcal{F}_{\alpha,\beta}$ is inherited from the components, \emph{i.e.}, it coincides with the decoration of the disjoint rooted trees forming a forest. In addition, $\mathcal{F}_{\alpha,\beta}^n\subset \mathcal{F}_{\alpha,\beta}$ is the subspace of coloured forests consisting of the concatenation product of $n$ rooted trees.
\end{definition}
For any $\tau\in\mathcal{T}_{\alpha,\beta}$ we introduce the map $\mathfrak{n}:\mathcal{V}(\tau)\to \N$ that associates to each vertex of $\tau$ the number of incoming edges.
For later convenience we denote by $\mathcal{V}_\alpha(\tau)$, $\mathcal{V}_\beta(\tau)$ the sets of vertices of type $\alpha$ and $\beta$, respectively. In addition $|\tau|\in\N$ denotes the number of vertices of $\tau\in\cT_{\alpha,\beta}$. All these maps extend to  $\cF_{\alpha,\beta}$ in an additive way with respect to the forest product. For the sake of a lighter notation, we will leave the colour decoration implicit by simply writing $\tau$ instead of $(\tau,\mathfrak{t})$ if there is no risk of confusion.
\begin{example}\label{Ex: coloured tree}
	Here is an example of coloured tree lying in $\cT_{\alpha,\beta}$ as per Definition \ref{Def: graphical notation perturbative}:
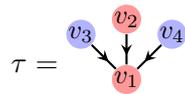
	\begin{equation*}
	\tau=	\begin{tikzpicture}[scale=0.2,baseline=0.1cm]
			\node at (0,0) [dotred] (center) {\small$v_1$};
			\node at (0,4)  [dotred] (centerc) {\mbox{\small $v_2$}};
			\node at (-3,3)  [dotblue] (centerl) {\mbox{\small $v_3$}};
			\node at (3,3)  [dotblue] (centerr) {\mbox{\small $v_4$}};
			\draw[kernel1] (centerc) to
			node [sloped,below] {\small }     (center);
			\draw[kernel1] (centerr) to
			node [sloped,below] {\small }     (center);
			\draw[kernel1] (centerl) to
			node [sloped,below] {\small }     (center);
		\end{tikzpicture}
	\end{equation*}
Henceforth any vertex $v\in\mathcal{V}(\tau)$ such that $\mathfrak{t}(v)=\alpha$ will be represented as a red dot, while if $\mathfrak{t}(v)=\beta$ we will depict it as a blue dot. The image of $\mathfrak{n}$ is $\mathfrak{n}(v_1)=3$, while $\mathfrak{n}(v_i)=0$ for $i\in\{2,3,4\}$.
\end{example}

\begin{remark}
	To make a comparison with the definitions of \cite{laurent2020exotic, bronasco2025exotic} we can interpret $\mathfrak{t}:\mathcal{V}\to \{\alpha,\beta\}$ as a \emph{grafting decoration}. Indeed we can think of vertices of type $\beta$ as being decorated with a symbol $\times$ to denote the presence of the noise, with the key difference that this time the grafted decorations are not restricted to leaves, as it is the case for processes abiding by Langevin dynamics considered in these works.
\end{remark}
At the heart of the theory of B-series there is the correspondence between rooted trees and vector fields involving the coefficients of the equation \cite{butcher1963coefficients}. This relation will clarify the use of coloured trees to label the coefficients of the It\^o-Taylor expansion of the solution as per Section \ref{Sec: ito-taylor}. The canonical iterative definition, crucial in the vector-valued setting, greatly simplifies in the scalar case. Hence we propose the following.
\begin{definition}\label{Def: elementary butcher}
	For any smooth functions $\alpha,\beta\in C^\infty(\R)$ and $u_0\in\R$, we introduce the map $\Upsilon_{u_0}^{\alpha,\beta}:\cT_{\alpha,\beta}\to \R$ defined as
	\begin{equation}\label{Eq: elementary differential non-exotic}
		\Upsilon_{u_0}^{\alpha,\beta}(\tau)=\prod_{v\in\mathcal{V}(\tau)}\mathfrak{t}(v)^{(\mathfrak{n}(v))}(u_0)\,,\qquad \tau\in\cT_{\alpha,\beta}\,.
	\end{equation} 
	It extends  to $\mathcal{F}_{\alpha,\beta}$ as an homomorphism with respect to the forest product. 
\end{definition}
The second key ingredient is a map that extracts the expression of the multiple It\^o integral from the combinatorial trees labelling the expansion. After \cite{gubinelli2010ramification} it became clear how rooted trees are well suited to inductively define iterated integral that do not abide by integration by parts (geometric rough paths). Such an elegant formulation of algebraic integration, crucial in the vector-valued setting, partially hides the link with the global structure of the tree, which will play a prominent role when associating analytic expressions to the trees labelling the series expansion of expectations. Namely, we anticipate that an additional, highly non-local decoration of the $\beta$-vertices prevents us from adopting the formalism typical of branched rough  paths.
Hence we follow a different strategy based on the use of the Heaviside theta function. According to the It\^o prescription, we recall that for us $\Theta(0)=0$.
\begin{definition}\label{Def: realization map butcher}
	The \emph{realization map} $\Pi:\cT_{\alpha,\beta}\times [0,T]^2_{\leq}\to\R$ is such that $\Pi_{st}(\bulletalpha)=\Pi_{st}(\bulletbeta)=1$, while it acts on $\tau\in\mathcal{T}_{\alpha,\beta}$ such that $|\tau|>1$ as
	\begin{align}
		\Pi_{st}(\tau):=\int_{[s,T]^{|\tau|-1}}&\prod_{\substack{e\in\mathcal{E}(\tau) \\ e_+=\rho(\tau)}}\Theta(t-t_{e_-})\left(\un{(\mathfrak{t}(e_-)=\beta)}\,dB(t_{e_-})+\un{(\mathfrak{t}(e_-)=\alpha)}dt_{e_-}\right)\nonumber\\
		&\prod_{\substack{e\in\mathcal{E}(\tau) \\ e_+\neq\rho(\tau)}}\Theta(t_{e_+}-t_{e_-})\left(\un{(\mathfrak{t}(e_-)=\beta)}\,dB(t_{e_-})+\un{(\mathfrak{t}(e_-)=\alpha)}dt_{e_-}\right)\,,\label{Eq: realization map 1}
	\end{align}
	where $e=(e_+,e_-)$ and the orientation is pointing towards the root. It extends to forests of coloured trees as a map $\Pi:\bigoplus_{n=1}^\infty (\mathcal{F}_{\alpha,\beta}^n\times[s,T]^n)\to \R$ whose action on $w=\tau_1\cdot\ldots\cdot\tau_n\in\mathcal{F}_{\alpha,\beta}^n $ reads
	\begin{equation*}
		\Pi_{s,t_1\ldots t_n}(w):=\Pi_{st_1}(\tau_1)\ldots\Pi_{st_n}(\tau_n)\,.
	\end{equation*}
\end{definition}
We can test the definition of realization map and elementary differential on the tree of Example \ref{Ex: coloured tree}. We have
	\begin{align*}
		&\Upsilon_{u_0}^{\alpha,\beta}(\tau)=(\alpha^{(3)}\alpha\beta^2)(u_0)\,,\\
		&\Pi_{st}(\tau)=\int_{[s,T]^4}\Theta(t-t_1)\Theta(t_1-t_2)\Theta(t_1-t_3)\Theta(t_1-t_4)\,dt_1dt_2dB(t_3)dB(t_4)\,.
	\end{align*}
where the stochastic integrals are in the sense of It\^o calculus. In the following we will mainly work with expansions centered at $s=0$. For this reason we will often adopt the shorthand notation $\Pi_t:=\Pi_{0t}$.

Another key ingredient in the B-series expansion of the solution to \eqref{Eq: SDE} is a number that accounts for the symmetries of the tree. Since the decoration prevents us from adopting an explicit, iterative definition as in \emph{e.g.} \cite{chartier2010algebraic, Hairer2006}, we adapt the notion of automorphism of a diagram given in \cite{bruned2025renormalising}.
\begin{definition}\label{Def: automorphism tree}
	An automorphism of the coloured rooted tree $(\tau,\mathfrak{t})\in\cT_{\alpha,\beta}$ is a map that permutes its edges and vertices to produce a tree sharing the same structure as $\tau$ and that preserves the endpoints of each edge and its orientation, as well as the decorations. They are the composition of two group actions  $g_\mathcal{V}:\mathcal{V}(\tau)\rightarrow \mathcal{V}(\tau)$ and $g_\mathcal{E}:\mathcal{E}(\tau)\rightarrow \mathcal{E}(\tau)$ such that
	\begin{equation}
		\tau':=g_\mathcal{V}\circ g_\mathcal{E}\,(\tau)\cong \tau\,,
	\end{equation}
	and such that, for any $e\in\mathcal{E}(\tau), v\in\mathcal{V}(\tau)$,
	\begin{align*}
		&\{g_{\mathcal{V}}(v_{e_+}),g_{\mathcal{V}}(v_{e_-})\}=\{v_{g_{\mathcal{E}}(e)+},v_{g_{\mathcal{E}}(e)-}\}\,,\\
		&\mathfrak{t}(g_{\mathcal{V}(v)})=\mathfrak{t}(v)\,.
	\end{align*} 
	We denote by $\text{Aut}(\tau,\mathfrak{t})$ the space of automorphisms of $(\tau,\mathfrak{t})\in\mathcal{T}_{\alpha,\beta}$ and by $	\sigma(\tau)$ its symmetry factor, defined as the size of its automorphism group 
	\begin{equation*}
		\sigma(\tau):=|Aut(\tau,\mathfrak{t})|\,.
	\end{equation*}
		For $w=(\tau_1)^{\cdot n_1}\cdot\ldots\cdot(\tau_k)^{\cdot n_k}\in\cF_{\alpha,\beta}$, $n_1,\ldots, n_k\in\N$, we define
	\begin{equation*}
		\sigma(w):=\prod_{i=1}^k\sigma(\tau_i)^{n_i}\,.
	\end{equation*}
\end{definition}

B-series proved extremely useful in representing the solution to a rough equation driven by non-differentiable paths
as a series over branched rough paths \cite{gubinelli2010ramification}. Namely, regarding any realization of the Wiener process $\{W_t\}_{t\in[0,T]}$ as a path of H\"older regularity $\frac{1}{2}-\varepsilon$ for all $\varepsilon>0$ and identifying the integrals characterizing $\Pi_t$ in \eqref{Eq: realization map 1} with the components of the branched rough path centered at $0$, the solution takes the form
\begin{equation}\label{Eq: rough path expansion}
	u_t=\sum_{\tau\in\cT_{\alpha,\beta}}\frac{\Upsilon_{u_0}^{\alpha,\beta}(\tau)}{\sigma(\tau)}\Pi_t(\tau)\,.
\end{equation}
Here, instead of a tree-labelled expansion of the solution, we are interested in an expression analogous to \eqref{Eq: rough path expansion} for the expectation of its composition with smooth functions. Naively, one could think of lifting the tree-based expansion of $u_t$ to an analogous expansion for $f(u_t)$ and, subsequently, to study the effect of taking the expectation of the iterated integrals characterizing the realization map. This procedure is correct in the discrete setting of numerical schemes, see \cite{laurent2020exotic}, where $\beta$-vertices are placeholders for Gaussian white noises that, after taking expectations, must be paired in all possible ways through \emph{lianas}. Yet, it a priori fails in the continuum setting where the branched nature of the stochastic integrals prevents us from resorting to a Wick-like theorem. Nevertheless, this erroneous viewpoint offers a heuristic motivation for the classes of trees that we are going to introduce. A central result of Section \ref{Sec: exotic s-series} is the proof that intuition leads to a correct result. 

Here we prefer to adopt the notation of \cite{bronasco2025exotic, rossler2004stochastic}, where lianas are replaced by a suitable decoration of the vertices. Indeed one can identify vertices connected by a liana via a common \emph{exotic decoration}. For a coloured forest $w\in\cF_{\alpha,\beta}$ we consider 
\begin{equation*}
	\mathfrak{d}:\mathcal{V}_\beta(w)\to \N\,,
\end{equation*}
such that $|\mathfrak{d}^{-1}(n)|=2$ for all $n\in\N$ and we denote by $\cF_{\alpha,\beta}^{p}$ the space of \emph{paired forests} $(w,\mathfrak{t},\mathfrak{d})$ endowed with a pairing decoration of this form and such such that $|\mathcal{V}_\beta|=2k$, $k\in\N$.
\begin{definition}\label{Def: Coloured exotic trees}
	We define the equivalence relation $\sim$ on $\cF_{\alpha,\beta}^{p}$ as follows: for $(w,\mathfrak{t}_1,\mathfrak{d}_1),(w, \mathfrak{t}_2,\mathfrak{d}_2)\in\cF_{\alpha,\beta}^p$, we say that $w_1\sim w_2$ if $\mathfrak{t}_1=\mathfrak{t}_2$ and if $\mathfrak{d}_1(v_1)=\mathfrak{d}_1(v_2)$ implies $\mathfrak{d}_2(v_1)=\mathfrak{d}_2(v_2)$ for any $v_1,v_2\in\mathcal{V}_\beta(w)$.
	Then we define the space of \emph{exotic coloured forests}
	\begin{equation*}
		\cF_{\alpha,\beta}^\e:=\frac{\cF_{\alpha,\beta}^{p}}{\sim}\,.
	\end{equation*} 
\end{definition}
The rather involved definition of coloured exotic forests is simply a rigorous way of formalizing that vertices of type $\beta$ must be paired, the pairing being represented by a common natural number. The quotient rules out possible re-labellings that chance the number used to represent the pairings. As an example, the decorated trees
\begin{equation*}
	(\tau,\mathfrak{t},\mathfrak{d})=
	\begin{tikzpicture}[scale=0.2,baseline=0.1cm]
		\node at (0,0) [dotred] (root) {};
		\node at (0,3) [dotred] (center) {};
		\node at (0,6)  [dotred] (centerc) {};
		\node at (-2,5)  [dotblue] (centerl) {\mbox{\small $1$}};
		\node at (2,5)  [dotblue] (centerr) {\mbox{\small $1$}};
		\draw[kernel1] (center) to
		node [sloped,below] {\small }     (root);
		\draw[kernel1] (centerc) to
		node [sloped,below] {\small }     (center);
		\draw[kernel1] (centerr) to
		node [sloped,below] {\small }     (center);
		\draw[kernel1] (centerl) to
		node [sloped,below] {\small }     (center);
	\end{tikzpicture}
	\qquad (\tau,\mathfrak{t},\mathfrak{d}')=
	\begin{tikzpicture}[scale=0.2,baseline=0.1cm]
		\node at (0,0) [dotred] (root) {};
		\node at (0,3) [dotred] (center) {};
		\node at (0,6)  [dotred] (centerc) {};
		\node at (-2,5)  [dotblue] (centerl) {\mbox{\small $2$}};
		\node at (2,5)  [dotblue] (centerr) {\mbox{\small $2$}};
		\draw[kernel1] (center) to
		node [sloped,below] {\small }     (root);
		\draw[kernel1] (centerc) to
		node [sloped,below] {\small }     (center);
		\draw[kernel1] (centerr) to
		node [sloped,below] {\small }     (center);
		\draw[kernel1] (centerl) to
		node [sloped,below] {\small }     (center);
	\end{tikzpicture}
\end{equation*}
identify the same coloured exotic tree $(\tau,\mathfrak{t},\mathfrak{e})$ and, in the language of \cite{rossler2006rooted}, they are equivalent. 

While the algebra $(\cF_{\alpha,\beta},\cdot)$ is generated by monomials of trees lying in $\cT_{\alpha,\beta}$ with respect to the forest product, this is not true for $(\cF_{\alpha,\beta}^\e,\cdot)$ since exotic forests are not necessarily the forest product of exotic trees. To wit, it is true that, for any $(\tau_1,\mathfrak{e}_1),\ldots, (\tau_N,\mathfrak{e}_N)\in\mathcal{F}^{\mathfrak{e}}_{\alpha,\beta}$, the forest product $\tau_1\cdot\ldots\cdot\tau_N$ identifies an exotic forest with exotic decoration $\mathfrak{e}:=\mathfrak{e}_1\sqcup\ldots\sqcup \mathfrak{e}_N$. Nevertheless, Definition \ref{Def: Coloured exotic trees} allows for exotic decorations that pair $\beta$-vertices of distinct rooted trees. This \emph{nonlocal feature} is rather cumbersome, especially for what concerns a detailed account of the combinatorial features of exotic coloured structures. For this reason we prefer to depart from the standard forest-based treatment of exotic S-series \cite{bronasco2024hopf, rossler2004stochastic} by resorting to a canonical isomorphism between forests and trees realized by the operator $B_+$ that maps a forest into a rooted tree by grafting all components of the former on a common, uncoloured root. 
\begin{definition}\label{Def: rxotic coloured trees}
	We define the space $\cT_{\alpha,\beta}^{\mathfrak{e}}$ of exotic coloured trees of the form $\tau=B_+(w)$ for $w\in\mathcal{F}_{\alpha,\beta}^{\mathfrak{e}}$ and we endow it with the exotic grading $|\cdot|_\e:\cT_{\alpha,\beta}^\e\to\N$ and edge-counting map $\mathfrak{l}:\cT_{\alpha,\beta}^\e\to\N$ as follows:
	\begin{align*}
		|\tau|_\e:=|\mathcal{V}_\alpha(\tau)|+\frac{1}{2}|\mathcal{V}_\beta(\tau)|+1\,,\hspace{1.5cm}\mathfrak{l}(\tau):=|\tau|_\e-1\,.
	\end{align*}
\end{definition} 
Henceforth, when clear from the context, we will leave the exotic decoration implied. 

Definition \ref{Def: elementary butcher} naturally extends to an elementary differential $\Upsilon_{u_0}^{\alpha,\beta,f}:\cT_{\alpha,\beta}^\e\to\R$ labelled by $f\in C^\infty(\R)$. Let $\tau=B_+(w)$ for $w\in\mathcal{F}_{\alpha,\beta}^n$. Then
	\begin{equation}\label{Eq: elementary differential trees}
		\Upsilon_{u_0}^{\alpha,\beta,f}(B_+(w)):=f^{(n)}(u_0)\Upsilon_{u_0}^{\alpha,\beta}(w)\,.
	\end{equation}
\begin{definition}
	An automorphism of an exotic coloured tree $(\tau,\mathfrak{t}, \mathfrak{e})\in\mathcal{T_{\alpha,\beta}^\mathfrak{e}}$ is an automorphism $g=g_\mathcal{V}\circ g_{\mathcal{E}}$ as per Definition \ref{Def: automorphism tree} which, in addition, preserves the exotic decoration. Namely, for any $v\in\mathcal{V}_\beta(\tau)$, 
	\begin{align*}
		\mathfrak{e}(g_{\mathcal{V}}(v))=\mathfrak{e}(v)\,.
	\end{align*} 
	We denote by $\text{Aut}(\tau,\mathfrak{t},\mathfrak{e})$ the automorphism group of $(\tau,\mathfrak{t},\mathfrak{e})\in\mathcal{T}^{\mathfrak{e}}_{\alpha,\beta}$ and we define its symmetry factor as $\sigma_{\mathfrak{e}}(\tau):=|Aut(\tau, \mathfrak{t},\mathfrak{e})|$.
\end{definition}

\begin{remark}\label{Rem: symmetry tree-feynman}
	To develop an intuition on how to compute the symmetry factor of an exotic tree, consider for instance
	\begin{equation*}
		(\tau,\mathfrak{t},\mathfrak{e})=
		\begin{tikzpicture}[scale=0.2,baseline=0.1cm]
			\node at (0,0) [dot] (root) {};
			\node at (0,3) [dotred] (center) {};
			\node at (-1,6)  [dotblue] (centerll) {\mbox{\small $i$}};
			\node at (1,6)  [dotblue] (centerrr) {\mbox{\small $k$}};
			\node at (-3,5)  [dotblue] (centerl) {\mbox{\small $i$}};
			\node at (3,5)  [dotblue] (centerr) {\mbox{\small $k$}};
			\draw[kernel1] (center) to
			node [sloped,below] {\small }     (root);
			\draw[kernel1] (centerr) to
			node [sloped,below] {\small }     (center);
			\draw[kernel1] (centerll) to
			node [sloped,below] {\small }     (center);
			\draw[kernel1] (centerrr) to
			node [sloped,below] {\small }     (center);
			\draw[kernel1] (centerl) to
			node [sloped,below] {\small }     (center);
		\end{tikzpicture}
	\end{equation*}
	where, according to Definition \ref{Def: Coloured exotic trees}, the exotic decoration takes any integer values $i,k\in\N$ such that $i\neq k$. We can distinguish two symmetries codified by transformations lying in $Aut(\tau,\mathfrak{t},\mathfrak{e})$:
	\begin{enumerate}
		\item \label{Item 1} Any permutation of the $\beta$-vertices belonging to a pair, \emph{i.e.} $v_1,v_2\in\mathcal{V}_\beta(\tau)$ such that $\mathfrak{e}(v_1)=\mathfrak{e}(v_2)$, 
		\item The permutation of the two pairs of vertices identified different exotic decorations, treated as single effective vertices. \label{Item 2} 
	\end{enumerate}
	Accordingly, given the definition of symmetry factor as the cardinality of the automorphism group, we conclude that
	\begin{equation*}
		\sigma_{\mathfrak{e}}(\tau)=(2!)^22!=8\,,
	\end{equation*}
	where the factor $(2!)^2$ comes from the symmetry described in Item \ref{Item 1}, while $2!$ accounts for the number of permutation of effective vertices mentioned in Item \ref{Item 2}. To make a comparison with non-exotic trees, the tree $(\tau,\mathfrak{t})$ obtained by forgetting $\e$ has $\sigma(\tau)=4!$. 
\end{remark}
We already stressed that, in the discrete realm of numerical analysis, taking expectations of the It\^o-Taylor expansion amounts to switching from coloured trees to exotic coloured trees by collapsing pairs of noises in all possible ways. We make the assumption that the same happens in the present setting, hence visualizing the effect of the exotic decoration on the realization map as the identification of integration variables attached to pairs $v_1,v_1\in\mathcal{V}_\beta$ such that $\e(\tau_1)=\e(\tau_2)$. At the core of Section \ref{Sec: Realization map and the Connes-Moscovici weight} there is the proof that this assumption fulfils our needs.  
\begin{definition}
	We define the \emph{exotic realization map}  $\Pi^{\mathfrak{e}}:\mathcal{T}_{\alpha,\beta}^{\mathfrak{e}}\times [0,T]_\leq^2\to\R$ by $\Pi^\e_{st}(\bullet)=1$ and
	\begin{align}
		\Pi_{st}^{\mathfrak{e}}&(\tau):=\un{(|\mathcal{V}_\beta(w)|=2k)}\int_{[s,T]^{|\tau|-1}}\prod_{\substack{e\in\mathcal{E}(\tau_i)\\ e_+=\rho_{\tau_i}}}\Theta(t-t_{e_-})\prod_{\substack{e\in\mathcal{E}(w)\\ e_+\neq\rho_\tau}}\Theta(t_{e_+}-t_{e_-})\prod_{\substack{(v,w)\in\mathcal{V}_\beta(\tau)\\\mathfrak{e}(v)=\mathfrak{e}(w)}}\delta(t_{v_1}-t_{v_2})\prod_{u\in\mathcal{V}(\tau)\setminus\{\rho(\tau)\}}dt_u\,,\label{Eq: realization map exotic trees correct}
	\end{align}
	for any $(\tau,\mathfrak{t},e)\in\cT_{\alpha,\beta}^\e$ such that $|\tau|_\e>1$ and $t\in[0,T]$.
\end{definition}
 The explicit expression of the realization map on $\cT_{\alpha,\beta}^e$ prompts the following observations
	\begin{itemize}
		\item the expression in \eqref{Eq: realization map exotic trees correct} may contain ill-defined products of distributions $\Theta(t_i-t_j)\delta(t_i-t_j)$. One should work at the level of mollified kernels and remove the regularization at the end. Yet, as we fixed the prescription $\Theta(0)=0$, in the limit all the products of this form vanish. A detailed discussion on this hurdle  is postponed to Remark \ref{Rem: problem products} 
		\item The oriented nature of the kernels represented by legs makes sure that $\Pi_{ st}(\tau)=0$ as soon as the tree contains a closed path, \emph{i.e.} an oriented, connected component of $\tau$ that starts and end at the same vertex.
	\end{itemize}
	These two situations occur when a pair of $\beta$-vertices with the same exotic decoration is shared by vertices that lie on a common path that connects a leaf with the root. As long as the realization map is concerned, we do not have to bother about this constraint, which is part of the definition of exotic trees in \emph{e.g.} \cite{rossler2004stochastic}.

\subsubsection{Grafting of decorated vertices}\label{Sec: grafting of decorated vertices}

Before addressing the problem of finding an exotic B-series expansion for $\mathbb{E}^{u_0}[f(u_t)]$ we need a specific operation on exotic coloured trees. Recalling the It\^o-Taylor expansion introduced in Section \ref{Sec: ito-taylor}, each contribution comes form the action of the differential operator $L^n$ on $f$, see \eqref{Eq: expansion Ef}. Guided by the $1$-to-$1$ correspondence between vector fields and rooted trees realized by the elementary differential $\Upsilon^{\alpha,\beta,f}_{u_0}$, we shall lift the action of $L$ at the level of exotic trees via a \emph{grafting operation}. 

\begin{definition}\label{Def: grafting}
	Let $\tau\in\cT_{\alpha,\beta}^\mathfrak{e}$ and $v,w\in\mathcal{V}(\tau)$. The grafting of an $\alpha$-vertex on $v$, denoted $\bulletalpha\curvearrowright_v\tau$, amounts to attaching a leaf $B_+(\bulletalpha)$ to the vertex $v$. The simultaneous grafting of two $\beta$-vertices on $v$ and $w$, denoted by $\doublebullet\curvearrowright_{v,w}\tau$, amounts to attaching a leaf $B_+(\bulletbeta)$ to the vertex $v$ and another one to $w$. Then we define the grafting operations
	\begin{equation}\label{Eq: grafting}
		\bulletalpha\curvearrowright\tau:=\sum_{v\in\mathcal{V}(\tau)}\bulletalpha\curvearrowright_v\tau\,,\qquad\doublebullet\curvearrowright\tau:=\sum_{v,w\in\mathcal{V}(\tau)}\doublebullet\curvearrowright_{v,w}\tau\,.
	\end{equation} 
	In both cases, the colour of the vertices on which leaves are grafted is preserved. 
\end{definition}

For later convenience, it is useful to investigate the interaction between the grafting of vertices and the action of the elementary differential. We adapt to the present setting a result analogous to \cite[Lem. 2.7, 2.8]{rossler2004stochastic}, see also \cite[Prop. 3.1]{bruned2025multi} for the same relation in the framework of multi-indices B-series.
\begin{theorem}\label{Thm: effect grafting}
	Let $\alpha,\beta,f \in C^\infty(\R)$, $u_0\in\R$ and consider the elementary differential $\Upsilon_{u_0}^{\alpha,\beta,f}$ as per \eqref{Eq: elementary differential trees}. Denoting by $L$ the operator defined in \eqref{Eq: operators}, it holds
	\begin{equation}\label{Eq: effect grafting}
		L\Upsilon_{u_0}^{\alpha,\beta,f}(\tau)=\Upsilon_{u_0}^{\alpha,\beta,f}(\bulletalpha\curvearrowright\tau)+\frac{1}{2}\Upsilon_{u_0}^{\alpha,\beta,f}(\doublebullet\curvearrowright\tau)\,.
	\end{equation}
\end{theorem}
\begin{proof}
	By definition
	\begin{align*}
		\Upsilon_{u_0}^{\alpha,\beta,f}(\bulletalpha\curvearrowright\tau)=\sum_{v\in\mathcal{V}(\tau)}\Upsilon_{u_0}^{\alpha,\beta,f}(\bulletalpha\curvearrowright_v\tau)=\alpha(u_0)\sum_{v\in\mathcal{V}(\tau)}\mathfrak{t}(v)^{\mathfrak{n}(v)+1}(u_0)\prod_{w\in\mathcal{V}(\tau)\setminus {v}}\mathfrak{t}(w)^{\mathfrak{n}(w)}(u_0)\,,
	\end{align*}
	while, by the Leibniz rule
	\begin{align*}
		\alpha\frac{\partial \Upsilon_{\cdot}^{\alpha,\beta,f}(\tau)}{\partial x}\Big\vert_{u_0}=\alpha\frac{\partial}{\partial x}\prod_{v\in\mathcal{V}(\tau)}\mathfrak{t}(w)^{\mathfrak{n}(v)}\Big\vert_{u_0}=\alpha(u_0)\sum_{v\in\mathcal{V}(\tau)}\mathfrak{t}(v)^{\mathfrak{n}(v)+1}(u_0)\prod_{w\in\mathcal{V}(\tau)\setminus {v}}\mathfrak{t}(w)^{\mathfrak{n}(w)}(u_0)\,.
	\end{align*}
	The two expressions coincide, proving the first part of  \eqref{Eq: effect grafting}. The second part of the identity follows suit.
\end{proof}

\subsection{Exotic B-series representation of the Feller semigroup}\label{Sec: exotic s-series}
As stated in the introduction, given an initial value problem of the form
\begin{equation}\label{Eq: generic ODE}
	dx_t=h(x_t)dt\,,\qquad x_0=x\,.
\end{equation}
for a sufficiently smooth function $h$, a numerical method
\begin{equation*}
	y_{n+1}=F(h,\delta)(x_n)\,
\end{equation*}
for the time-step $\delta\in\R$ admits a B-series representation if its Taylor expansion abides by the form
\begin{equation}\label{Eq: B-series generic}
	B(h,\delta,a)=\sum_{\tau\in\cT}\frac{a(\tau)}{\sigma(\tau)}\Upsilon^h_x(\tau)\,\delta^{|\tau|}\,.
\end{equation}
where $\cT$ is the space of non-planar rooted trees, \emph{i.e.}, a non-decorated version of $\cT_{\alpha,\beta}$, $\Upsilon_x^h$ is the elementary differential \eqref{Eq: elementary differential non-exotic} for $\alpha=h$ and $\beta=0$, while $\sigma(\tau)$ is the size of the automorphism group of $\tau$. Without entering into details, we just underline that the choice of a specific functional $a:\cT\to\R$ completely characterizes the numerical scheme of our interest. We also mention that under suitable growth assumptions on $h$ and on $a$, the latter as a function of the number of vertices $|\tau|$, the power series converges, see \cite{Hairer2006}. 
As discovered by Butcher in \cite{butcher1963coefficients}, the choice of $a$ that corresponds to the Taylor expansion of the exact solution to \eqref{Eq: generic ODE} is
\begin{equation}\label{Eq: exact character}
	a(\tau)=(\tau!)^{-1}\,,
\end{equation}
where the \emph{tree factorial} $\tau!$ is defined recursively: we shall set  $\bullet!=1$ and, for any $\tau=B_+(\tau_1,\ldots, \tau_n)\in\cT$, $\tau_1,\ldots, \tau_n\in\cT$, we have
\begin{equation}\label{Eq: tree factrorial classic}
	\tau!=|\tau|\,\tau_1!\ldots \tau_n!\,.
\end{equation} 
For a concise proof that this choice of $a$ yields the exact solution, we refer to \cite{chartier2010algebraic}. 
The summand of the B-series in \eqref{Eq: B-series generic} for the choice \eqref{Eq: exact character} can be rewritten as
\begin{equation*}
	\frac{|\tau|!}{\sigma(\tau)\tau!}\Upsilon_x^h(\tau)\frac{t^{|\tau|}}{|\tau|!}\,,
\end{equation*}
unveiling the presence of the \emph{Connes-Moscovici weights}
\begin{equation}\label{Eq: CM weight}
	CM(\tau):=\frac{|\tau|!}{\sigma(\tau)\tau!}\,,
\end{equation}
introduced in \cite{kreimer1999chen} as the number of ways of obtaining a given rooted tree from subsequent grafting of vertices, starting from the root. This operation is called \emph{natural growth} in the literature \cite{broadhurst1999renormalization, brouder2000runge}.

This form of the Butcher series for the exact solution should not surprise us since, as observed in Section \ref{Sec: ito-taylor}, the construction of the It\^o-Taylor expansion reduces to the classical Taylor expansion if $J=0$ and $L=h\frac{\partial}{\partial x}$. According to Theorem \ref{Thm: effect grafting}, each contribution at order $n$ derives from the grafting of $n$ vertices starting from the root. Hence, each contribution will be of the form $\Upsilon_x^h(\tau)\frac{t^{|\tau|}}{|\tau|!}$ multiplied by the Connes-Moscovici weight of $\tau\in\cT$. In what follows we will adopt this interpretation as an ansatz for the exotic B-series expansion of $\mathbb{E}^{u_0}[f(u_t)]$ and show that it coincides with the It\^o-Taylor expansion. As a first step we need to find a suitable expression for the Connes-Moscovici weights for exotic coloured trees. 
\begin{example}
	To appreciate the higher complexity in counting the number of ways of obtaining a given tree via a natural growth that involves the grafting of $\bulletalpha$ and the simultaneous grafting of $\frac{1}{2}\doublebullet$, consider the trees
	\begin{equation*}
			\tau=
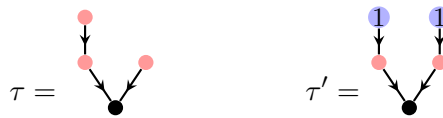
		\begin{tikzpicture}[scale=0.2,baseline=0.1cm]
			\node at (0,0) [dot] (root) {};
			\node at (-2,3) [dotred] (centerl) {};
			\node at (+2,3) [dotred] (centerr) {};
			\node at (-2,6)  [dotred] (centerll) {};
			\draw[kernel1] (centerl) to
			node [sloped,below] {\small }     (root);
			\draw[kernel1] (centerr) to
			node [sloped,below] {\small }     (root);
			\draw[kernel1] (centerll) to
			node [sloped,below] {\small }     (centerl);
		\end{tikzpicture}\hspace{2cm}
	\tau'=
		\begin{tikzpicture}[scale=0.2,baseline=0.1cm]
			\node at (0,0) [dot] (root) {};
			\node at (-2,3) [dotred] (centerl) {};
			\node at (+2,3) [dotred] (centerr) {};
			\node at (-2,6)  [dotblue] (centerll) {\mbox{\small $1$}};
			\node at (2,6)  [dotblue] (centerrr) {\mbox{\small $1$}};
			\draw[kernel1] (centerl) to
			node [sloped,below] {\small }     (root);
			\draw[kernel1] (centerr) to
			node [sloped,below] {\small }     (root);
			\draw[kernel1] (centerrr) to
			node [sloped,below] {\small }     (centerr);
			\draw[kernel1] (centerll) to
			node [sloped,below] {\small }     (centerl);
		\end{tikzpicture}
	\end{equation*}
	We have $3$ different ways of building $\tau$ via natural growth:
	\begin{equation*}
		\bullet\quad\to\quad	
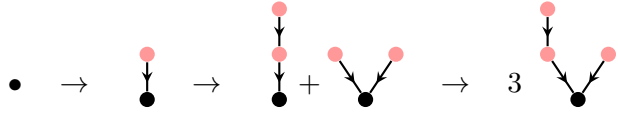
\begin{tikzpicture}[scale=0.2,baseline=0.1cm]
			\node at (0,0) [dot] (root) {};
			\node at (0,3) [dotred] (center) {};
			\draw[kernel1] (center) to node [sloped,below] {\small }     (root);
		\end{tikzpicture}\quad \to\quad	
	\begin{tikzpicture}[scale=0.2,baseline=0.1cm]
		\node at (0,0) [dot] (root) {};
		\node at (0,3) [dotred] (center) {};
		\node at (0,6) [dotred] (center2) {};
		\draw[kernel1] (center) to node [sloped,below] {\small }     (root);
		\draw[kernel1] (center2) to node [sloped,below] {\small }     (center);
	\end{tikzpicture}
	\,+
	\begin{tikzpicture}[scale=0.2,baseline=0.1cm]
		\node at (0,0) [dot] (root) {};
		\node at (-2,3)  [dotred] (centerll) {};
		\node at (2,3)  [dotred] (centerrr) {};
		\draw[kernel1] (centerll) to
		node [sloped,below] {\small }     (root);
		\draw[kernel1] (centerrr) to
		node [sloped,below] {\small }     (root);
	\end{tikzpicture}\quad\to\quad
	3\,\begin{tikzpicture}[scale=0.2,baseline=0.1cm]
	\node at (0,0) [dot] (root) {};
	\node at (-2,3) [dotred] (centerl) {};
	\node at (+2,3) [dotred] (centerr) {};
	\node at (-2,6)  [dotred] (centerll) {};
	\draw[kernel1] (centerl) to
	node [sloped,below] {\small }     (root);
	\draw[kernel1] (centerr) to
	node [sloped,below] {\small }     (root);
	\draw[kernel1] (centerll) to
	node [sloped,below] {\small }     (centerl);
	\end{tikzpicture}
	\end{equation*}
	which coincides with $CM(\tau)$, as it can be readily verified through \eqref{Eq: CM weight}. The same procedure applied to $\tau'$ yields
	\begin{equation*}
		\bullet\quad\to\quad	
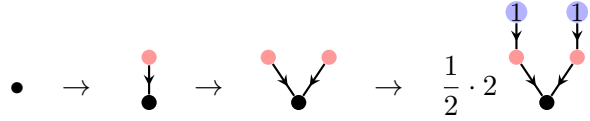
\begin{tikzpicture}[scale=0.2,baseline=0.1cm]
			\node at (0,0) [dot] (root) {};
			\node at (0,3) [dotred] (center) {};
			\draw[kernel1] (center) to node [sloped,below] {\small }     (root);
		\end{tikzpicture}\quad \to\quad	
		\begin{tikzpicture}[scale=0.2,baseline=0.1cm]
			\node at (0,0) [dot] (root) {};
			\node at (-2,3)  [dotred] (centerll) {};
			\node at (2,3)  [dotred] (centerrr) {};
			\draw[kernel1] (centerll) to
			node [sloped,below] {\small }     (root);
			\draw[kernel1] (centerrr) to
			node [sloped,below] {\small }     (root);
		\end{tikzpicture}\quad\to\quad
		\frac{1}{2}\cdot 2\;
		\begin{tikzpicture}[scale=0.2,baseline=0.1cm]
			\node at (0,0) [dot] (root) {};
			\node at (-2,3) [dotred] (centerl) {};
			\node at (+2,3) [dotred] (centerr) {};
			\node at (-2,6)  [dotblue] (centerll) {\mbox{\small $1$}};
			\node at (2,6)  [dotblue] (centerrr) {\mbox{\small $1$}};
			\draw[kernel1] (centerl) to
			node [sloped,below] {\small }     (root);
			\draw[kernel1] (centerr) to
			node [sloped,below] {\small }     (root);
			\draw[kernel1] (centerrr) to
			node [sloped,below] {\small }     (centerr);
			\draw[kernel1] (centerll) to
			node [sloped,below] {\small }     (centerl);
		\end{tikzpicture}
	\end{equation*}
	where the two possible ways of simultaneously grafting $\doublebullet$  on distinct vertices is counterbalanced by the prefactor $\frac{1}{2}$. This second example suggests that an hypothetical extension of the Connes-Moscovici weights to $\cT_{\alpha,\beta}^\e$ should yield $CM^\e(\tau')=2$. Yet, as it well be clarified below, it is more convenient to include the prefactor $\frac{1}{2}$ inside it. Hence we search for a definition that would yield $CM^\e(\tau')=1$.
\end{example}
 Our goal is to keep the form of \eqref{Eq: CM weight} for the extended $CM^{\e}:\cT_{\alpha,\beta}^\e\to\N$. Yet, \eqref{Eq: tree factrorial classic} is ill-defined for trees lying in $\cT_{\alpha,\beta}^\e$ since, as observed above, the exotic decoration prevents us from seeing them as products of trees grafted on a common root via $B_+$. For this reason we resort to an alternative characterization of the tree factorial introduced in  \cite{kreimer1999chen}, that best suits our needs. For a given tree $\tau\in\cT$ we denote by $\cT^-(\tau)$ the space of subtrees of $\tau$ obtained by removing a single leaf and by $m_{\bullet}(\tau',\tau)$ the number of ways of cutting a lead from $\tau$ to obtain $\tau'$. Then we define
\begin{equation}\label{Eq: tree factorial kreimer}
	\frac{|\tau|}{\tau!}=\sum_{\tau'\in\cT^-(\tau)}\frac{m_{\bullet}(\tau',\tau)}{\tau'!}\,,
\end{equation}
where the sum runs over non-isomorphic trees. 
The proof of the equivalence between the iterative characterization, see \eqref{Eq: tree factrorial classic}, and \eqref{Eq: tree factorial kreimer} can be found in \cite[\S 5]{kreimer1999chen}. Contrary to the canonical definition proposed in the literature \cite{chartier2010algebraic,Hairer2006}, this version has the net advantage of being well defined on trees with directed closed paths as soon as we consider $\doublebullet$ as a leaf in the natural growth process.
\begin{definition}\label{Def: tree factorial}
	For $\tau\in\cT_{\alpha,\beta}^\e$, let us denote by $\cT_{\bulletalpha}^-(\tau)$ and $\cT_{\bulletbeta}^-(\tau)$ the collections of trees obtained from $\tau$ by removing a leaf $\bulletalpha$ or $\doublebullet$, respectively in all possible ways. Moreover we denote by $m_{\bulletalpha}(\tau',\tau)$ and $m_{\bulletbeta}(\tau',\tau)$ the number of ways of obtaining $\tau'$ from $\tau$ by removing a leaf of type $\bulletalpha$ or $\doublebullet$, respectively. Then we define the tree factorial $\tau!\in\N$ as
	\begin{equation*}
		\frac{|\tau|}{\tau!}:=\sum_{\tau'\in\cT_{\bulletalpha}^-(\tau)}\frac{m_{\bulletalpha}(\tau',\tau)}{\tau'!}+\sum_{\tau'\in\cT_{\bulletbeta}^-(\tau)}\frac{m_{\bulletbeta}(\tau',\tau)}{\tau'!}\,.
	\end{equation*}
\end{definition}

Instead of assuming for the Connes-Moscovici weight on $\cT^\e_{\alpha,\beta}$ an expression like the one in \eqref{Eq: CM weight}, we shall take our moves from the recursive definition proposed in \cite{kreimer1999chen} for trees $\tau\in\cT$:
\begin{equation}\label{Eq: first kreimer formula}
	\sigma(\tau)\,CM(\tau)=\sum_{\tau'\in\cT^-(\tau)}m_\bullet(\tau',\tau)\,\sigma(\tau')\,CM(\tau')\,.
\end{equation}
This identity, that characterizes the Connes-Moscovici weights of non-decorated trees, is stated in \cite{kreimer1999chen} without a detailed derivation, even if it is shown to imply a form for $CM$ that coincides with the character of the Butcher series representing the exact solution of an ODE. Before addressing its generalization to exotic coloured trees, we give a proof of \eqref{Eq: first kreimer formula} that highlights crucial aspects which carry over, with suitable adaptations, to the decorated setting. We adopt the notation of \cite{hoffman2003combinatorics} and, for a given $\tau'\in\cT^-(\tau)$, we denote by $n_{\bullet}(\tau',\tau)$ the number of way that we can attach a new vertex via an edge to $\tau'$ to obtain $\tau$, while by $m_{\bullet}(\tau',\tau)$ we denote the number of ways of cutting a leaf of $\tau$ to obtain $\tau'$.
\begin{lemma}\label{Lem: kreimer's formula}
	Denoting by $CM(\tau)$ the number of ways of building $\tau\in\cT$ via natural growth, the identity in \eqref{Eq: first kreimer formula} holds true. 
\end{lemma}
\begin{proof}
The definition of $CM(\tau)$ as the number of ways of obtaining $\tau$ via natural growth implies the identification
\begin{align}\label{Eq: intermediate formula kreimer}
	CM(\tau)=\sum_{\tau'\in\cT^-(\tau)}CM(\tau')\,n_\bullet(\tau',\tau)=\sum_{\tau'\in\cT^-(\tau)}m_\bullet(\tau',\tau)\,CM(\tau')\frac{n_\bullet(\tau',\tau)}{m_\bullet(\tau',\tau)}\,,
\end{align}
where, in the second identity, we added for convenience the factor $m_\bullet(\tau',\tau)$. We recall that the sums run over non-isomorphic trees. We conclude by recalling a key result on the combinatorics of rooted trees \cite[Prop. 2.1]{hoffman2003combinatorics}: if $\tau'\in\cT^-(\tau)$, then
\begin{equation*}
	\frac{n_\bullet(\tau',\tau)}{m_\bullet(\tau',\tau)}=\frac{\sigma(\tau')}{\sigma(\tau)}\,.
\end{equation*}
Therefore \eqref{Eq: intermediate formula kreimer} coincides with \eqref{Eq: first kreimer formula}.
\end{proof}
As a second step, we wish to adapt this derivation of Kreimer's formula to the case at hand. Since the action of $L$ amounts to the grafting of a pair $\doublebullet$ weighted with a coefficient $\frac{1}{2}$, instead of including the latter into the elementary differential, we resort to a variant of the Connes-Moscovici weights that coincides with the number of ways of building $\tau\in\cT_{\alpha,\beta}^\e$ via a natural growth that involves $\bulletalpha$ and $\doublebullet$, multiplied by $2^{-|\mathcal{V}_\beta(\tau)|}$. As it turns out, this is the object that has a characterization of the form \eqref{Eq: first kreimer formula}. 

For a fixed tree $\tau\in\cT_{\alpha,\beta}^{\e}$, if $\tau'\in\cT_{\bulletalpha}^-(\tau)$ we denote by $n_{\bulletalpha}(\tau',\tau)$ the analogous of $n(\tau',\tau)$ defined above, which involves the grating of a vertex $\bulletalpha$. Accordingly, if $\tau'\in\cT_{\bulletbeta}^-(\tau)$ we will adopt the notations $n_{\bulletbeta}(\tau',\tau)$ for the number of ways of grafting a pair $\doublebullet$ on $\tau'$ to obtain $\tau$.

\begin{lemma}\label{Lem: CM exotic}
	For $\tau\in\cT_{\alpha,\beta}^\e$, let us denote by $CM^\e(\tau)$ the number of ways of obtaining $\tau$ via natural growth with respect to vertices $\bulletalpha$ and $\doublebullet$, weighted with $2^{-|\mathcal{V}_\beta(\tau)|}$. Then
	\begin{equation}\label{Eq: definition cm}	\sigma_\e(\tau)\,CM^\e(\tau)=\sum_{\tau'\in\cT_{\bulletalpha}^-(\tau)}m_{\bulletalpha}(\tau',\tau)\,\sigma_\e(\tau')\,CM^\e(\tau')+\sum_{\tau'\in\cT_{\bulletbeta}^-(\tau)}m_{\bulletbeta}(\tau',\tau)\,\sigma_\e(\tau')\,CM^\e(\tau')\,.
	\end{equation}
\end{lemma}
\begin{proof}
	Following the strategy of proof of Lemma \ref{Lem: kreimer's formula}, the definition of $CM^\e$ implies 
	\begin{align}
		CM^\e(\tau)=\sum_{\tau'\in\cT_{\bulletalpha}^-(\tau)}m_{\bulletalpha}(\tau',\tau)\,CM^\e(\tau')\frac{n_{\bulletalpha}(\tau',\tau)}{m_{\bulletalpha}(\tau',\tau)}+	\sum_{\tau'\in\cT_{\bulletbeta}^{-}(\tau)}\frac{1}{2}m_{\bulletbeta}(\tau',\tau)\,CM^\e(\tau')\frac{n_{\bulletbeta}(\tau',\tau)}{m_{\bulletbeta}(\tau',\tau)}\,.\label{Eq: intermediate relation CM}
	\end{align}
	The first sum involves the grafting of a single vertex and a reasoning completely analogous to the one for undecorated trees applies. However, the second contribution to $CM^\e(\tau)$ is complicated by the nonlocal nature of the simultaneous grafting. We adapt the proof of \cite[Prop. 2.1]{hoffman2003combinatorics} to the decorated setting. If $\tau'\in\cT^-_{\bulletbeta}$, let $(v_1,v_2)\in\mathcal{V}(\tau')$ be such that, when a pair of vertices $(w_1,w_2)=(\bulletbetadec,\bulletbetadec)$ is grafted on them, the resulting tree coincides with $\tau$. Denote by $|O((v_1,v_2),\tau')|$ the cardinality of the orbit of $(v_1,v_2)$ under the automorphisms lying in $Aut(\tau')$ and by $|S((v_1,v_2),\tau'))|$ the cardinality of the subset of $Aut(\tau')$ that leaves $(v_1,v_2)$ unchanged. We shall divide the rest of the proof in two possible scenarios:
	\begin{enumerate}
		\item Assume that the pair $(w_1,w_2)\in\mathcal{V}_{\beta}(\tau)$ is linked to two distinct vertices $v_1\neq v_2\in\mathcal{V}(\tau')$. Contrary to the undecorated setting, $n_{\bulletbeta}(\tau',\tau)$ does not simply coincides with the size of the orbit of the pair $(v_1,v_2)$, since the latter account for the symmetry of the branch that contains them, which does not corresponds to two different spots for the grafting of $\doublebullet$. In addition, for a fixed choice of distinct $v_1$ and $v_2$, we always have two possible ways of grafting $\doublebullet$ by exchanging them. Hence, denoting by $\sigma_\e(b((v_1,v_2),\tau'))$ the cardinality of the symmetry group of the smaller branch $b_{\tau'}((v_1,v_2)$ containing the leaves $v_1$ and $v_2$, we have
		\begin{equation*}
			n_{\bulletbeta}(\tau',\tau)=2\frac{|O((v_1,v_2),\tau')|}{\sigma_\e(b((v_1,v_2),\tau'))}\,.
		\end{equation*}
		The Orbit-Stabilizer entails that
		\begin{equation*}
			n_{\bulletbeta}(\tau',\tau)=\frac{2}{\sigma_\e(b((v_1,v_2),\tau'))}\frac{\sigma_\e(\tau')}{|S((v_1,v_2),\tau'))|}\,.
		\end{equation*}
		To express $m_{\bulletbeta}(\tau',\tau)$ we apply the same reasoning, yet observing that the we remove the factor $2$ since its definition involves cuts instead of the grafting of $\doublebullet$. Then 
		\begin{equation*}	m_{\bulletbeta}(\tau',\tau)=\frac{|O((w_1,w_2),\tau')|}{\sigma_\e(b((w_1,w_2),\tau))}=\frac{\sigma_\e(\tau)}{\sigma_\e(b((w_1,w_2),\tau))|S((w_1,w_2),\tau))|}\,.
		\end{equation*}
		We conclude that
		\begin{equation*}
			\frac{1}{2}\frac{n_{\bulletbeta}(\tau',\tau)}{m_{\bulletbeta}(\tau',\tau)}=\frac{\sigma_\e(\tau')}{\sigma_\e(\tau)}\frac{\sigma_\e(b((w_1,w_2),\tau))}{\sigma_\e(b((v_1,v_2),\tau'))}=\frac{\sigma_\e(\tau')}{\sigma_\e(\tau)}\,,
		\end{equation*}
		where the last identity follows from the natural identification $\sigma_\e(b((v_1,v_2),\tau'))=\sigma_\e(b((w_1,w_2),\tau))$
		\item Assume that the pair $(w_1,w_2)\in\mathcal{V}_{\beta}(\tau)$ is linked to a single vertex $v\in\mathcal{V}(\tau')$. This time no corrections to $	n_{\bulletbeta}(\tau',\tau)$ are needed, since the pair $(w_1,w_2)$ behaves as a single effective vertex under grafting. 
		\begin{equation*}
			n_{\bulletbeta}(\tau',\tau)=|O(v,\tau')|=\frac{\sigma_\e(\tau')}{|S(v,\tau')|}\,,
		\end{equation*}
		while, when computing $m_{\bulletbeta}((w_1,w_2),\tau)$, we must divide the size of the orbit by $\sigma_\e(b((w_1,w_2),\tau))$ has this hidden symmetry does not identify two different cuts.  To wit, we have
		\begin{equation*}
			m_{\bulletbeta}((w_1,w_2),\tau)=\frac{|O((w_1,w_2),\tau)|}{\sigma_\e(b((w_1,w_2),\tau))}=\frac{\sigma_\e(\tau)}{2|S((w_1,w_2),\tau)|}\,,
		\end{equation*}
		where $\sigma_\e(b((w_1,w_2),\tau))=2$ as $w_1,w_2$ are linked to the same vertex $v$. Hence, making again the identification $|S((w_1,w_2),\tau)|=|S(v,\tau')|$, 
		\begin{equation*}
			\frac{1}{2}\frac{n_{\bulletbeta}(\tau',\tau)}{m_{\bulletbeta}(\tau',\tau)}=\frac{\sigma_\e(\tau')}{\sigma_\e(\tau)}\,.
		\end{equation*}
	\end{enumerate}
	On account of the previous analysis, the identity in  \eqref{Eq: intermediate relation CM} reads
	\begin{equation*}
			CM^\e(\tau)=\sum_{\tau'\in\cT_{\bulletalpha}^-(\tau)}m_{\bulletalpha}(\tau',\tau)\,CM^\e(\tau')\frac{\sigma_\e(\tau')}{\sigma_\e(\tau)}+	\sum_{\tau'\in\cT_{\bulletbeta}^{-}(\tau)}m_{\bulletbeta}(\tau',\tau)\,CM^\e(\tau')\frac{\sigma_\e(\tau')}{\sigma_\e(\tau)}\,.
	\end{equation*}
	and it coincides with \eqref{Eq: definition cm}.
\end{proof}
As shown in \cite{kreimer1999chen}, the \emph{leaf first} expression for the tree factorial \eqref{Eq: tree factorial kreimer} and the definition of $CM^\e$ as per \eqref{Eq: definition cm} lead to an elegant representation of the latter that coincides with the one for non-decorated trees. 
\begin{lemma}\label{Lem: connes-moscovici}
	For any $\tau\in\cT^\e_{\alpha,\beta}$, let $\tau!$ be the tree factorial as per  Definition \eqref{Def: tree factorial}. Then
	\begin{equation}\label{Eq: connes-moscovici weight}
		CM^\e(\tau)=\frac{|\tau|!}{\sigma_\mathfrak{e}(\tau)\tau!}\,.
	\end{equation}
\end{lemma}
\begin{proof}
	The thesis trivially holds true for the tree consisting of the sole root. Then we assume that \eqref{Eq: connes-moscovici weight} is satisfied by any exotic tree with at most $n-1$ vertices. For $\tau\in\cT^\e_{\alpha,\beta}$ such that $|\tau|=n$ we have
	\begin{align*}
		\sigma_{\mathfrak{e}}(\tau)\,CM^\e(\tau)&=\sum_{\tau'\in\cT_{\bulletalpha}^-(\tau)}m_{\bulletalpha}(\tau',\tau)\,\sigma_{\mathfrak{e}}(\tau')\,CM^\e(\tau')+\sum_{\tau'\in\cT_{\bulletbeta}^-(\tau)}m_{\bulletbeta}(\tau',\tau)\,\sigma_{\mathfrak{e}}(\tau')\,CM^\e(\tau')\\
		&=\sum_{\tau'\in\cT_{\bulletalpha}^-(\tau)}m_{\bulletalpha}(\tau',\tau)\,\sigma_{\mathfrak{e}}(\tau')\frac{|\tau'|!}{\sigma_\mathfrak{e}(\tau')\tau'!}+\sum_{\tau'\in\cT_{\bulletbeta}^-(\tau)}m_{\bulletbeta}(\tau',\tau)\,\sigma_{\mathfrak{e}}(\tau')\frac{|\tau'|!}{\sigma_\mathfrak{e}(\tau')\tau'!}\\
		&=(|\tau|-1)!\Bigl(\sum_{\tau'\in\cT_{\bulletalpha}^-(\tau)}\frac{m_{\bulletalpha}(\tau',\tau)}{\tau'!}+\sum_{\tau'\in\cT_{\bulletbeta}^-(\tau)}\frac{m_{\bulletbeta}(\tau',\tau)}{\tau'!}\Bigr)=(|\tau|-1)!\frac{|\tau|}{\tau!}=\frac{|\tau|!}{\tau!}\,,
	\end{align*}
	where in the second identity we resorted to the inductive hypothesis on $\tau'\in\cT_\e^-(\tau)$. This concludes the proof. 
\end{proof}
Eventually, we prove the main result of this section.
\begin{theorem}\label{Thm: main result S-series}
	Let $\alpha,\beta\in C^{2(n+1)}(\R)$ such that \eqref{Eq: SDE} admits a unique solution $u$  with initial condition $u_0\in\R$ on $[0,T]$ and let $f\in C^{2(n+1)}(\R)$. Then, for any $t\in[0,T]$, $\mathbb{E}^{u_0}[f(u_t)]$ coincides with the truncated power series
	\begin{equation}\label{Eq: Ito-Taylor trees}
		\mathbb{E}^{u_0}[f(u_t)]=\sum_{\substack{\tau\in\cT^\e_{\alpha,\beta}\\|\tau|\leq n}}\frac{|\tau|}{\sigma_\mathfrak{e}(\tau)\tau!}\Upsilon_{u_0}^{\alpha,\beta,f}(\tau)t^{|\tau|-1}+R^n(t)\,,
	\end{equation}
	where $\tau!$ is the extended tree factorial as per \eqref{Eq: tree factorial kreimer}, while $\mathfrak{l}(\tau)$ is the number of edges of $\tau$, see Definition \ref{Def: rxotic coloured trees}. The remainder has the form
	\begin{equation*}	R^n(t)=\Bigl(\sum_{\substack{\tau\in\mathcal{T}^{\mathfrak{e}}_{\alpha,\beta}\\|\tau|=n+1}}\frac{n+1}{\sigma_{\mathfrak{e}}(\tau)\,\tau!}\Upsilon_{u_{t_\tau}}^{\alpha,\beta,f}(\tau)\Bigr)\,t^{n}\,,
	\end{equation*}
	for a random time $t_\tau\in[0,T]$.
\end{theorem}
\begin{proof} We follow the lines of \cite[3.2]{rossler2004stochastic}, with the key difference that our definition of $CM^\e$ allows for a concise proof that yields a more detailed expansion. Resorting to the It\^o-Taylor expansion in \eqref{Eq: expansion Ef} and focusing on the truncated part, we must check that
	\begin{equation*}
		\sum_{k=0}^nL^kf(u_0)\frac{t^k}{k!}=\sum_{\substack{\tau\in\mathcal{T}^{\mathfrak{e}}_{\alpha,\beta}\\|\tau|\leq n+1}}\frac{|\tau|}{\sigma_{\mathfrak{e}}(\tau)\,\tau!}\Upsilon_{u_0}^{\alpha,\beta,f}(\tau) t^{|\tau|-1}\,,
	\end{equation*} 
	or, equivalently, that for all $k\leq n$
	\begin{equation}\label{Eq: inductive hypothesis}
		L^kf(u_0)=\sum_{\substack{\tau\in\mathcal{T}^{\mathfrak{e}}_{\alpha,\beta}\\|\tau|=k+1}}\frac{|\tau|!}{\sigma_{\mathfrak{e}}(\tau)\,\tau!}\Upsilon_{u_0}^{\alpha,\beta,f}(\tau)\,.
	\end{equation} 
We proceed by induction on the number of vertices. The first step involves only  the tree consisting of a single root, for which the identity trivially holds true. Now assume that \eqref{Eq: inductive hypothesis} holds true for all $m\leq k$ and consider
\begin{align}\label{Eq: Lk}
L^{k+1}f(u_0)=L\sum_{\substack{\tau\in\mathcal{T}^{\mathfrak{e}}_{\alpha,\beta}\\|\tau|= k+1}}\frac{|\tau|!}{\sigma_{\mathfrak{e}}(\tau)\,\tau!}\Upsilon_{u_0}^{\alpha,\beta,f}(\tau)=\sum_{\substack{\tau\in\mathcal{T}^{\mathfrak{e}}_{\alpha,\beta}\\|\tau|= k+1}}\frac{|\tau|!}{\sigma_{\mathfrak{e}}(\tau)\,\tau!}L\Upsilon_{u_0}^{\alpha,\beta,f}(\tau)\,.
\end{align}	
By Theorem \ref{Thm: effect grafting} we can rewrite the expression in \eqref{Eq: Lk} as
\begin{align}\label{Eq: intermediate equation proof}
	\sum_{\substack{\tau\in\mathcal{T}^{\mathfrak{e}}_{\alpha,\beta}\\|\tau|=k+1}}\frac{|\tau|!}{\sigma_{\mathfrak{e}}(\tau)\,\tau!}\Upsilon_{u_0}^{\alpha,\beta,f}\Bigl(\bulletalpha\curvearrowright \tau+\frac{1}{2}	\doublebullet\curvearrowright \tau\Bigr)\,.
\end{align}
We focus first on the terms involving the grafting of a vertex of colour $\alpha$. 
\begin{align*}
	\sum_{\substack{\tau\in\mathcal{T}^{\mathfrak{e}}_{\alpha,\beta}\\|\tau|=k+1}}&\frac{|\tau|!}{\sigma_{\mathfrak{e}}(\tau)\,\tau!}\Upsilon_{u_0}^{\alpha,\beta,f}(\bulletalpha\curvearrowright \tau)=\sum_{\substack{\tau'\in\mathcal{T}_{\alpha,\beta}^{\mathfrak{e}}\\|\tau'|=k+2}}\Bigl(\sum_{\substack{\tau\in\cT_{\alpha,\beta}^\e\\\bulletalpha\curvearrowright\tau=\tau'}}CM^\e(\tau)\,n_{\bulletalpha}(\tau,\tau')\Bigr)\Upsilon_{u_0}^{\alpha,\beta,f}(\tau')\\
	&=\sum_{\substack{\tau'\in\mathcal{T}_{\alpha,\beta}^{\mathfrak{e}}\\|\tau'|=k+2}}\Bigl(\sum_{\substack{\tau\in\cT_{\bulletalpha}^-(\tau')}}m_{\bulletalpha}(\tau,\tau')\,CM^\e(\tau)\,\frac{n_{\bulletalpha}(\tau,\tau')}{m_{\bulletalpha}(\tau,\tau')}\Bigr)\Upsilon_{u_0}^{\alpha,\beta,f}(\tau')\,.
\end{align*}
Analogously, we can rewrite the second term as
\begin{align*}
	&\sum_{\substack{\tau\in\mathcal{T}^{\mathfrak{e}}_{\alpha,\beta}\\|\tau|=k+1}}\frac{|\tau|!}{2\sigma_{\mathfrak{e}}(\tau)\,\tau!}\Upsilon_{u_0}^{\alpha,\beta,f}(\doublebullet\curvearrowright \tau)=	\sum_{\substack{\tau'\in\mathcal{T}_{\alpha,\beta}^{\mathfrak{e}}\\|\tau'|=k+2}}\Bigl(\sum_{\substack{\tau\in\mathcal{T}^{\mathfrak{e}}_{\alpha,\beta}\\\doublebullet\curvearrowright \tau=\tau'}}\frac{CM^\e(\tau)}{2}\,n_{\bulletbeta}(\tau,\tau')\Bigr)\Upsilon_{u_0}^{\alpha,\beta,f}(\tau')\\
	&=\sum_{\substack{\tau'\in\mathcal{T}_{\alpha,\beta}^{\mathfrak{e}}\\|\tau'|=k+2}}
	\sum_{\tau\in\cT_{\bulletbeta}^{-}(\tau)}m_{\bulletbeta}(\tau,\tau')\,CM^\e(\tau)\frac{n_{\bulletbeta}(\tau,\tau')}{2m_{\bulletbeta}(\tau,\tau')}\Upsilon_{u_0}^{\alpha,\beta,f}(\tau')\,,
\end{align*}
where we resorted to the notations introduced above. A comparison with Lemma \ref{Lem: CM exotic}, see \eqref{Eq: intermediate relation CM}, yields
\begin{align*}
	&\sum_{\substack{\tau\in\mathcal{T}^{\mathfrak{e}}_{\alpha,\beta}\\|\tau|=k+1}}\frac{|\tau|!}{\sigma_{\mathfrak{e}}(\tau)\,\tau!}\Upsilon_{u_0}^{\alpha,\beta,f}\Bigl(\bulletalpha\curvearrowright \tau+\frac{1}{2}	\doublebullet\curvearrowright \tau\Bigr)\\
	&=\sum_{\substack{\tau'\in\mathcal{T}_{\alpha,\beta}^{\mathfrak{e}}\\|\tau'|=k+2}}\Bigl(\sum_{\substack{\tau\in\cT_{\bulletalpha}^-(\tau')}}m_{\bulletalpha}(\tau,\tau')\,CM^\e(\tau)\,\frac{n_{\bulletalpha}(\tau,\tau')}{m_{\bulletalpha}(\tau,\tau')}+\sum_{\tau\in\cT_{\bulletbeta}^{-}(\tau)}m_{\bulletbeta}(\tau,\tau')\,CM^\e(\tau)\frac{n_{\bulletbeta}(\tau,\tau')}{2m_{\bulletbeta}(\tau,\tau')}\Bigr)\Upsilon_{u_0}^{\alpha,\beta,f}(\tau')\\
	&=\sum_{\substack{\tau'\in\mathcal{T}_{\alpha,\beta}^{\mathfrak{e}}\\|\tau'|=k+2}}CM^\e(\tau')\Upsilon_{u_0}^{\alpha,\beta,f}(\tau')\,.
\end{align*}
Eventually, Lemma \ref{Lem: connes-moscovici} allows to conclude that \eqref{Eq: inductive hypothesis} holds for $k=n+1$ to close the induction argument. From the derivation of the It\^o-Taylor expansion of Section \ref{Sec: ito-taylor} it follows that the reminder of the truncated series reads
\begin{equation*}
	R^n(t)=\int_{[0,t]_{\leq}^{n}}\mathbb{E}^{u_0}[L^{n}f(u_{s_{n}})]\,ds_{n}\ldots ds_1\,.
\end{equation*}
To determine its behaviour in time we observe that the solution to \eqref{Eq: SDE} admits an almost surely continuous version, which we consider. In addition, smoothness of $\alpha,\beta,f$ up to $2n$ ensures that $L^{n}f$ is a continuous function. Then, for each $\tau\in\cT^\e_{\alpha,\beta}$ such that $|\tau|=n+1$, there exists a random time $t_\tau\in[0,T]$ such that 
\begin{equation*}
	\mathbb{E}^{u_0}[L^{n}f(u_{s_{n+1}})]=\sum_{\substack{\tau\in\cT^\e_{\alpha,\beta}\\|\tau|=n+1}}\frac{|\tau|!}{\sigma_\e(\tau)\tau!}\Upsilon_{u_{t_\tau}}^{\alpha,\beta,f}\,.
\end{equation*}
Hence
\begin{equation*}
	R^n(t)=\Bigl(\sum_{\substack{\tau\in\mathcal{T}^{\mathfrak{e}}_{\alpha,\beta}\\|\tau|=n+1}}\frac{n+1}{\sigma_{\mathfrak{e}}(\tau)\,\tau!}\Upsilon_{u_{t_\tau}}^{\alpha,\beta,f}(\tau)\Bigr)\,t^{n}\,.
\end{equation*}
This concludes the proof.
\end{proof}

\subsection{Comments on convergence}\label{Sec: asymptotic}
In this section we outline the problems that hinder a possible proof of convergence of the series under reasonable assumptions on the coefficients.
First let us comment on the requirement that $\alpha,\beta,f \in C^\infty(\R)$ are real analytic. If this is the case and the solution to \eqref{Eq: SDE} with initial condition $u_0\in\R$ exists on $[0,T]$, then Theorem \ref{Thm: main result S-series} ensures that the semigroup of the solution admits the exotic B-series representation
\begin{equation}\label{Eq: convergent exact expansion expectation}
	\mathbb{E}[f(u_t)]=\sum_{\tau\in\cT_{\alpha,\beta}^\mathfrak{e}}\frac{|\tau|}{\sigma_\mathfrak{e}(\tau)\tau!}\Upsilon_{u_0}^{\alpha,\beta,f}(\tau)\,t^{|\tau|-1}\,,\hspace{1.5cm}t\in[0,T]\,.
\end{equation}
The requirement of analyticity of the coefficients, as well as of the test function $f$, is crucial: in the absence of analyticity, we could perturb $\alpha,\beta$ and $f$ with smooth functions whose derivative of all orders vanish at $u_0$. This would modify the right-hand side of \eqref{Eq: convergent exact expansion expectation} while leaving the right-hand side unchanged, hence leading to an absurd. 

For what concerns the convergence problem, we comment on two different approaches. In \cite{rossler2004stochastic} the author estimates the error in \eqref{Eq: Ito-Taylor trees} under the assumption of Lipschitz continuous coefficients with uniform polynomial growth of their derivatives to show that is vanishes in the limit as $n$ goes to infinity. Such assumptions allow to control the elementary differentials by resorting to a classical bound on moments of the solution. Yet this approach falls short of accounting for the growing degeneracy ascribed to the Connes-Moscovici weights, which contrasts the vanishing of the analytic contributions. 
More specifically, in order for this estimate to be conclusive we would need a suitable control on the Connes-Moscovici weights of trees of order $n$, as well as on the number of exotic trees at a fixed order, which are out of reach for the moment. This is a common problem when studying convergence of series over trees, which was overcome in the case of B-series for the solutions to ODEs and controlled equations, see \cite[Thm. 5.1]{gubinelli2010ramification}. The strategy is to bound the series of interest with the B-series of the exact solution to a suitable ODE, which is absolutely convergent for small enough time. Assume that $\alpha,\beta$ and $f$ are analytic functions with radii $r_\alpha,r_\beta$ and $r_f$ around $u_0\in\R$, respectively. A direct application of Cauchy formula for complex analytic functions entails that there exist constants $M_\alpha,M_\beta,M_f>0$ such that, for all $n\in\N$,
\begin{align*}
	|f^{(n)}(u_0)|\leq \frac{M_f n!}{r_f^{n}}\,, \hspace{2cm} |\alpha^{(n)}(u_0)|\leq \frac{M_\alpha n!}{r_\alpha^{n}}\,,\hspace{2cm}|\beta^{(n)}(u_0)|\leq \frac{M_\beta n!}{r_\beta^{n}}\,.
\end{align*}
In addition let us define, $R:=\min\{r_f,r_\alpha,r_\beta\}$ and $M=\max\{M_f,M_\alpha,M_\beta\}$. Then it follows that, for $h\in\{f,\alpha,\beta\}$,
\begin{equation}\label{Eq: derivatives h}
	|h^{(n)}(u_0)|\leq \frac{Mn!}{R^n}\,.
\end{equation}
In particular, these bounds imply that the elementary differential $\Upsilon^{\alpha,\beta,f}_{u_0}(\tau)$ is controlled by the elementary differential of the differential equation
\begin{equation}\label{Eq: equation r}
	\dot{r}=g(r)\,,\hspace{1cm} r_0=0\,,\hspace{2cm}g(r):=MR(R-r)^{-1}
\end{equation}
since the right-hand side coincides with $g^{(n)}(0)$. To wit, denoting by $\bar{\tau}\in\cT$ the undecorated rooted tree obtained from $\tau$ by forgetting the colour of each vertex and the exotic decoration, we have
\begin{align*}
	|\Upsilon^{\alpha,\beta,f}_{u_0}(\tau)|\leq \Upsilon^{g}_0(\bar{\tau})\,.	
\end{align*}
As a result we can bound the contributions at order $n$ as
\begin{equation*}
	\sum_{\substack{\tau\in\cT_{\alpha,\beta}^\mathfrak{e}\\|\tau|=n}}\frac{|\tau|}{\sigma_\mathfrak{e}(\tau)\tau!}|\Upsilon_{u_0}^{\alpha,\beta,f}(\tau)|\,t^{|\tau|-1}\leq\sum_{\substack{\tau\in\cT_{\alpha,\beta}^\mathfrak{e}\\|\tau|=n}}\frac{|\tau|}{\sigma_\mathfrak{e}(\tau)\tau!}\Upsilon_{0}^{g}(\bar{\tau})\,t^{|\tau|-1}\,.
\end{equation*}
If we where able to bound the right-hand side with the Butcher series of \eqref{Eq: equation r}, then a reasoning analogous to the one adopted in \cite{gubinelli2010ramification} would apply, entailing convergence of the series for small enough times. Unfortunately this step seems out of reach as the combinatorial contributions may change wildly when switching from an exotic tree to its undecorated counterpart. Another possibility would be to resort to the reduction operator of Appendix \ref{Sec: reduction}, but again this operation does not interact well either with the elementary differential or with the combinatorial contributions. In Appendix \ref{Sec: applications} we consider two examples in which the linear structure of the coefficients ensures convergence uniformly in time. 

\subsection{Realization map and the Connes-Moscovici weight}\label{Sec: Realization map and the Connes-Moscovici weight}

The B-series expansion \eqref{Eq: convergent exact expansion expectation} cannot be readily compared with an analogous expansion based on the path integral representation of the Feller semigroup of the It\^o process. The reason resides in the different role that combinatorial trees play in the It\^o-Taylor expansion and in the perturbative expansion of the path integral measure. As the proof of Theorem \ref{Thm: main result S-series} shows, trees are used to account for the combinatorics of the contributions generated by the iterated action of the operator $L$, while the behaviour in time simply reduces to an iterated integral in time over a simplex. On the other hand, in the MSR expansion the behaviour in time is dictated by branched integrals represented by exotic trees, as it is the case in quantum field theory where expectations are represented as series of integrals described by Feynman diagrams \cite{berglund2022perturbation}. This section is devoted to finding a representation of \eqref{Eq: convergent exact expansion expectation} in terms of the realization map, see Definition \ref{Def: realization map butcher}.

The action of a realization map analogous to the one in \ref{Def: realization map butcher} on non-coloured rooted trees is well understood and boils down to a specific (non-rough) instance of geometric rough paths \cite{gubinelli2010ramification}.
Recent results cover the algebraic aspects of multiple stochastic integrals, \emph{i.e.} the iterated integrals introduced in Section \ref{Sec: ito-taylor}, in connection to generalizations of Chen's formula, see for instance \cite{malham2009stochastic, ebrahimi2012algebraic}.
However, the presence of exotic decorations makes our framework intractable with standard techniques, mainly due to the nonlocal features introduced by closed structures. Our goal is to investigate the action of the realization map on exotic trees to make contact with the contributions to the B-series expansion. In Appendix \ref{Sec: reduction} we study the action of $\Pi^\e$ on closed branches. Here instead we adopt a strategy based on a novel representation of the tree factorial. For a fixed $\tau\in\cT^\e_{\alpha,\beta}$, we introduce the family $\mathcal{R}_{\e}^-(\tau)\subset \cT_{\alpha,\beta}$ of trees obtained from $\tau$ by removing one of the \emph{effective edges} that link the branches to the root and by grafting the rest of the branch to the root. An edge is effective if
\begin{itemize}
	\item it connects the root with a vertex of colour $\alpha$,
	\item it is a double edge that links the root with a pair of $\beta$-vertices with the same exotic decoration. 
\end{itemize}
Note that if an edge links a $\beta$-vertex to the root, but not the other paired one, it is not effective and cannot be cut when looking for trees lying in $\mathcal{R}_{\e}^-(\tau)$. 
\begin{example}
	Considering once again the tree
	\begin{equation*}
		\tau=
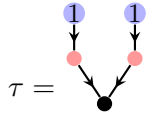
		\begin{tikzpicture}[scale=0.2,baseline=0.1cm]
			\node at (0,0) [dot] (root) {};
			\node at (-2,3) [dotred] (centerl) {};
			\node at (+2,3) [dotred] (centerr) {};
			\node at (-2,6)  [dotblue] (centerll) {\mbox{\small $1$}};
			\node at (2,6)  [dotblue] (centerrr) {\mbox{\small $1$}};
			\draw[kernel1] (centerl) to
			node [sloped,below] {\small }     (root);
			\draw[kernel1] (centerr) to
			node [sloped,below] {\small }     (root);
			\draw[kernel1] (centerrr) to
			node [sloped,below] {\small }     (centerr);
			\draw[kernel1] (centerll) to
			node [sloped,below] {\small }     (centerl);
		\end{tikzpicture}
	\end{equation*}
	The removal of the effective edges entering the root, together with the vertices at their top, leads to
	\begin{equation*}
		\mathcal{R}_{\e}^-(\tau)=\{\quad 
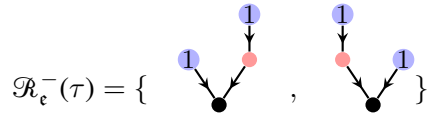
\begin{tikzpicture}[scale=0.2,baseline=0.1cm]
			\node at (0,0) [dot] (root) {};
			\node at (+2,3) [dotred] (centerr) {};
			\node at (-2,3)  [dotblue] (centerll) {\mbox{\small $1$}};
			\node at (2,6)  [dotblue] (centerrr) {\mbox{\small $1$}};
			\draw[kernel1] (centerll) to
			node [sloped,below] {\small }     (root);
			\draw[kernel1] (centerr) to
			node [sloped,below] {\small }     (root);
			\draw[kernel1] (centerrr) to
			node [sloped,below] {\small }     (centerr);
		\end{tikzpicture}
		\quad,\quad 
		\begin{tikzpicture}[scale=0.2,baseline=0.1cm]
			\node at (0,0) [dot] (root) {};
			\node at (+2,3) [dotblue] (centerr) {\mbox{\small $1$}};
			\node at (-2,3)  [dotred] (centerll) {};
			\node at (-2,6)  [dotblue] (centerrr) {\mbox{\small $1$}};
			\draw[kernel1] (centerll) to
			node [sloped,below] {\small }     (root);
			\draw[kernel1] (centerr) to
			node [sloped,below] {\small }     (root);
			\draw[kernel1] (centerrr) to
			node [sloped,below] {\small }     (centerl);
		\end{tikzpicture}
		\}
	\end{equation*}
\end{example}
\begin{theorem}\label{Thm: equivalence factorials}
	Given $\tau\in\cT_{\alpha,\beta}^\e$ such that $|\tau|> 1$ and denoting by $\tau!$ the tree factorial of $\tau\in\cT^\e_{\alpha,\beta}$ as per  \eqref{Eq: tree factorial kreimer}, it holds
	\begin{equation}\label{Eq: alterntive version factorial}
		\frac{|\tau|!}{\tau!}=\sum_{\tau'\in\mathcal{R}_{\e}^-(\tau)}\frac{|\tau'|!}{\tau'!}\,.
	\end{equation}
\end{theorem}
\begin{proof}
	We rely once more on an induction procedure over the number of vertices. The only trees in $\cT_{\alpha,\beta}^\e$ with two vertices are
	\begin{equation}\label{Eq: trees proof}
		\tau_1=\begin{tikzpicture}[scale=0.2,baseline=0.1cm]
			\node at (0,0) [dot] (root) {};
			\node at (0,3)  [dotred] (centerll) {};
			\draw[kernel1] (centerll) to
			node [sloped,below] {\small }     (root);
		\end{tikzpicture}\,,\qquad
		\tau_2=
		\begin{tikzpicture}[scale=0.2,baseline=0.1cm]
			\node at (0,0) [dot] (root) {};
			\node at (-1.5,2.5)  [dotblue] (centerll) {\mbox{\small $1$}};
			\node at (1.5,2.5)  [dotblue] (centerrr) {\mbox{\small $1$}};
			\draw[kernel1] (centerll) to
			node [sloped,below] {\small }     (root);
			\draw[kernel1] (centerrr) to
			node [sloped,below] {\small }     (root);
		\end{tikzpicture}\,.
	\end{equation}
	 The identity \eqref{Eq: alterntive version factorial} trivially holds true since $\cT^-_\e(\tau_i)=\mathcal{R}_{\e}^-(\tau_i)=\{\bullet\}$ for $i=1,2$, as the operation defining $\mathcal{R}_{\e}^-(\tau_2)$ treats the edges stemming from a pair of vertices sharing the same exotic decoration as a single one. By the inductive hypothesis we can write
	\begin{align*}
		\frac{|\tau|!}{\tau!}&=\sum_{\tau'\in\cT^-_{\bulletalpha}(\tau)}m_{\bulletalpha}(\tau',\tau)\frac{|\tau'|!}{\tau'!}+\sum_{\tau'\in\cT^-_{\bulletbeta}(\tau)}m_{\bulletbeta}(\tau',\tau)\frac{|\tau'|!}{\tau'!}\\
		&=\sum_{\tau'\in\cT^-_{\bulletalpha}(\tau)}m_{\bulletalpha}(\tau',\tau)\sum_{\tau''\in\mathcal{R}_{\e}^-(\tau')}\frac{|\tau''|!}{\tau''!}+\sum_{\tau'\in\cT^-_{\bulletbeta}(\tau)}m_{\bulletbeta}(\tau',\tau)\sum_{\tau''\in\mathcal{R}_{\e}^-(\tau')}\frac{|\tau''|!}{\tau''!}\,.
	\end{align*}
	We observe that the two operations that define $\cT^-_\e(\tau)$ and $\mathcal{R}_{\e}^-(\tau)$ commute. Indeed, if $\tau'\in\cT^-_\e(\tau)$ was generated by subtracting a leaf attached to the root, then $\tau'\in\mathcal{R}_{\e}^-(\tau)$ as well. If this was not the case, then the sets of edges on which the two operations act do not intersect. Here asking that pairs of $\beta$-vertices that are not both linked to the root do not form effective edges is crucial. 
	As a result
	\begin{align*}
		\frac{|\tau|!}{\tau!}=\sum_{\tau'\in\mathcal{R}_{\e}^-(\tau)}\sum_{\tau''\in\cT^-_{\bulletalpha}(\tau')}m_{\bulletalpha}(\tau'',\tau')\frac{|\tau''|!}{\tau''!}+\sum_{\tau'\in\mathcal{R}_{\e}^-(\tau)}\sum_{\tau''\in\cT^-_{\bulletbeta}(\tau')}m_{\bulletbeta}(\tau'',\tau')\frac{|\tau''|!}{\tau''!}=\sum_{\tau'\in\mathcal{R}_{\e}^-(\tau)}\frac{|\tau'|!}{\tau'!}\,.
	\end{align*}
	This concludes the proof.
\end{proof}
This alternative characterization of the tree factorial for trees lying in $\cT^\e_{\alpha,\beta}$ is of paramount importance to express the realization map as a power of the time parameter. Indeed, a direct inspection of Definition \ref{Def: realization map butcher} suggests that the operation defining $\mathcal{R}_{\e}^-(\tau)$ realizes the action of time derivatives on $\Pi^\e_t(\cdot)$. This observation is at the heart of the following statement. 
\begin{proposition}\label{Prop: explicit form Pi}
	Let $\Pi^\e:\cT_{\alpha,\beta}^\mathfrak{e}\times[0,T]\to\R$ be the realization map of Definition \ref{Def: realization map butcher}. Then, for any $\tau\in \cT_{\alpha,\beta}^\mathfrak{e}$ and $t\in[0,T]$ we have 
	\begin{equation}\label{Eq: explicit expression pi}
		\Pi_t^\mathfrak{e}(\tau)=\frac{|\tau|}{\tau!}t^{|\tau|-1}\,.
	\end{equation}
\end{proposition}
\begin{proof}
	We shall consider the action of the time derivative on the realization map. By definition, see \eqref{Eq: realization map exotic trees correct}, the only time dependence is carried by the kernels representing edges that connect the root to the rest of the tree. Since the weak derivative of the Heaviside Theta is a Dirac delta, this operation boils down to identifying the vertices at the extrema of the edge hit by the derivative. Observe that, if the time derivative falls on a non-effective edge, this identification forms a closed path. Then, the oriented nature of $\Theta$ forces such contribution to vanish. Hence, at a graphical level, taking the derivative of the realization map coincides with the operation that defines the set $\mathcal{R}_{\e}^-$. Hence
	\begin{equation}\label{Eq: derivative pi}
		\frac{d}{dt}\Pi^\e_t(\tau)=\sum_{\tau'\in \mathcal{R}_{\e}^-(\tau)}\Pi^\e_t(\tau')\,,
	\end{equation}
	At the same time, by deriving the right-hand side of \eqref{Prop: explicit form Pi} we have
	\begin{equation}\label{Eq: derivative non pi}
		\frac{d}{dt}\left(\frac{|\tau|}{\tau!}t^{|\tau|-1}\right)=\frac{|\tau|(|\tau|-1)}{\tau!}t^{|\tau|-2}
	\end{equation}
	We will use this information in the induction on the number of vertices to prove \eqref{Eq: explicit expression pi}. A direct inspection show that the latter trivially holds true for $\tau_1$ and $\tau_2$ as in \eqref{Eq: trees proof}. Now assume that \eqref{Eq: explicit expression pi} is satisfied by any tree with at most $n$ vertices and we consider $\tau\in\cT^\e_{\alpha,\beta}$ such that $|\tau|=n+1$.
	From the equivalence shown in Theorem \ref{Thm: equivalence factorials} and by exploiting the induction hypothesis on $\tau'\in\mathcal{R}_{\e}^-(\tau)$, we have
	\begin{align*}
		\sum_{\tau'\in \mathcal{R}_{\e}^-(\tau)}\Pi^\e_t(\tau')=\sum_{\tau'\in \mathcal{R}_{\e}^-(\tau)}\frac{|\tau'|}{\tau'!}t^{|\tau'|-1}=\Bigl(\sum_{\tau'\in \mathcal{R}_{\e}^-(\tau)}\frac{|\tau'|!}{\tau'!}\Bigr)\frac{t^{|\tau|-2}}{(|\tau|-2)!}=\frac{|\tau|!}{\tau!}\frac{t^{|\tau|-2}}{(|\tau|-2)!}\,.
	\end{align*}
	As a result
	\begin{equation*}
		\frac{d}{dt}\left(\Pi^\e_t(\tau)-\frac{|\tau|!}{\tau!}\frac{t^{\mathfrak{l}(\tau)}}{\mathfrak{l}(\tau)!}\right)=0\,.
	\end{equation*}
	Since the solution to this differential equation must have vanishing initial condition at $t=0$, we conclude that
	\begin{equation*}
		\Pi^\e_t(\tau)=\frac{|\tau|!}{\tau!}t^{|\tau|-1}\,.
	\end{equation*}
\end{proof}

This result is in accordance with the standard representation of iterated integrals via rooted trees, as showed in \cite{gubinelli2010ramification} for an equation driven by generic paths. Indeed, \eqref{Eq: explicit expression pi} should be read as
\begin{equation*}
	\Pi^\e_t(\tau)=\frac{t^{|B_-(\tau)|}}{B_-(\tau)!}\,,
\end{equation*} 
where $B_-(\tau)$ coincides with the exotic forest obtained by cutting the edges entering the root. This is in accordance with the ancillary role played by the root in our setting. 

\begin{remark}
	Proving \eqref{Eq: explicit expression pi} without passing through the equivalent definition of tree factorial of Theorem \ref{Thm: equivalence factorials} is a daunting task. A major difficulty resides in the process of constructing trees lying in $\cT_{\alpha,\beta}^\mathfrak{e}$, which does not boil down to products of trees and convolutions with a kernel, at the heart of the algebraic integration adopted in \emph{e.g.} rough paths and regularity structures. To wit, the nonlocal features of exotic decorations produce trees excluded from the codomain of such simple operations. 
\end{remark}

\begin{corollary}\label{Cor: key corollary}
		Let $\alpha,\beta,f\in C^\infty(\R)$ be real analytic functions such that  \eqref{Eq: SDE} admits a unique solution $u$  with initial condition $u_0\in\R$ on $[0,T]$. Then
	\begin{equation}\label{Eq: exact expansion expectation Pi}
		\mathbb{E}^{u_0}[f(u_t)]=\sum_{\tau\in\cT^\e_{\alpha,\beta}}\frac{\Upsilon_{u_0}^{\alpha,\beta,f}(\tau)}{\sigma_\mathfrak{e}(\tau)}\Pi^\e_t(\tau)\,,
	\end{equation}
	as a formal series.
\end{corollary}
\begin{proof}
	The statement follows from combining the formal exotic B-series expansion in \eqref{Eq: convergent exact expansion expectation} and Proposition \ref{Prop: explicit form Pi}. 
\end{proof}

\section{MSR expectations via Feynman multi-indices}\label{Sec: MSR expectations via Feynman multi-indices}


As underlined in the introduction, the perturbative expansion of MSR expectations heavily relies on the interpretation of the exponential of the free action as a Gaussian measure on a space of paths $\mathcal{C}\oplus\tilde{\mathcal{C}}=H^1([0,T])\oplus H^1([0,T])$ with quadratic form induced by the operator $T:H^1([0,T])\oplus H^1([0,T])\to L^2([0,T])\oplus L^2([0,T])$ defined in \eqref{Eq: free part MSR}. Then $T^{-1}$ is usually interpreted as the covariance matrix of the Gaussian measure $(\psi,\tilde{\psi})$, whose non-diagonal structure suggests the properties \eqref{Eq: properties random fields}. Yet, even if  $T^{-1}$ is symmetric on the subset of $H^1([0,T])\oplus H^1([0,T])$ of paths $\tilde{\psi}$ such that $\tilde{\psi}(0)=\tilde{\psi}(T)=0$, the requirement of being non-negative cannot be satisfied, preventing us from interpreting $e^{-\langle(\psi,\tilde{\psi}), T(\psi,\tilde{\psi})\rangle}\mathcal{D}\psi\mathcal{D}\tilde{\psi}$ as a Gaussian measure. Indeed , if $\int_{0}^T\tilde{\psi}_t\dot{\psi}_t\,dt> 0$ for a pair $(\psi,\tilde{\psi})\in\mathcal{C}\oplus\tilde{\mathcal{C}}$, then $\int_{0}^T\tilde{\psi}'_t\dot{\psi}'_t\,dt< 0$ for $(\psi',\tilde{\psi}')=(\psi,-\tilde{\psi})\in\mathcal{C}\oplus\tilde{\mathcal{C}}$.
This remark invalidates the usual, heuristic derivation of path integral expectations in the MSR formalism. Nevertheless, explicit computations in manifold application of this framework have proven efficient in reproducing the correct statistics of the diffusion process under scrutiny. A posteriori we will show that this improper assumption on the construction of the MSR expectations yields the correct result.

We shall adopt the functional point of view on perturbative algebraic quantum field theory \cite{rejzner2016perturbative, hawkins2020star}, in which additional information from the model are implemented at the algebraic level via a deformation of the product structure of suitable spaces of functionals over smooth configurations. 
As advocated in \cite{bonicelli2023algebraic}, the problem of constructing MSR expectations via maximal contractions of the auxiliary fields can be rigorously tackled by resorting to a formulation based on distribution-valued functionals. Since our ultimate goal is just to rigorously define a map that implements at the analytic level the contraction of vertices to form Feynman diagrams, we present directly the form of the MSR expectations in term of distribution-valued functionals and contraction maps. For the sake of readability we collect all definitions in Appendix \ref{Sec: deformation maps}, as they are not crucial to understand the rest of the paper. 

The MSR action \eqref{Eq: action} can be represented as a distribution-valued functional lying in $\mathcal{F}_{loc}(\R,\mathsf{T}')$:
\begin{equation}\label{Eq: functional interacting vertex}
	S_{int}[\psi,\psi]\quad\mapsto\quad \mathsf{S}(\psi,\tilde{\psi}):= (\tilde{\Psi}\alpha(\Psi)-\theta_0\alpha^{(1)}(\Psi)+\frac{1}{2}\tilde{\Psi}^2\beta^2(\Psi)-\theta_0\tilde{\Psi}\beta(\Psi)\beta^{(1)}(\Psi))(\psi,\tilde{\psi})\,,
\end{equation}
where $\Psi,\tilde{\Psi}$ are the field functionals defined in \eqref{Eq: field functionals}. In particular, defining $\chi(t):=\Theta(t)-\Theta(t-T)$, 
\begin{equation*}
	S_{int}[\psi,\psi]:=\langle \mathsf{S}(\psi,\tilde{\psi}),\chi\rangle\,.
\end{equation*}
Notice that, even if $\chi$ is not a test function lying in $\mathcal{D}(\R)$, the pairing in \eqref{Eq: functional interacting vertex} is well defined. Then the MSR expectation of a function $f\in C^\infty(\R)$ of the diffusion process solving \eqref{Eq: SDE} is defined as the formal perturbative series
\begin{align*}
	\mathbb{E}^{\smallMSR}[f(u_t)]:=\Gamma_\Theta\left[f(\Psi)\cdot e^{\mathsf{S}_{int}(\Psi,\Psi)}\right](\delta_t\otimes \chi)\Big\vert_{\substack{\psi=u_0\\\tilde{\psi}=0}}
\end{align*}
The action of $\Gamma_\Theta$ amounts to contracting in all possible ways an increasing number of pairs $(\Psi,\tilde{\Psi})$ into an instance of the integral kernel of the Heaviside theta. 
For notational ease we shall write distribution-valued functionals in terms of their integral kernels. Accordingly, the latter expression can be rewritten as
\begin{align}
	\mathbb{E}^{\smallMSR}[f(u_t)]&=\Gamma_\Theta\left[f(\psi(t))\cdot e^{\langle \tilde{\psi}\alpha(\psi)-\theta_0\alpha^{(1)}(\psi)+\frac{1}{2}\tilde{\psi}^2\beta^2(\psi)-\theta_0\tilde{\psi}\beta(\psi)\beta^{(1)}(\psi),\chi\rangle}\right]\Big\vert_{\substack{\psi=u_0\\\tilde{\psi}=0}}\nonumber\\
	&=\Gamma_{\Theta}\left[f(\psi(t))\cdot e^{\langle\Gamma_{\Theta}^{-1}(\widetilde{\psi}\alpha(\psi))+\frac{1}{2} \widetilde{\psi}^2\beta^2(\psi)-\theta_0 \widetilde{\psi}\beta\beta^{(1)}(\psi),\chi\rangle}\right]\Big\vert_{\substack{\psi=u_0\\\tilde{\psi}=0}}\,\label{Eq: fist step MSR}
\end{align}
where we used the identity $\Gamma_{\Theta}^{-1}(\tilde{\psi}\alpha(\psi))=\Gamma_{-\Theta}(\tilde{\psi}\alpha(\psi))=\tilde{\psi}\alpha(\psi)-\Theta(0)\alpha^{(1)}(\psi)$ and the identification $\Theta(0)=\theta_0$. Resorting to the nonlocal version of the contraction map
\begin{equation*}
	\Gamma_{\Theta}^{mloc}:\mathcal{F}_{mloc}(\R,\mathsf{T}')\,\hat{\otimes}\,\mathcal{F}_{mloc}(\R,\mathsf{T}') \rightarrow\mathcal{F}_{mloc}(\R,\mathsf{T}')\,,
\end{equation*} 
defined in \eqref{Eq: nonlocal gamma}, Equation \eqref{Eq: fist step MSR} takes the form 
\begin{align*}
	\mathbb{E}^{\smallMSR}[f(u_t)]&=\Gamma^{mloc}_{\Theta}\left[\Gamma_\Theta (f(\psi(t)))\,\hat{\otimes}\, e_{\hat{\otimes}}^{\langle\Gamma_\Theta\Gamma_{\Theta}^{-1}(\widetilde{\psi}\alpha(\psi))+\frac{1}{2} \Gamma_\Theta\widetilde{\psi}^2\beta^2(\psi)-\theta_0 \Gamma_\Theta\widetilde{\psi}\beta\beta^{(1)}(\psi),\chi\rangle}\right]\Big\vert_{\substack{\psi=u_0\\\tilde{\psi}=0}}\\
	&=\Gamma^{mloc}_{\Theta}\left[f(\psi(t))\,\hat{\otimes}\, e_{\hat{\otimes}}^{\langle\widetilde{\psi}\alpha(\psi)+\frac{1}{2} \Gamma_\Theta\widetilde{\psi}^2\beta^2(\psi)-\theta_0 \Gamma_\Theta\widetilde{\psi}\beta\beta^{(1)}(\psi),\chi\rangle}\right]\Big\vert_{\substack{\psi=u_0\\\tilde{\psi}=0}}\,,
\end{align*}
where $e_{\hat{\otimes}}$ denotes the exponential of functional-valued distributions with respect to the tensor product $\hat{\otimes}$, see Appendix \ref{Sec: deformation maps}.
On account of the identification of $\Theta(0)$ with $\theta_0$, an explicit calculation shows that
\begin{align*}
	\Gamma_\Theta\left(\frac{1}{2}\tilde{\psi}^2\beta^2(\psi)-\theta_0\tilde{\psi}\beta\beta^{(1)}(\psi)\right)(t)=\frac{1}{2}\tilde{\psi}^2\beta^2(\psi)(t)+\theta_0\tilde{\psi}\beta\beta^{(1)}(\psi)(t)\,.
\end{align*}
Hence
\begin{align}
	\mathbb{E}^{\smallMSR}[f(u_t)]&=\Gamma_{\Theta}^{mloc}\left[f(\psi(t))\,\hat{\otimes}\, e_{\hat{\otimes}}^{\langle \widetilde{\psi}\alpha(\psi)+\frac{1}{2} \widetilde{\psi}^2\beta^2(\psi)+\theta_0 \widetilde{\psi}\beta\beta^{(1)}(\psi),\chi\rangle}\right]\Big\vert_{\substack{\psi=u_0\\\tilde{\psi}=0}}\,.
\end{align}
Most notably, this expression involves uniquely the multilocal contraction map. This means that no contractions between pairs of fields $(\psi,\tilde{\psi})$ belonging to the same interacting vertex $\tilde{\psi}\alpha(\psi)$ or $\tilde{\psi}^2\beta^2(\psi)$ will appear in the calculations, preventing loops from appearing in a diagrammatic representation of the perturbative expansion. Yet there is a residue of the arbitrariness in extending $\Theta$ of the form $\theta_0\beta\beta^{(1)}$. Since we are interested in studying the expansion of expectation for a solution of \eqref{Eq: SDE} in the sense of It\^o calculus, we shall set $\theta_0=0$. Hence we are left with the expression
\begin{equation}\label{Eq: final formula MSR}
	\mathbb{E}^{\smallMSR}[f(u_t)]=\Gamma_{\Theta}^{mloc}\left[f(\psi(t))\,\hat{\otimes}\, e_{\hat{\otimes}}^{\langle \widetilde{\psi}\alpha(\psi)+\frac{1}{2} \widetilde{\psi}^2\beta^2(\psi),\chi\rangle}\right]\Big\vert_{\substack{\psi=u_0\\\tilde{\psi}=0}}\,.
\end{equation}

\subsection{MSR Feynman diagrams}\label{Sec: MSR Feynman diagrams}
The idea at the heart of the perturbative expansion of path integral expectations as per \eqref{Eq: final formula MSR} is to Taylor expand the interaction around the initial condition $u_0\in\R$ in powers of $\tilde{\psi}$ and $\psi-u_0$ and to contract nonlocal pairs of fields in all possible ways. A key difference with respect to usual applications in Euclidean field theory is that, in the present framework, the Lagrangian density is not a combination of polynomials of the fields. Indeed, for the sake of generality we consider the case of an SDE with smooth coefficients $\alpha,\beta\in C^\infty(\R)$ so that, loosely speaking, the MSR action \eqref{Eq: action} involves an infinite number of Taylor polynomials coming from the formal expansion of $\alpha$ and $\beta^2$. The contraction operation codified by $\Gamma_\Theta$ has a natural graphical representation: following the one briefly explained in the introduction, we can think of $\psi$ and $\tilde{\psi}$ as free legs of different nature. Hence a contraction amounts to linking the free legs representing $\psi$ and $\tilde{\psi}$, respectively into an edge. This observation prompts the definition of \emph{MSR Feynman diagrams} as connected, oriented graphs $\Gamma=(\mathcal{E}(\Gamma),\mathcal{V}(\Gamma))$ identified with a set of oriented edges $\mathcal{E}(\Gamma)$ and of vertices $\mathcal{V}(\Gamma)$ that abide by the following rules
\begin{enumerate}
	\item There is a single vertex with no outgoing edges, that we call the root. 
	\item Each vertex different from the root can have either one or two outgoing edges and no restriction on the number of incoming edges.  
	\item \label{Item: closed loops} There are no oriented closed paths, \emph{i.e.}, connected collection of edges that, when followed along their orientation, start and end at the same vertex.
\end{enumerate}
We denote by $\mathcal{F}_{\smallMSR}$ the space of MSR Feynman diagrams. Examples of diagrams abiding by these rules are
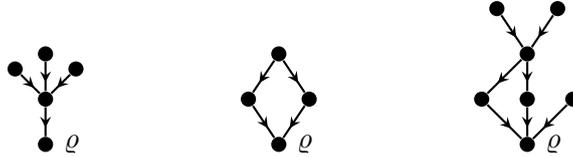
\begin{equation*}
	\begin{tikzpicture}[scale=0.2,baseline=0.1cm]
		\node at (0,0) [dot,label=right:$\rho$] (root) {};
		\node at (0,3) [dot] (center) {};
		\node at (0,6)  [dot] (centerc) {};
		\node at (-2,5)  [dot] (centerl) {};
		\node at (2,5)  [dot] (centerr) {};
		\draw[kernel1] (center) to
		node [sloped,below] {\small }     (root);
		\draw[kernel1] (centerc) to
		node [sloped,below] {\small }     (center);
		\draw[kernel1] (centerr) to
		node [sloped,below] {\small }     (center);
		\draw[kernel1] (centerl) to
		node [sloped,below] {\small }     (center);
	\end{tikzpicture}\hspace{2cm}
	\begin{tikzpicture}[scale=0.2,baseline=0.1cm]
		\node at (0,0) [dot, label=right:$\rho$] (root) {};
		\node at (-2,3) [dot] (centerl) {};
		\node at (+2,3) [dot] (centerr) {};
		\node at (0,6)  [dot] (centerll) {};
		\draw[kernel1] (centerl) to
		node [sloped,below] {\small }     (root);
		\draw[kernel1] (centerr) to
		node [sloped,below] {\small }     (root);
		\draw[kernel1] (centerll) to
		node [sloped,below] {\small }     (centerr);
		\draw[kernel1] (centerll) to
		node [sloped,below] {\small }     (centerl);
	\end{tikzpicture}\hspace{2cm}
	\begin{tikzpicture}[scale=0.2,baseline=0.1cm]
		\node at (0,0) [dot, label=right:$\rho$] (root) {};
		\node at (-3,3) [dot] (centerl) {};
		\node at (0,3) [dot] (centerc) {};
		\node at (3,3) [dot] (centerr) {};
		
		\node at (0,6)  [dot] (centerrr) {};
		\node at (-2,9) [dot] (lastl) {};
		\node at (2,9) [dot] (lastr) {};
		\draw[kernel1] (centerl) to
		node [sloped,below] {\small }     (root);
		\draw[kernel1] (centerr) to
		node [sloped,below] {\small }     (root);
		\draw[kernel1] (centerc) to
		node [sloped,below] {\small }     (root);
		\draw[kernel1] (centerrr) to
		node [sloped,below] {\small }     (centerc);
		\draw[kernel1] (centerrr) to
		node [sloped,below] {\small }     (centerl);
		\draw[kernel1] (lastl) to
		node [sloped,below] {\small }     (centerrr);
		\draw[kernel1] (lastr) to
		node [sloped,below] {\small }     (centerrr);
	\end{tikzpicture}
\end{equation*}
while Item \ref{Item: closed loops} rules out diagrams of the form
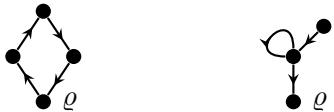
\begin{equation*}
	\begin{tikzpicture}[scale=0.2,baseline=0.1cm]
		\node at (0,0) [dot, label=right:$\rho$] (root) {};
		\node at (-2,3) [dot] (centerl) {};
		\node at (+2,3) [dot] (centerr) {};
		\node at (0,6)  [dot] (centerll) {};
		\draw[kernel1] (root) to
		node [sloped,below] {\small }     (centerl);
		\draw[kernel1] (centerr) to
		node [sloped,below] {\small }     (root);
		\draw[kernel1] (centerll) to
		node [sloped,below] {\small }     (centerr);
		\draw[kernel1] (centerl) to
		node [sloped,below] {\small }     (centerll);
	\end{tikzpicture}\hspace{2cm}
	\begin{tikzpicture}[scale=0.2,baseline=0.1cm]
		\node at (0,0) [dot,label=right:$\rho$] (root) {};
		\node at (0,3) [dot] (center) {};
		\node at (2,5)  [dot] (centerr) {};
		\draw[kernel1] (center) to
		node [sloped,below] {\small }     (root);
		\draw[kernel1] (centerr) to
		node [sloped,below] {\small }     (center);
		\draw[kernel1, loop, min distance=3cm, radius=1cm] (center) to [out=90, in=180] (center);
	\end{tikzpicture}
\end{equation*}
According to the Feynman rules, each MSR Feynman diagram represents an integral involving the kernel of the propagator \emph{i.e.}, the Green's function of the differential operator ruling the linear dynamics of the field theory. In this case the latter is the Heaviside step function and the integral associated to the diagram reads
\begin{equation*}
	\Pi_t(\Gamma)=\int_{[0,T]^{|\mathcal{V}|-1}}\prod_{\substack{e\in\mathcal{E}(\Gamma)\\e_+=\rho}}\Theta(t-s_{e-})\prod_{\substack{e\in\mathcal{E}(\Gamma)\\e_+\neq\rho}}\Theta(s_{e_+}-s_{e_-})\,ds\,,\hspace{2cm}\Gamma\in\mathcal{F}_{\smallMSR}\,,
\end{equation*} 
where $e=(e_+,e_-)$ and the edge points towards $e_+$, $\rho$ is the root of $\Gamma$ and $|\mathcal{V}|$ denotes the cardinality of the set of vertices. Integration is over all variables associated to internal vertices. Given the compact nature of the domain of integration, no test function at the root is needed, as it would be the case in general, see \emph{e.g} \cite[\S 2]{hairer2016analyst}. The motivation behind our choice of a rather specific class of Feynman diagram is dictated by the procedure outlined in the introduction for building all Feynman diagrams that contribute to the perturbative expansion of \eqref{Eq: perturbative expectation action}. Indeed, $\mathbb{E}^{\smallMSR}[f(u_t)]$ is a formal series whose coefficients are the contributions coming from contractions of monomials of the form $\tilde{\psi}\psi^n$ and $\tilde{\psi^2}\psi^{k_1+k_2}$ with $f(\psi)$ via the map $\Gamma_\Theta$, see \eqref{Eq: final formula MSR}. Observe that integration is over the interval $[0,T]$ due to the specific form of the cut-off function $\chi$.

\subsection{Feynman multi-indices}\label{Sec: Feynman multi-indices}
Equation \eqref{Eq: final formula MSR} formalized the idea that MSR expectations are obtained by maximally contracting pairs of auxiliary fields $(\psi,\tilde{\psi})$ generated by the perturbative expansion of the MSR measure around the free theory. Then, the decomposition in terms of MSR Feynman diagrams presented in the previous section follows. In this work we adopt the point of view of \cite{bruned2025renormalising} on the construction and BPHZ renormalization of the cumulants generating functional for a scalar quantum field theory based on multi-indices. This change of paradigm amounts to switching from Feynman diagrams to pre-Feynman diagrams, intended as collections of vertices with free legs that can form proper Feynman diagrams after a pairing procedure analogous to the one presented above.
In the present section we show how multi-indices are well-suited for deriving an explicit expansion of the MSR expectations, see \eqref{Eq: final formula MSR}.

The problem under scrutiny is characterized by an action involving two generic nonlinearities of the form $\tilde{\psi}\alpha(\psi)$ and $\frac{1}{2}\tilde{\psi}^2\beta^2(\psi)$ for $\alpha,\beta\in C^\infty(\R)$. This is rather peculiar in comparison to standard scalar (or complex) field theories since the term involving the square of the diffusivity complicates the combinatorial aspects of the problem. To wit, aiming at a comprehensive treatment of all nonlinearities at once, care must be taken in adapting standard techniques for studying path integrals in terms of Feynman diagrams.
The stark difference between the types of nonlinearities suggests to introduce coloured multi-indices, as \emph{e.g.} in \cite{chandra2024flowapproachgeneralizedkpz}. 
\begin{definition}\label{Def: Feynman multi-indices}
	A \emph{Feynman multi-index} is a map
	\begin{align*}
		\gamma:&\,\N\times \N^2_{\leq}\rightarrow \N\times \N\times \{0,1\}\\
		&(n,\k)\mapsto\gamma(n,\k):=(\gamma_{\alpha}(n),\gamma_\beta(\k),\gamma_\bullet(n))\,,
	\end{align*}
	with compact support, that is $\gamma(n,\k)=0$ for all but finitely many $(n,\k)\in \N\times\N^2$, and such that $\gamma_\bullet(n)=\un{(n=m)}$ for a fixed $m\in\N$. We denote by $\mathcal{M}_{\smallF}$ the linear space of Feynman multi-indices endowed with the natural notion of sum
	\begin{equation*}
		(\gamma+\gamma')(n,\k):=\gamma(n,\k)+\gamma'(n,\k)\,,\qquad \gamma,\gamma'\in\mathcal{M}_{\smallF},n\in\N,\k\in\N^2_{\leq}\,.
	\end{equation*} 
\end{definition}
Given abstract variables $\{z_{\alpha,n},z_{\bullet,n},z_{\beta,\k}\}_{n\in\N,\k\in\N^2}$, a multi-index uniquely identifies a monomial of the form
\begin{equation*}
	z^\gamma:=\prod_{n\in\N,\k\in\N^2}z_{\alpha,n}^{\gamma_{\alpha}(n)}z_{\beta,\k}^{\gamma_\beta(\k)}z_{\bullet,n}^{\gamma_\bullet(n)}\,.
\end{equation*}
With a slight abuse of notation, throughout the paper we will identify monomials with multi-indices and alternate between the two notations to facilitate the reader. The space of abstract monomials forms a commutative algebra when endowed with the product
\begin{equation}\label{Eq: pointwise product multi}
	z^\gamma z^{\gamma'}=z^{\gamma+\gamma'}\,\qquad\gamma,\gamma'\in\mathcal{M}_{\smallF}\,.
\end{equation}
For later convenience, we assign to multi-indices the following gradings: 
\begin{subequations}
	\begin{align}
		&[\gamma]=\sum_{n\in \N}(\gamma_{\alpha}(n)+\gamma_{\bullet}(n))+\sum_{\k\in \N^2_\leq}\gamma_{\beta}(\k)\,,\label{Eq: length multi-index}\\
		&|\gamma|:=\sum_{n\in \N}n(\gamma_{\alpha}(n)+\gamma_\bullet(n))+\sum_{k\in \N^2_\leq}(k_1+k_2)\gamma_{\beta}(\k)\,,\label{Eq: arity}\\
		&\llbracket\gamma\rrbracket:=\sum_{n\in\N}\gamma_\alpha(n)+\sum_{\k\in\N^2_{\leq}}2\gamma_\beta(\k)\,,\label{Eq: number snaky edges}
	\end{align}
\end{subequations}
as well as their restriction to a single component of the multi-index:
\begin{align}
	&[\gamma]_\eta=\sum_{n\in \N}\gamma_{\eta}(n)\,\qquad[\gamma]_\beta=\sum_{\k\in \N^2_\leq}\gamma_{\beta}(\k)\,,\qquad \eta\in\{\alpha,\bullet\}\,.\label{Eq: length multi-index 2}
\end{align}
These gradings have an interpretation in relation to the graphical representation of Feynman diagrams. Recalling that multi-indices codify the number of vertices with a given fertility, we regard $[\gamma]$ as the length of the multi-index, which counts the total number of vertices in the associated pre-Feynman diagram. The integer $|\gamma|$ accounts for the total number of free legs associated to the field $\psi$, while $\llbracket\gamma\rrbracket$ counts the total number of legs representing $\tilde{\psi}$.

Given the interpretation of the multi-index components as representing the number of vertices of a given type having fixed number of outgoing legs, one would be tempted to righteously consider $\beta^2(u)$ as the square of $\beta$ and then to resort to a single-component multi-index as for $\alpha$. Yet, our choice of treating $\beta^2$ as a different object is crucial and motivated by the problem of performing contractions, where the \emph{local nature} of $\beta^2$ plays a role. The symmetry of the square justifies our choice of restricting to triangular $\gamma_\beta$:
\begin{equation}
	\gamma_\beta(k_1,k_2)\neq 0 \quad\Leftrightarrow\quad k_1\leq k_2\,.
\end{equation} 
Observe that $\mathcal{M}_{\smallF}$ encompasses multi-indices which cannot give rise to Feynman diagrams after a leg contraction procedure on the vertices that they represent. To rule them out, we resort once more to the graphical interpretation of Feynman multi-indices, which yields a crucial constraint: The evaluation of the expression in \eqref{Eq: final formula MSR} at the field configurations $\psi=u_0,\tilde{\psi}=0$ together with the recentering at $u_0$ due to the Taylor expansion of the coefficients, prevents free legs from surviving after the contraction procedure that generate Feynman diagrams. This observation, together with the interpretation of the gradings defined above, leads to the following natural sub-class of Feynman multi-indices. 
\begin{definition}\label{Def: population}
	A multi-index $\gamma\in\mathcal{M}_{\smallF}$ is said to be \emph{populated} if
	\begin{equation*}
		\llbracket\gamma\rrbracket=|\gamma|\,.
	\end{equation*}
	We denote by $\mathcal{M}_{\smallF}^p\subset \mathcal{M}_{\smallF}$ the sub-space of populated Feynman multi-indices.
\end{definition}
The repeated action of the differential operator $D_\Theta$ defining the contraction map naturally produces derivatives of the drift and diffusion, as well as of the test-function $f$, evaluated at the initial condition $u_0$. In much the same spirit of Section \ref{Sec: Exotic B-series}, each multi-index $\gamma\in\mathcal{M}_{\smallF}$ uniquely identifies an \emph{elementary differential}.
\begin{definition}\label{Def: elementary differential feynman}
	Let $\alpha,\beta,f\in C^\infty(\R)$ and $u_0\in\R$. The elementary differential $\Upsilon_{u_0}^{\alpha,\beta,f}:\mathcal{M}_{\smallF}\rightarrow\R$ acts on $z^\gamma\in\mathcal{M}_{\smallF}$ as
	\begin{equation}\label{Eq: map multi-elementary}
		\Upsilon_{u_0}^{\alpha,\beta,f}(z^\gamma):=\prod_{n\in\N}  (\alpha^{(n)}(u_0))^{\gamma_\alpha (n)}(f^{(n)}(u_0))^{\gamma_\bullet(n)} 
		\prod_{\k=(k_1,k_2)\in\N^2_\leq}(\beta^{(k_1)}(u_0)\beta^{(k_2)}(u_0))^{\gamma_\beta(\k)}\,.
	\end{equation}
\end{definition}
Observe that the expression in \eqref{Eq: map multi-elementary} is well-defined as it involves a finite number of contributions thanks to the finite support assumption on multi-indices lying in   $\mathcal{M}_{\smallF}$.

Another key ingredient is the notion of symmetry factor of the objects used to label the contributions to the expansion. A reduced notion of symmetry factor for multi-indices was introduced in \cite{bruned2025multi}. In this setting, we adapt the one proposed in \cite{bruned2025renormalising}.
\begin{definition}\label{Def: symmetry factor Feynman multi}
	Let $\gamma\in\mathcal{M}_{\smallF}$. We denote by $\sigma_{\smallF}(z^\gamma)$ its symmetry factor, defined as
	\begin{equation}\label{Eq: symmetry factor Feynman multi}
		\sigma_{\smallF}(z^\gamma):=\prod_{n\in\N}\gamma_\alpha(n)!(n!)^{\gamma_\alpha(n)+\gamma_\bullet(n)}\prod_{\k\in\N^2_\leq}\gamma_\beta(\k)!(\k!)^{\gamma_\beta(\k)}2^{\sum_{i\in\N}\gamma_\beta(i,i)}\,.
	\end{equation}
\end{definition}
The expression of $\sigma_{\smallF}(z^\gamma)$ in Definition \ref{Def: symmetry factor Feynman multi} is motivated as follows: 
	\begin{itemize}
		\item For any $n\in\N,\k\in\N^2_\leq$, the terms $\gamma_\alpha(n)!$ and $\gamma_{\beta}(\k)!$ coincide with the number of permutation of vertices of the same type and same fertility, which are considered indistinguishable within the pre-Feynman diagram.  Indeed, any of these permutations leaves the induced Feynman diagrams unchanged.  
		\item The terms $(n!)^{\gamma_\alpha(n)+\gamma_\bullet(n)}$ and $(\k!)^{\gamma_\beta(\k)}$ count the number of permutations of $\psi$-legs attached to each vertex of the pre-Feynman diagram. 
		\item the factor $2^{\sum_{i\in\N}\gamma_\beta(i,i)}$ is related to the exchange of $\tilde{\psi}$-legs. We should think of each vertex of type $\beta$ as being bipartite, each part corresponding to one of the $\beta$ forming $\beta^2$. Accordingly, we split $\tilde{\psi}^2$ by attaching a single auxiliary field to each part of the bipartite vertex. As a result, when the number of $\psi$-legs attached to the two parts coincide, an exchange of the $\tilde{\psi}$-legs leaves the pre-Feynman diagram unchanged. This crucial hidden symmetry of the pre-Feynman diagrams is an artefact of our choice to treat pairs of $\beta$-vertices as single ones. 
	\end{itemize}
We are now in a position to introduce the map which, to any pre-Feynman diagram $z^\gamma\in\mathcal{M}_{\smallF}$, associates the combination of integrals corresponding to all possible Feynman diagrams that it induces. 
\begin{definition}\label{Def: realization map M}
	The \emph{realization map}
	$\Pi:\mathcal{M}_{\smallF}\times[0,T]\to\R$ is defined by 
	\begin{equation*}
		\Pi_t(z_{\bullet,n})=\Pi_t(z_{\alpha,n})=\Pi_t(z_{\beta,\k})=1\hspace{1.5cm}\forall n\in\N,\k\in\N^2_\leq\,,
	\end{equation*}
	while, for $z^\gamma\in\mathcal{M}_{\smallF}$ such that $[\gamma]>1$,  
	\begin{align}\label{Eq: PiM}
		&\Pi_t(z^\gamma):=\nonumber\\
		&\quad\langle\Gamma_\Theta^{mloc}\bigl[  \prod_{m=0}^\infty\Psi^{k\gamma_\bullet(k)}  \hat{\bigotimes}_{n=0}^\infty(\tilde{\Psi}\Psi^n)^{\otimes \gamma_\alpha(n)}\hat{\bigotimes}_{\k\in\N^2_\leq}(\tilde{\Psi}^2\Psi^{|\k|})^{\hat{\otimes}\gamma_\beta(\k)}\bigr]\Big\vert_{\substack{\psi=0\\\tilde{\psi}=0}},\delta_t\otimes \chi^{\otimes( [\gamma]_\alpha+[\gamma]_\beta)}\rangle\,.
	\end{align}
\end{definition}

\begin{remark}\label{Rem: non connected components}
	In \cite{bruned2025renormalising} the realization map on multi-indices differs from the one in \eqref{Eq: PiM} by the subtraction of all non-connected Feynman diagrams, \emph{i.e.} products of Feynman diagrams generated by the contractions of subgroups of the vertices forming the pre-Feynman diagrams. We could in principle modify \eqref{Eq: PiM} by subtracting them as well. Yet, due to the oriented nature of the propagator any disconnected Feynman diagram vanishes identically, see \cite{bonicelli2023algebraic}. 
\end{remark}

\begin{remark}\label{Rem: population not needed}
	The evaluation at $\psi=\tilde{\psi}=0$ in \eqref{Eq: PiM} ensures that, as soon a the number of $\psi$-legs and $\tilde{\psi}$-legs in $z^\gamma$ differs, $\Pi_t(z^\gamma)=0$. As already remarked in \cite{bruned2025renormalising}, this guarantees that $\Pi_t(z^\gamma)=0$ for all $z^\gamma\in\mathcal{M}_{\smallF}\setminus\mathcal{M}_{\smallF}^p$. Hence, every time the realization map is involved, we are authorized to drop the population condition. We stress that in \eqref{Eq: final formula MSR} distributional-valued functionals are evaluated at $\psi=u_0$. In the proof of Theorem \ref{Thm: MSR multi-indices} we will show that, by expanding the all functions of the field functional $\Psi$ around $u_0 \one$, such evaluation coincides with centering fields on the vanishing configuration.
\end{remark}

\subsection{Derivation of Equation \ref{Eq: MSR Butcher}}\label{Sec: Multi-indices decomposition of expectations}
In this section we derive an expansion of the path integral expectations of the solution to \eqref{Eq: SDE} labelled by Feynman multi-indices. Our strategy is reminiscent of the one adopted in the proof of \cite[Lem. 4.12]{bruned2025renormalising} to derive a formal expansion of the cumulants generating functional for a generic scalar field theory.
\begin{theorem}\label{Thm: MSR multi-indices}
	Let $\alpha,\beta, f$ be real analytic functions and fix $u_0\in\R$. In addition, let $\Upsilon_{u_0}^{\alpha,\beta}$ and $\Pi$ be the maps introduced in Definitions \ref{Def: elementary differential feynman} and \ref{Def: realization map M}, respectively.
	Then, denoting by $u$ the solution of \eqref{Eq: SDE} with initial value $u_0$, the path integral expectation value $\mathbb{E}^{\smallMSR}[f(u_t)]$, $t\in[0,T]$, is represented by the formal series
	\begin{equation}\label{Eq: MSR multi-indices}
		\mathbb{E}^{\smallMSR}[f(u_t)]=\sum_{\gamma\in\mathcal{M}_{\smallF}}\frac{\Upsilon_{u_0}^{\alpha,\beta,f}(z^\gamma)}{\sigma_{\smallF}(z^\gamma)}\Pi_{t}(z^\gamma)\,,
	\end{equation}
	where $\sigma_{\smallF}(z^\gamma)$ is the symmetry factor defined in \eqref{Eq: symmetry factor Feynman multi}. 
\end{theorem}
\begin{proof}
	As a preliminary step we identify the interacting potentials that characterize the MSR action with their Taylor series around $u_0\one$. Let us denote by $r_\alpha,r_\beta, r_f>0$ the radii of convergence of $\alpha$, $\beta$ and $f$ around $u_0$, respectively and let us set $r:=\min\{r_\alpha,r_\beta, r_f\}$. Given the auxiliary nature of the smooth configurations $\psi,\tilde{\psi}\in C^\infty(\R)$, we have the right to cast upon them additional constraints. In particular, we shall pick $\psi$ such that $|\psi(t)-u_0|<r$ for $t\in[0,T]$. Under this assumption, the Taylor expansions of $\alpha(\psi(t))$ and $\beta(\psi(t))$ around $u_0$ are absolutely convergent. Hence we are in the position to consider
	\begin{align*}							&\tilde{\psi}(t)\alpha(\psi(t))=\sum_{\ell=0}^\infty\frac{\alpha^{(\ell)}(u_0)}{\ell!}\tilde{\psi}(t)(\psi(t)-u_0)^\ell\,,\\
	&\tilde{\psi}(t)^2\beta^2(\psi(t))=\sum_{n=0}^\infty\frac{1}{n!}(\beta^2)^{(n)}(u_0)\tilde{\psi}(t)^2(\psi(t)-u_0)^n=\sum_{n=0}^\infty\frac{1}{n!}\Bigl(\sum_{k=0}^n\binom{n}{k}\beta^{(k)}(u_0)\beta^{(n-k)}(u_0)\Bigr)\tilde{\psi}(t)^2(\psi(t)-u_0)^n
		\end{align*}
		\begin{align*}
		&=2\sum_{\substack{n<k}}\frac{1}{n! k!}\beta^{(n)}(u_0)\beta^{(k)}(u_0)\tilde{\psi}(t)^2(\psi(t)-u_0)^{n+k}+\sum_{n=0}^\infty\frac{1}{(n!)^2}(\beta^{(n)}(u_0))^2\tilde{\psi}(t)^2 (\psi(t)-u_0)^{2n}\,.
	\end{align*}
	The separation in the last line is crucial, as it will become apparent in the following. The same procedure must be carried over for $f(\psi(t))$. As a second step we address the expansion of the exponential of the action \eqref{Eq: action} in terms of integral kernels. 
	\begin{align}
		&\exp\left(\int_{0}^T\tilde{\psi}(s)\alpha(\psi(s))\,ds+\frac{1}{2}\int_0^T\tilde{\psi}^2(s)\beta^2(\psi(s))\,ds\right)f(\psi(t))\label{Eq: argument expectation}\\
		&=\sum_{N=0}^\infty\frac{1}{N!}\left(\int_0^T\tilde{\psi}(s)\alpha(\psi(s))\,ds+\frac{1}{2}\int_0^T\tilde{\psi}^2(s)\beta^2(\psi(s))\,ds\right)^N\sum_{k=0}^\infty\frac{f^{(k)}(u_0)}{ k!}(\psi(t)-u_0)^{k}\nonumber\\
		&=\sum_{N=0}^\infty\sum_{M=0}^N\frac{1}{M!(N-M)!}\left(\sum_{l=0}^\infty\frac{\alpha^{(l)}(u_0)}{l!}\int_0^T\tilde{\psi}(s)(\psi(s)-u_0)^l\,ds\right)^M\nonumber\\
		&\hspace{2cm}\Bigl(\sum_{\substack{n<k}}\frac{1}{n! k!}\beta^{(n)}(u_0)\beta^{(k)}(u_0)\int_0^T\tilde{\psi}^2(s)(\psi(s)-u_0)^{k+n}\,ds\nonumber\\
		&\hspace{2.5cm}+\frac{1}{2}\sum_{n=0}^\infty\frac{1}{(n!)^2}(\beta^{(n)}(u_0))^2\int_0^T \tilde{\psi}^2(s)(\psi(s)-u_0)^{2n}\,ds\Bigr)^{N-M}\sum_{k=0}^\infty\frac{f^{(k)}(u_0)}{ k!}(\psi(t)-u_0)^{k}\,.\nonumber
	\end{align}
	In the last identity we exchanged series and integrals by Fubini's theorem thanks to the absolute convergence of the former under the analyticity assumption on the coefficients and the choice of $\psi$ as above. We adopt the short-hand notations
	\begin{equation}
		A_l:=\int_0^T \tilde{\psi}(s)(\psi(s)-u_0)^l\,ds\,\qquad B_{j,i}:=\int_0^T\tilde{\psi}^2(s)(\psi(s)-u_0)^{i+j}\,ds\,.
	\end{equation}
	Resorting to the multinomial theorem we have
	\begin{align}
		&\left(\sum_{l=0}^\infty\frac{\alpha^{(l)}(u_0)}{l!}A_l\right)^M=\sum_{\substack{\{m_l\}_{l\in\N}\\\sum_{l=0}^\infty m_l=M}}\frac{M!}{\prod_{l=0}^\infty m_l!}\prod_{l=1}^\infty (\alpha^{(l)}(u_0))^{m_l}\Bigl(\frac{1}{l!}A_l\Bigr)^{m_l}\label{Eq: expansion alpha}\,.
	\end{align}
	Analogously, an iterated application of the multinomial theorem on the terms involving $\beta$ results in the following more explicit expression
	\begin{align}
		&\Bigl(\sum_{\substack{n<k}}\frac{1}{n! k!}\beta^{(n)}(u_0)\beta^{(k)}(u_0)B_{k,n}+\frac{1}{2}\sum_{n=0}^\infty\frac{1}{(n!)^2}(\beta^{(n)}(u_0))^2B_{n,n}\Bigr)^{N-M}\nonumber\\
		&=\sum_{\substack{p_1,p_2\nonumber\\ p_1+p_2=N-M}}\frac{(N-M)!}{p_1!p_2!}\Bigl(\sum_{\substack{n<k}}\frac{1}{n! k!}\beta^{(n)}(u_0)\beta^{(k)}(u_0)B_{k,n}\Bigr)^{p_1}\Bigl(\frac{1}{2}\sum_{n=0}^\infty\frac{1}{(n!)^2}(\beta^{(n)}(u_0))^2B_{n,n}\Bigr)^{p_2}\nonumber\\
		&=\sum_{\substack{p_1,p_2\\ p_1+p_2=N-M}}\frac{(N-M)!}{p_1!p_2!}\sum_{\substack{\{q_i\}_{i\in\N}, \{r_i\}_{i\in\N}\\\sum_{i}q_i=p_1\\\sum_ir_i=p_2}}\frac{p_1!p_2!}{\prod_{i}q_i!\prod_{i}r_i!}\prod_{i}\Bigl(\sum_{\substack{n=0}}^{i}\frac{1}{n! i!}\beta^{(n)}(u_0)\beta^{(i)}(u_0)B_{n,i}\Bigr)^{q_i}\nonumber\\
		&\hspace{8cm}\prod_{i}\frac{1}{2^{r_i}}\Bigl(\frac{1}{(i!)^2}(\beta^{(i)}(u_0))^2B_{i,i}\Bigr)^{r_i}\nonumber
		\end{align}
		\begin{align}
		&=\sum_{\substack{p_1,p_2\\ p_1+p_2=N-M}}\sum_{\substack{\{q_i\}_{i\in\N}, \{r_i\}_{i\in\N}\\\sum_{i}q_i=p_1\\\sum_ir_i=p_2}}\frac{(N-M)!\prod_{i}\beta^{(i)}(u_0)^{2r_i}}{\prod_{i}r_i!2^{r_i}}\prod_{i}\Bigl(\frac{1}{(i!)^2}B_{i,i}\Bigr)^{r_i}\nonumber\\
		&\hspace{3cm}\prod_{i}\sum_{\substack{\{z_j\}_{j=0}^{i-1}\\\sum_jz_j=q_i}}\frac{\prod_{j=0}^{i-1}(\beta^{(j)}(u_0)\beta^{(i)}(u_0))^{z_j}}{\prod_{j=0}^{i-1}z_j!}\Bigl(\frac{1}{j! i!}B_{j,i}\Bigr)^{z_j}\label{Eq: expansion beta}\,.
	\end{align}
	Plugging \eqref{Eq: expansion alpha} and \eqref{Eq: expansion beta} in \eqref{Eq: argument expectation} and rearranging the finite sums, we obtain
	\begin{align*}
		&\sum_{N=0}^\infty\sum_{M=0}^N\frac{1}{M!(N-M)!}\sum_{\substack{\{m_l\}_{l\in\N}\\\sum_{l=0}^\infty m_l=M}}\frac{M!}{\prod_{l=0}^\infty m_l!}\prod_{l} (\alpha^{(l)}(u_0))^{m_l}\Bigl(\frac{1}{l!}A_l\Bigr)^{m_l}\\
		&\sum_{\substack{p_1,p_2\\ p_1+p_2=N-M}}\sum_{\substack{\{q_i\}_{i\in\N}, \{r_i\}_{i\in\N}\\\sum_{i}q_i=p_1\\\sum_ir_i=p_2}}\frac{(N-M)!\prod_{i}\beta^{(i)}(u_0)^{2r_i}}{\prod_{i}r_i!2^{r_i}}\prod_{i}\frac{1}{2^{r_i}}\Bigl(\frac{1}{(i!)^2}B_{i,i}\Bigr)^{r_i}\\
		&\hspace{2.5cm}\prod_{i}\sum_{\substack{\{z_j\}_{j=0}^{i-1}\\\sum_jz_j=q_i}}\frac{\prod_{j=0}^{i-1}(\beta^{(j)}(u_0)\beta^{(i)}(u_0))^{z_j}}{\prod_{j=0}^{i-1}z_j!}\Bigl(\frac{1}{j! i!}B_{j,i}\Bigr)^{z_j}\sum_{k=0}^\infty\frac{f^{(k)}(u_0)}{ k!}(\psi(t)-u_0)^{k}
	\end{align*}
	\begin{align}
		&=\sum_{\substack{\{m_l\}_{l\in\N}\\ \{z^i_j\}_{(j,i)\in\N^2_<}\\ \{r_i\}_{i\in\N}}}\frac{1}{\prod_{l=0}^\infty m_l!}\prod_{l} (\alpha^{(l)}(u_0))^{m_l}\Bigl(\frac{1}{l!}A_l\Bigr)^{m_l}\frac{\prod_{i}\beta^{(i)}(u_0)^{2r_i}}{\prod_{i}r_i!2^{r_i}}\prod_{i}\Bigl(\frac{1}{(i!)^2}B_{i,i}\Bigr)^{r_i}\nonumber\\
		&\hspace{3cm}\prod_{(j,i)\in\N^2_<}\frac{(\beta^{(j)}(u_0)\beta^{(i)}(u_0))^{z^i_j}}{z^i_j!}\Bigl(\frac{1}{j! i!}B_{j,i}\Bigr)^{z^i_j}\sum_{k=0}^\infty\frac{f^{(k)}(u_0)}{ k!}(\psi(t)-u_0)^{k}\,.\label{Eq: complicated series}
	\end{align}
	As a last step we apply the contraction map $\Gamma_\Theta$. Observe that the products of monomial within the series closely resembles the arguments of $\Gamma_\Theta$ in the definition of the realization map $\Pi$, see \eqref{Eq: PiM}. We stress that, since we do not have a convergence result for the expression in \eqref{Eq: complicated series}, we define the action of $\Gamma_\Theta$ on it by bringing the linear operator inside the series. The proof that the ensuing series coincides with the exotic B-series derived in the previous section justifies this ansatz, which is true for every truncated version of the expansion. We can turn the sum over $\{m_l\},\{z^i_j\},\{r_i\}$ into a sum over Feynman multi-indices via the following identification:
	\begin{itemize}
		\item $\gamma_\beta(j,i)=z_j^i$ for all $(j,i)\in\N_{<}^2$,
		\item $\gamma_\beta(i,i)=r_i$ for all $i\in\N$,
		\item $\gamma_{\alpha}(l)=m_l$ for all $l\in\N$,
		\item $\gamma_{\bullet}(k)=1$ and $\gamma_{\bullet}(p)=0$ for all $p\in\N\setminus\{k\}$,
	\end{itemize}
	Hence
	\begin{align*}
		\mathbb{E}^{\smallMSR}[f(u_t)]&=\sum_{\gamma\in\mathcal{M}_{\smallF}}\frac{\prod_{l} (\alpha^{(l)}(u_0))^{\gamma_\alpha(l)}(f^{(l)}(u_0))^{\gamma_\bullet(l)}\prod_{(j,i)\in\N^2_\leq}(\beta^{(j)}(u_0)\beta^{(i)}(u_0))^{\gamma_{\beta}(j,i)}}{\prod_{l} \gamma_\alpha(l)!(l!)^{\gamma_\alpha(l)+\gamma_\bullet(l)}\prod_{(j,i)\in\N^2_\leq}\gamma_\beta(j,i)!(i!j!)^{\gamma_\beta(j,i)} 2^{\sum_i\gamma_\beta(i,i)}k!}\\
		&\hspace{2cm}\Gamma_\Theta\Bigl(\prod_{l} \Bigl(A_l\Bigr)^{\gamma_\alpha(l)}
		\prod_{(j,i)\in\N^2_\leq}\Bigl(B_{j,i}\Bigr)^{\gamma_\beta(j,i)}\prod_k(\psi_{t}-u_0)^{k\gamma_\bullet(k)}\Bigr)\Big\vert_{\substack{\psi=u_0\\ \tilde{\psi}=0}}\,.
	\end{align*}
	Recalling the definitions of $\Upsilon_{u_0}^{\alpha,\beta,f}, \sigma_{\smallF}(z^\gamma)$ and $\Pi_t(z^\gamma)$ in \eqref{Eq: map multi-elementary}, \eqref{Eq: symmetry factor Feynman multi} and \eqref{Eq: PiM}, respectively we conclude.
\end{proof}

\begin{remark}
	Since the expansion of MSR expectations reduces to the combinatorial problem of accounting for all possible contractions realized by the map $\Gamma_\Theta$, it would be natural to adopt a construction based on MSR Feynman diagrams instead of multi-indices, with the major advantage of adapting it to vector-valued SDEs. Despite this, a formalism based on multi-indices is crucial to handle the combinatorial aspects of the problem, as symmetry factors arise naturally, see the proof of Theorem \ref{Thm: MSR multi-indices}, in stark contrast with \cite{bonicelli2023algebraic} where the use of Feynman diagrams is efficient in the additive and multiplicative case only.   
\end{remark}

\section{Multi-indices exotic B-series}\label{Sec: Multi-indices exotic B-series}
One of the goals of this paper is to draw a comparison between the multi-indices MSR expansion of expectations derived in Section \ref{Sec: Multi-indices decomposition of expectations} and the exotic B-series formulation of the semigroup for the diffusion process obtained in Section \ref{Sec: exotic s-series}. The expansion in Corollary \ref{Cor: key corollary} is close in spirit to the formal series of Theorem \ref{Thm: MSR multi-indices}, the only difference being the set of combinatorial objects used to label their terms. Hence, to conclude our program we must translate the tree-based expansions obtained in the previous sections in the language of multi-indices.

We observe that the exotic coloured trees of Definition \ref{Def: Coloured exotic trees} coincide with the MSR Feynman diagrams introduced in Section \ref{Sec: MSR Feynman diagrams}: They share the same oriented nature, while the exotic decoration codifies the fact that $\beta$ appears in the interacting action through its square. Hence it is natural to adopt the Feynman multi-indices of Section \ref{Sec: Multi-indices decomposition of expectations} as the class of combinatorial object that will replace exotic trees.
Guided by the multi-index representation of Feynman diagrams proposed in \cite{bruned2025renormalising}, we devote this section to a reformulation of the expansion in Corollary \ref{Cor: key corollary} in terms of multi-indices. 
\begin{definition}\label{Def: counting map exotic}
	We introduce a map $\Psi:\mathcal{T}^{\mathfrak{e}}_{\alpha,\beta}\to \mathcal{M}_{\smallF}$ whose action on $\tau\in \mathcal{T}^{\mathfrak{e}}_{\alpha,\beta}$ reads
	\begin{equation*}
		\Psi(\tau):=z_{\bullet,n}\un{(\mathfrak{n}(\rho_{\tau})=n)}\prod_{v\in\mathcal{V}_{\alpha}(\tau)}z_{\alpha,\mathfrak{n}(v)}\prod_{\substack{v_1,v_2\in\mathcal{V}_\beta(\tau) \\ \mathfrak{n}(v_1)\leq \mathfrak{n}(v_2)\\ \mathfrak{e}(v_1)=\mathfrak{e}(v_2)}}z_{\beta, (\mathfrak{n}(v_1), \mathfrak{n}(v_2))}\,,
	\end{equation*}
	where $\mathfrak{n}:\mathcal{V}(\tau)\to\N$ counts the fertility of a vertex, see Definition \ref{Def: graphical notation perturbative}.
\end{definition}

\begin{remark}\label{Rem: population psi}
	The population condition \ref{Def: population} translates at the level of multi-indices the requirement that, to represent Feynman diagrams, the total number of $\tilde{\psi}$-fields and $\psi$-fields must coincide. Since the map $\Psi$ of Definition \ref{Def: counting map exotic} practically counts the vertices of the exotic tree, it automatically maps Feynman diagrams into populated multi-indices. One can adapt the procedure devised in \cite[\S 2]{bruned2025multi} to prove that the image of $\Psi$ coincides with $\mathcal{M}_{\smallF}^p$.
\end{remark}
We are also interested in a map that does the opposite job, namely that sends a populated Feynman multi-index into all the exotic trees that the multi-index represents. Observe that in general this map will not be injective since, for a fixed set of fertilities, we can build different trees. As an example to bear in mind consider the multi-index $z^\gamma=z_{\bullet,1}z_{\alpha,2}z_{\alpha,1}z_{\alpha,2}^2$ and the trees
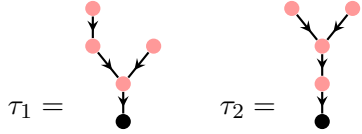
\begin{equation*}
	\tau_1=
	\begin{tikzpicture}[scale=0.2,baseline=0.1cm]
		\node at (0,0) [dot] (root) {};
		\node at (0,2.5) [dotred] (root1) {};
		\node at (-2,5) [dotred] (centerl) {};
		\node at (+2,5) [dotred] (centerr) {};
		\node at (-2,7.5)  [dotred] (centerll) {};
		\draw[kernel1] (centerl) to
		node [sloped,below] {\small }     (root1);
		\draw[kernel1] (centerr) to
		node [sloped,below] {\small }     (root1);
		\draw[kernel1] (centerll) to
		node [sloped,below] {\small }     (centerl);
		\draw[kernel1] (root1) to
		node [sloped,below] {\small }     (root);
	\end{tikzpicture}
	\qquad
	\tau_2=
	\begin{tikzpicture}[scale=0.2,baseline=0.1cm]
		\node at (0,0) [dot] (root) {};
		\node at (0,2.5) [dotred] (root1) {};
		\node at (0,5) [dotred] (root2) {};
		\node at (-2,7.5) [dotred] (centerl) {};
		\node at (+2,7.5) [dotred] (centerr) {};
	\draw[kernel1] (centerl) to
	node [sloped,below] {\small }     (root2);
	\draw[kernel1] (centerr) to
	node [sloped,below] {\small }     (root2);
	\draw[kernel1] (root1) to
	node [sloped,below] {\small }     (root);
	\draw[kernel1] (root2) to
	node [sloped,below] {\small }     (root1);
	\end{tikzpicture}
\end{equation*}
It is a simple exercise to check that $\Psi(\tau_1)=\Psi(\tau_2)=z^\gamma$. Once again we follow the lines of \cite{bruned2025renormalising} and we introduce a suitable \emph{dual} of $\Psi$, whose form will be justified ultimately by the fact that it provides the result that we need. We refer to \cite{zhu2024free} and \cite[\S 3]{bruned2025renormalising} for a detailed explanation of why the symmetry factors are there and in which sense the two maps are in duality.
\begin{definition}\label{Def: map Phi}
	The map 
	$\Phi:\mathcal{M}^p_{\smallF}\to\langle\mathcal{T}^\mathfrak{e}_{\alpha,\beta}\rangle$  is defined as 
	\begin{equation}
		\Phi(z^\gamma):=\sum_{\substack{\tau\in\mathcal{T}^\mathfrak{e}_{\alpha,\beta} \\ \Psi(\tau)=z^\gamma}}\frac{\sigma_{\smallF}(z^\gamma)}{\sigma_{\mathfrak{e}}(\tau)}\tau\,,\qquad z^\gamma\in\mathcal{M}^p_{\smallF}\,,
	\end{equation}
	where $\sigma_{\smallF}(z^\gamma)$ and $\sigma_\mathfrak{e}(\tau)$ are the symmetry factors defined in \ref{Def: symmetry factor Feynman multi} and \ref{Def: automorphism tree}, respectively. 
\end{definition}
In Definition \ref{Def: automorphism tree} we introduced the notion of automorphism of an exotic coloured tree (or Feynman diagram from now on). Since Feynman multi-indices represent pre-Feynman diagrams, \emph{i.e.} collections of vertices with attached free legs that must be joined to form proper diagrams, there is a larger class of transformation that leave the ensuing exotic tree unchanged. 
\begin{definition}[\cite{bruned2025renormalising}, Def. 4.2]\label{Def: isomorphism}
		An isomorphism of the exotic coloured tree $(\tau,\mathfrak{t},\e)\in\cT_{\alpha,\beta}^{\mathfrak{e}}$ is a map that permutes the free legs and vertices of the corresponding pre-Feynman diagram to produce a tree with the same shape as $\tau$ and that preserves the endpoints of each free leg, as well as the decorations. They are the composition of two group actions: $g_\mathcal{V}$ permuting vertices and $g_\mathcal{E}$ permuting free legs such that
	\begin{equation}
		\tau':=g_\mathcal{V}\circ g_\mathcal{E}\,(\tau)\cong \tau\,,
	\end{equation}
	We denote by $\text{Iso}(\tau)$ the space of isomorphisms of $\tau\in\mathcal{T}^\mathfrak{e}_{\alpha,\beta}$.
\end{definition}
The size of the isomorphism group of $\tau$ coincides with $\sigma_{\smallF}(\Psi(\tau))$. Indeed, for any $n\in\N$ and $(k_1,k_2)\in\N^2_\leq$, the factors $(n!)^{\gamma_\alpha(n)}$ and $(k_1!k_2!)^{\gamma_\beta(k_1,k_2)}$ in  $\sigma_{\smallF}(\Psi(\tau))$ correspond to the number of permutations of free legs attached to the vertices, while $\gamma_\alpha(n)!$, $\gamma_\beta(k_1,k_2)!$ account for permutations of vertices with the same fertility. Both operations leave the corresponding exotic tree unchanged. A key difference with respect to \cite{bruned2025renormalising} is the presence of a hidden symmetry responsible for the factor $2^{\sum_{i}\gamma_{\beta}(i,i)}$: if the vertices of $\doublebullet$ in the pre-Feynman diagram have the same number of incoming legs, their permutation belongs to $\text{Iso}(\Phi(z^\gamma))$ and must be accounted for in $\sigma_{\smallF}(z^\gamma)$.
For an example that illustrates the difference between homomorphisms and isomorphisms on a similar class of graphs we refer the reader to \cite[Ex. 5]{bruned2025renormalising}.

\begin{lemma}\label{Lem: number of ways}
	Let $\tau\in\mathcal{T}_{\alpha,\beta}^{\mathfrak{e}}$ and let $z^\gamma=\Psi(\tau)$. Then, the factor $\frac{\sigma_{\smallF}(z^\gamma)}{\sigma_{\mathfrak{e}}(\tau)}$ coincides with the number of different ways of linking the free legs of the pre-Feynman diagram represented by $\Psi(\tau)$ that yield an exotic tree isomorphic to $\tau$.
\end{lemma}
\begin{proof}
	The proof is identical to the one of \cite[Prop. 4.3]{bruned2025renormalising}, to which we refer for the detains. We sketch it here for completeness. 
	Given the isomorphism group of $\tau\in\cT_{\alpha,\beta}^\e$, its stabilizer $S(\tau)$ is the subgroup of transformations that preserve the pairing of free legs, whose cardinality coincides with $\sigma_{\mathfrak{e}}(\tau)$. At the same time the orbits $O(\tau)$ of $\tau$ are those transformations lying in $Iso(\tau)$ that alter the contractions of free legs. In other words, 
	\begin{equation*}
		O(\tau)=\frac{Iso(\tau)}{Aut(\tau)}\,.
	\end{equation*}
	The Orbit-Stabilizer Theorem allows to conclude that
	\begin{equation}
		\sigma_{\smallF}(z^\gamma)=|Iso(\tau)|=|O(\tau)|| Aut(\tau)|=|O(\tau)|\sigma_{\mathfrak{e}}(\tau)\,.
	\end{equation}
	Hence we conclude by identifying the cardinality of $O(\tau)$ with the number of different ways of linking the free legs of the pre-Feynman diagram to obtain $\tau$.
\end{proof}

%

\begin{theorem}\label{Thm: identity expectation multi}
	Let $\Pi$ and $\Pi^{\mathfrak{e}}$ be the realization maps as per Definition \ref{Def: realization map M} and Equation \ref{Eq: realization map exotic trees correct}, respectively and let us consider the intertwining map $\Phi:\mathcal{M}^p_{\smallF}\to\langle\mathcal{T}^\mathfrak{e}_{\alpha,\beta}\rangle$ of Definition \eqref{Def: map Phi}. Then, for any $\gamma\in\mathcal{M}^{\smallB}$,
	\begin{equation}\label{Eq: identity pi phi}
		\Pi^\mathfrak{e}_{t}(\Phi(z^\gamma))=\Pi_{t}(z^\gamma)\,.
	\end{equation}
\end{theorem}
\begin{proof}
	Given the action of $\Gamma_\Theta$ as a map that contracts pairs of free legs $(\tilde{\psi},\psi)$ carried by the vertices of the pre-Feynman diagram in all possible ways, we can represent $\Pi_t(z^\gamma)$ via a combination of Feynman diagrams obtained by linking in all possible ways the free legs of the pre-Feynman diagram represented by $z^\gamma$. On account of Lemma \ref{Lem: number of ways}, this remark translates to the following identity
	\begin{equation*}
		\Pi_t(z^\gamma)=\sum_{\substack{\tau\in\cT_{\alpha,\beta}^{\mathfrak{e}}\\\Psi(\tau)=z^\gamma}}\frac{\sigma_{\smallF}(z^\gamma)}{\sigma_{\mathfrak{e}}(\tau)}\Pi_t^\mathfrak{e}(\tau)\,,
	\end{equation*}
	where the sum is over all non-isomorphic exotic trees and where the factor $\frac{\sigma_{\smallF}(z^\gamma)}{\sigma_{\mathfrak{e}}(\tau)}$ counts the number of ways of joining free legs to create isomorphic trees. 
	Then, resorting to the form of $\Phi$ and to the linearity of $\Pi_t^\mathfrak{e}$, we conclude. Note that, unlike in \cite[Cor. 4.5]{bruned2025renormalising}, the proof is simpler since we did not subtract disconnected Feynman diagrams in the definition of $\Pi_t$, see Remark \ref{Rem: non connected components}.
\end{proof}
\begin{proposition}\label{Prop: expectation S-series multi-indices}
	Under the assumptions of Corollary \ref{Cor: key corollary}, the path integral expectation $\mathbb{E}^{\smallMSR}[f(u_t)]$ coincides with the exotic B-series representing $\mathbb{E}^{u_0}[f(u_t)]$, see Corollary \ref{Cor: key corollary}.
\end{proposition}
\begin{proof}
	First we observe that, for any but fixed $\tau\in\cT_{\alpha,\beta}^\mathfrak{e}$, the associated elementary differential $\Upsilon_{u_0}^{\alpha,\beta,f}(\tau)$ as per Definition \ref{Def: elementary butcher} is not sensitive to the topology of the tree, but only to the fertility of each vertex. This implies that $\Upsilon_{u_0}^{\alpha,\beta,f}(\tau)=\Upsilon_{u_0}^{\alpha,\beta,f}(\Psi(\tau))$. Starting from the expression in \eqref{Eq: exact expansion expectation Pi} and resorting to Theorem \ref{Thm: identity expectation multi} we have
	\begin{align*}
		\mathbb{E}^{u_0}\left[f(u_t)\right]&=\sum_{\tau\in\mathcal{T}^{\mathfrak{e}}_{\alpha,\beta}}\frac{\Upsilon_{u_0}^{\alpha,\beta,f}(\tau)}{\sigma_\mathfrak{e}(\tau)}\Pi_t(\tau)=\sum_{\gamma\in\mathcal{M}_{\smallF}}\sum_{\substack{\tau\in\mathcal{T}_{\alpha,\beta}^\mathfrak{e}\\ \Psi(\tau)=z^\gamma}}\frac{\Upsilon_{u_0}^{\alpha,\beta,f}(\tau)}{\sigma_\mathfrak{e}(\tau)}\Pi_t(\tau)\\
		&=\sum_{\gamma\in\mathcal{M}_{\smallF}} \Upsilon_{u_0}^{\alpha,\beta,f}(z^\gamma)  \sum_{\substack{\tau\in\mathcal{T}^\mathfrak{e}_{\alpha,\beta}\\ \Psi(\tau)=z^\gamma}}\frac{\Pi_t(\tau)}{\sigma_{\mathfrak{e}}(\tau)}\stackrel{\eqref{Eq: identity pi phi}}{=}\sum_{\gamma\in\mathcal{M}_{\smallF}}\frac{\Upsilon_{u_0}^{\alpha,\beta,f}(z^\gamma)}{\sigma_{\smallF}(z^\gamma)}\Pi_t(z^\gamma)\,.
	\end{align*}
	where the sum is over all of $\mathcal{M}_{\smallF}$ as $\Pi$ vanishes identically on $\mathcal{M}_{\smallF}\setminus\mathcal{M}^p_{\smallF}$.
	Notice that, for the first identity to hold true it is crucial that the map $\Psi$ is surjective, see remark \ref{Rem: population psi}.
\end{proof}
On account of the proper construction of the exotic B-series representing the action of the semigroup, see Theorem \ref{Thm: main result S-series}, Proposition \ref{Prop: expectation S-series multi-indices} guarantees that the perturbative computation of the MSR expectations as per \eqref{Eq: MSR multi-indices} yields the correct result, at each perturbative order. In addition, a comparison with \eqref{Thm: main result S-series} yields explicit behaviour in time of the truncated series which, for specific instances of the coefficients, may also lead to convergence results for small enough time horizon. 

As anticipated in the introduction, the whole derivation of both the exotic Butcher series for the action of the semigroup, together with the result obtained in Proposition \ref{Prop: expectation S-series multi-indices}, provides us with a new perspective on the interpretation of the path integral. To wit, the erroneous interpretation of the MSR path integral as the expectation of the interacting part with respect to a Gaussian measure happens to coincide with the procedure of natural growth at the basis of the construction of \eqref{Eq: convergent exact expansion expectation}. Indeed, by interpreting legs of the Feynman diagram as the kernels of Heaviside thetas, the contraction of a field $\psi$ with the interaction vertices $\tilde{\psi}\alpha(\psi)$ and $\tilde{\psi}^2\beta^2(\psi)$ coincides with the grafting of an $\alpha$-vertex and a double $\beta$-vertex, respectively. In addition, the topology of the exotic tree obtained via this contraction procedure gives rise to an iterated integral that, once calculated, coincides with a suitable power in time weighted with the expected combinatorial factor. In conclusion, we should think of the formal series \eqref{Eq: MSR multi-indices} as the definition of $\mathbb{E}^{\smallMSR}[f(u_t)]$.

\appendix

\section{Reduction of exotic trees}\label{Sec: reduction}
In Section \ref{Sec: Realization map and the Connes-Moscovici weight} we completely characterized the combinatorial factor that allows to express the image of the realization map $\Pi_t^\e$ as a power of $t\in[0,T]$. In doing so, a crucial role was played by the alternative definition of the tree factorial for exotic trees given in Theorem \ref{Thm: equivalence factorials}. Here we adopt a different point of view based on the form of the integrals codified by closed paths.
\begin{definition}\label{Def: closed branch}
	Let $\tau\in\cT_{\alpha,\beta}^{\mathfrak{e}}$ and let us denote by $\mathsf{C}:\cT_{\alpha,\beta}^{\mathfrak{e}}\to \cF_{\smallMSR}$ the operator that identifies $\beta$-vertices sharing the same exotic decoration and that forgets the colour decoration $\mathfrak{t}$. We call \emph{closed branch} any closed oriented component $\tau'\subseteq \mathsf{C}(\tau)$. Every closed branch has two \emph{closing paths}, defined as the collections of edges connecting the tip of the closed branch to a common vertex.
\end{definition}

The integral represented by a closed branch coincides with a combination of iterated integrals over all possible permutations of the integration variables corresponding to the internal vertices of the closed branch, that maintain their internal order within the closing paths. This specific type of permutation of two sets, referred to as \emph{shuffle product}, is crucial in the algebraic formulation of the geometric condition for products of iterated integrals labelled by words. Observe that the integral codified by a closed branch is almost like a product of iterated integrals, with the exotic decoration acting like an additional constraint. Therefore $\Pi_t(\tau)$ coincides with $\binom{\mathfrak{l}(\tau_1)+\mathfrak{l}(\tau_2)-2}{\mathfrak{l}(\tau_1)-1}\frac{t^{\mathfrak{l}(\tau)-1}}{(\mathfrak{l}(\tau)-1)!}$, where $\tau_1,\tau_2$ are the closing paths of $\tau$. The binomial factor coincides with the number of ways of shuffling the edges of the closing paths, i.e. of sorting them by preserving the relative orders. This reduction procedure is reminiscent of the result derived in  \cite{hairer2015geometric}, where it is shown that any geometric rough path can be turned into a branched one by a suitable shuffle procedure on the branches. Note that, if the closed path does not have subtrees attached to it, all possible outcomes are identical. However, this is not true in the opposite case, where the choice of where to insert a specific vertex leads to a different tree lying in $\cT$. For this reason we introduce the following \emph{reduction of exotic trees}.   

\begin{definition}\label{Def: redaction operator}
	We define a reduction operator $\mathsf{R}:\cT_{\alpha,\beta}^{\mathfrak{e}}\to \langle\cT\rangle$ defined via the following algorithm:
	\begin{enumerate}
		\item for any $\tau\in\cT_{\alpha,\beta}^{\mathfrak{e}}$, we turn it into an MSR Feynman diagram via the operator $\mathsf{C}$ of Definition \ref{Def: closed branch}
		\item starting from the root of $\mathsf{C}(\tau)$ and following the direction of the tree backwards, we search for the first closed branch. Then, picking one of the closing paths, we insert each of its vertices inside the edges composing the other closing path in all possible ways by respecting the original order. 
		\item we repeat this procedure for all closed branches. 
	\end{enumerate}
	We denote by $\mathcal{R}(\tau)$ the family of trees obtained from $\tau\in\cT_{\alpha,\beta}^{\mathfrak{e}}$ via the reduction operation.
\end{definition}


\begin{example}
	Consider the exotic coloured tree
	\begin{equation*}
		\tau=
		\begin{tikzpicture}[scale=0.2,baseline=0.1cm]
			\node at (0,0) [dot] (root) {};
			\node at (-3,3) [dotred] (centerl) {};
			\node at (0,3) [dotred] (centerc) {};
			\node at (3,3) [dotred] (centerr) {};
			\node at (-3,6)  [dotblue] (centerll) {\mbox{\small $1$}};
			\node at (0,6)  [dotblue] (centerrr) {\mbox{\small $1$}};
			\node at (-2,9) [dotred] (lastl) {};
			\node at (2,9) [dotred] (lastr) {};
			\draw[kernel1] (centerl) to
			node [sloped,below] {\small }     (root);
			\draw[kernel1] (centerr) to
			node [sloped,below] {\small }     (root);
			\draw[kernel1] (centerc) to
			node [sloped,below] {\small }     (root);
			\draw[kernel1] (centerrr) to
			node [sloped,below] {\small }     (centerc);
			\draw[kernel1] (centerll) to
			node [sloped,below] {\small }     (centerl);
			\draw[kernel1] (lastl) to
			node [sloped,below] {\small }     (centerrr);
			\draw[kernel1] (lastr) to
			node [sloped,below] {\small }     (centerrr);
		\end{tikzpicture}\hspace{2cm}
		\mathsf{C}(\tau)=\begin{tikzpicture}[scale=0.2,baseline=0.1cm]
			\node at (0,0) [dot] (root) {};
			\node at (-3,3) [dot] (centerl) {};
			\node at (0,3) [dot] (centerc) {};
			\node at (3,3) [dot] (centerr) {};
			\node at (0,6)  [dot] (centerrr) {};
			\node at (-2,9) [dot] (lastl) {};
			\node at (2,9) [dot] (lastr) {};
			\draw[kernel1] (centerl) to
			node [sloped,below] {\small }     (root);
			\draw[kernel1] (centerr) to
			node [sloped,below] {\small }     (root);
			\draw[kernel1] (centerc) to
			node [sloped,below] {\small }     (root);
			\draw[kernel1] (centerrr) to
			node [sloped,below] {\small }     (centerc);
			\draw[kernel1] (centerrr) to
			node [sloped,below] {\small }     (centerl);
			\draw[kernel1] (lastl) to
			node [sloped,below] {\small }     (centerrr);
			\draw[kernel1] (lastr) to
			node [sloped,below] {\small }     (centerrr);
		\end{tikzpicture}
	\end{equation*}
	By applying the algorithm we get
	\begin{equation*}
		\mathsf{R}(\tau)= 2\;
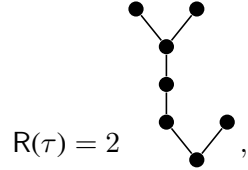
		\begin{tikzpicture}[scale=0.2,baseline=0.1cm]
			\node at (0,0) [dot] (root) {};
			\node at (-2,2.5) [dot] (centerlb) {};
			\node at (-2,5) [dot] (centerl) {};
			\node at (2,2.5) [dot] (centerr) {};
			\node at (-2,7.5)  [dot] (centerll) {};
			\node at (-4,10) [dot] (lastl) {};
			\node at (-0,10) [dot] (lastr) {};
			
			\draw[kernel] (centerlb) to
			node [sloped,below] {\small }     (root);
			\draw[kernel] (centerl) to
			node [sloped,below] {\small }     (centerlb);
			\draw[kernel] (centerr) to
			node [sloped,below] {\small }     (root);
			\draw[kernel] (centerll) to
			node [sloped,below] {\small }     (centerl);
			\draw[kernel] (lastl) to
			node [sloped,below] {\small }     (centerll);
			\draw[kernel] (lastr) to
			node [sloped,below] {\small }     (centerll);
		\end{tikzpicture}\,,	
	\end{equation*}
	where the factor $2$ derives from the fact that inserting the only internal vertex of one branch above or below the one of the other does not change the structure of the ensuing tree. It is useful to go through an example in which such distinction is crucial:
	\begin{equation*}
		\tau'=
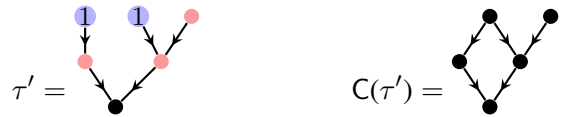
\begin{tikzpicture}[scale=0.2,baseline=0.1cm]
			\node at (0,0) [dot] (root) {};
			\node at (-2,3) [dotred] (centerl) {};
			\node at (3,3) [dotred] (centerr) {};
			\node at (-2,6)  [dotblue] (centerll) {\mbox{\small $1$}};
			\node at (1.5,6)  [dotblue] (centerrr) {\mbox{\small $1$}};
			\node at (5,6) [dotred] (topr) {};
			\draw[kernel1] (centerl) to
			node [sloped,below] {\small }     (root);
			\draw[kernel1] (centerr) to
			node [sloped,below] {\small }     (root);
			\draw[kernel1] (centerrr) to
			node [sloped,below] {\small }     (centerr);
			\draw[kernel1] (centerll) to
			node [sloped,below] {\small }     (centerl);
			\draw[kernel1] (topr) to
			node [sloped,below] {\small }     (centerr);
		\end{tikzpicture}\hspace{2cm}
		\mathsf{C}(\tau')=
		\begin{tikzpicture}[scale=0.2,baseline=0.1cm]
			\node at (0,0) [dot] (root) {};
			\node at (-2,3) [dot] (centerl) {};
			\node at (2,3) [dot] (centerr) {};
			\node at (0,6)  [dot] (centerrr) {};
			\node at (4,6) [dot] (topr) {};
			\draw[kernel1] (centerl) to
			node [sloped,below] {\small }     (root);
			\draw[kernel1] (centerr) to
			node [sloped,below] {\small }     (root);
			\draw[kernel1] (centerrr) to
			node [sloped,below] {\small }     (centerr);
			\draw[kernel1] (centerrr) to
			node [sloped,below] {\small }     (centerl);
			\draw[kernel1] (topr) to
			node [sloped,below] {\small }     (centerr);
		\end{tikzpicture}
	\end{equation*}
	This time we have
	\begin{equation*}
		\mathsf{R}(\tau')=
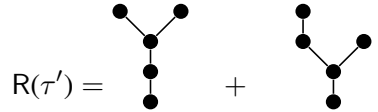
		\begin{tikzpicture}[scale=0.2,baseline=0.1cm]
			\node at (0,0) [dot] (root) {};
			\node at (0,2) [dot] (1) {};
			\node at (0,4) [dot] (2) {};
			\node at (-2,6)  [dot] (3) {};
			\node at (2,6)  [dot] (4) {};
			
			\draw[kernel] (1) to node [sloped,below] {\small } (root);
			\draw[kernel] (2) to node [sloped,below] {\small } (1);
			\draw[kernel] (3) to node [sloped,below] {\small } (2);
			\draw[kernel] (4) to node [sloped,below] {\small } (2);
		\end{tikzpicture}
		\quad+\quad
		\begin{tikzpicture}[scale=0.2,baseline=0.1cm]
			\node at (0,0) [dot] (root) {};
			\node at (0,2) [dot] (1) {};
			\node at (-2,4) [dot] (2) {};
			\node at (2,4)  [dot] (4) {};
			\node at (-2,6)  [dot] (3) {};

			\draw[kernel] (1) to node [sloped,below] {\small } (root);
			\draw[kernel] (2) to node [sloped,below] {\small } (1);
			\draw[kernel] (3) to node [sloped,below] {\small } (2);
			\draw[kernel] (4) to node [sloped,below] {\small } (1);
		\end{tikzpicture}\,.
	\end{equation*}
\end{example}
From what we observed above, given $\tau\in\cT_{\alpha,\beta}^\e$ and denoting by $\mathcal{R}(\tau)\subset \cT$ the space of trees obtained via the reduction operator $\mathsf{R}$, it follows that 
\begin{equation}\label{Eq: reduced Pi}
	\Pi_{st}^\e(\tau)=\sum_{\bar{\tau}\in\mathcal{R}(\tau)}\Pi_{st}(\bar{\tau})\,,
\end{equation}
where $\Pi_{st}(\bar{\tau})$ is the branched integral represented by the non-decorated rooted tree $\bar{\tau}$. Interestingly enough, the explicit expression of the realization map on $\cT_{\alpha,\beta}^\e$ derived in Proposition \ref{Prop: explicit form Pi} implies the following result, which is of independent interest. 
\begin{lemma}\label{Lem: identity factorials}
	Let $\tau\in\cT^\e_{\alpha,\beta}$ and let $\mathcal{R}(\tau)$ the family of trees obtained from $\tau$ via the reduction operator $\mathsf{R}$, see Definition \ref{Def: redaction operator}. Then
	\begin{equation}
		\frac{1}{\tau!}=\sum_{\bar{\tau}\in\mathcal{R}(\tau)}\frac{1}{\bar{\tau}!}\,.
	\end{equation}
\end{lemma}
\begin{proof}
	The evaluation maps in the right-hand side of \eqref{Eq: reduced Pi} fall under the scope of \eqref{Eq: explicit expression pi} by regarding $\tau\in\cT$ as a tree lying in $\cT_{\alpha}\subset \cT_{\alpha,\beta}$ by identifying all the vertices aside from the root with $\alpha$-decorated ones. Then
	\begin{equation*}
		\Pi^\e_t(\tau)=\sum_{\bar{\tau}\in\mathcal{R}(\tau)}\frac{|\bar{\tau}|}{\bar{\tau}!}t^{|\bar{\tau}|-1}\,.
	\end{equation*}
	At the same time, \eqref{Eq: explicit expression pi} guarantees that $\Pi^\e_t(\tau)=\frac{|\tau|}{\tau!}t^{|\tau|-1}$. Recalling that $\mathsf{R}$ preserves the number of vertices we conclude. 
\end{proof}

The identification in \eqref{Eq: reduced Pi} can be further refined to compare our definition of exotic tree factorial with the definition of factorial for a generic combinatorial Hopf algebra proposed in \cite{curry2020planarly}. First we observe that the action of $\Pi_{st}$ on $\cT$ identifies iterated integrals of the identity path, hence it satisfies a Chen's relation of the form
	\begin{equation}\label{Eq: Chen exotic}
		\Pi_{st}(\bar{\tau})=(\Pi_{su}\circ B_+\otimes \Pi_{ut}\circ B_+)\Delta B_-(\tau)\,,
	\end{equation}
	where $\Delta$ denotes the Butcher-Connes-Kreimer coproduct on $\cT$, see \cite{hoffman2003combinatorics} for a definition. That \eqref{Eq: Chen exotic} holds true follows from observing that, denoting by $\bar{\Pi}_{st}(\tau)$ the iterated integral described by a rooted trees $\tau\in\cT$ as recursively defined in \cite{gubinelli2010ramification}, we have the identification 
	\begin{equation*}
		\Pi_{st}(\tau)=\bar{\Pi}_{st}(B_-(\tau))\,,\hspace{2cm}\forall\tau\in\cT\,.
	\end{equation*}
	Accordingly
	\begin{align*}
		\Pi_{st}(\tau)=\bar{\Pi}_{st}(B_-(\tau))=\sum_{(B_-(\tau))}\bar{\Pi}_{su}(\tau_1)\bar{\Pi}_{ut}(\tau_2)=\sum_{(B_-(\tau))}\Pi_{su}(B_+(\tau_1))\Pi_{ut}(B_+(\tau_2))\,,
	\end{align*}
	where we adopted Sweedler's notation for the Butcher-Connes-Kreimer coproduct.
	Exploiting the explicit form of the exotic realization map as per \eqref{Eq: explicit expression pi} implies
	\begin{align*}
		\frac{|\tau|}{\tau!}(t-s)^{|\tau|-1}&=\sum_{\bar{\tau}\in\mathcal{R}(\tau)}\Pi_{st}(\bar{\tau})=\sum_{\bar{\tau}\in\mathcal{R}(\tau)}\sum_{(B_-(\bar{\tau}))}\Pi_{su}(B_+(\bar{\tau}_1))\Pi_{ut}(B_+(\bar{\tau}_2))\\
		&=\sum_{\bar{\tau}\in\mathcal{R}(\tau)}\sum_{(B_-(\bar{\tau}))}\bar{\Pi}_{su}(\bar{\tau}_1)\bar{\Pi}_{ut}(\bar{\tau}_2)=\sum_{\bar{\tau}\in\mathcal{R}(\tau)}\sum_{(B_-(\bar{\tau}))}\frac{(u-s)^{|\bar{\tau}_1|}(t-u)^{|\bar{\tau}_2|}}{\bar{\tau}_1!\bar{\tau}_2!}
	\end{align*}
	By choosing $t-s=2$, $u-s=t-u=1$ we derive the relation
	\begin{equation}\label{Eq: formula tree factorial 2}
		\frac{|\tau|}{\tau!}=\frac{1}{2^{|\tau|-1}}\sum_{\bar{\tau}\in\mathcal{R}(\tau)}\sum_{(B_-(\bar{\tau}))}\frac{1}{\bar{\tau}_1!\bar{\tau}_2!}\,.
	\end{equation}
The identity \eqref{Eq: formula tree factorial 2} should be considered as an equivalent definition of tree factorial for exotic rooted trees, closely resembling the one in \cite{curry2020planarly} exept for the presence of the reduction operator. In particular, the left-hand side could be heuristically thought of as $(B_-(\tau)!)^{-1}$, where $B_-$ is the inverse of $B_+$ which maps an exotic tree into the corresponding exotic coloured forest as per Definition \ref{Def: Coloured exotic trees}. Even if Definition \ref{Def: tree factorial} adapts to exotic coloured forests, a proof of the identity $\frac{|\tau|}{\tau!}=\frac{1}{B_-(\tau)!}$ seems out of reach. We postpone its investigation to a future work. 
The similarity with \cite[Eq. 15]{curry2020planarly} should be regarded as a further justification of Definition \ref{Def: tree factorial}.

\section{Examples}\label{Sec: applications}
To show the strength of the exotic B-series representation of expectations of the solution to an It\^o diffusion, we analyse how our construction specializes to situations where the statistics of the solution is known. In the following examples the formal series \eqref{Eq: convergent exact expansion expectation} converges uniformly in time.
Until now, by working with generic coefficients $\alpha$ and $\beta$ in \eqref{Eq: SDE}, no constraints on the topology of the exotic trees used to label the B-series were imposed. In practice, the form of the coefficients forces us to adopt a \emph{rule}, \emph{i.e.}, a restriction on the admissible fertility of the vertices that compose the tree. 
\subsection*{Ornstein-Uhlenbeck process}
Let us consider the well known Ornstein-Uhlenbeck process, defined as the solution of the Cauchy problem on $\R_+$
\begin{equation}
	\begin{cases}
		du_t=-au_tdt+\sigma dW_t\,,\qquad \\
		u(0)=u_0
	\end{cases}\,,
\end{equation}
where $a,\sigma>0$ are non-negative constants. To make contact with the notation of the previous sections we set $\alpha(u_t):=-au_t$ and $\beta(u_t):=\sigma$. Here we can better understand the notion of rule: given the construction of the It\^o-Taylor expansion of Section \ref{Sec: ito-taylor} in terms of differential operators and their interpretation as grafting operations, it follows that vertices of type $\alpha$ can have either zero or one incoming edges, while $\beta$-vertices must have fertility zero. An additional rule is set by the choice of the smooth function $f$. Being interested in the expectation of the process, let us consider first $f(u):=u$. At the level of trees, this translates into asking for the root to have at most one incoming edge as well. As a result, the only trees allowed are linear trees obtained by concatenation of $\alpha$-vertices. For any $\tau\in\cT_{\alpha}$ of this form it immediately follows that
\begin{equation*}
	\frac{|\tau|!}{\sigma_{\mathfrak{e}}(\tau)\,\tau!}=1\,,\qquad \Upsilon_{u_0}^{\alpha,\beta,f}(\tau)=u_0 (-a)^{|\tau|-1}\,.
\end{equation*}
Observing that simple linear trees can be identified with their number of vertices,
the expression in \eqref{Eq: convergent exact expansion expectation} reduces to 
\begin{equation*}
	\mathbb{E}^{u_0}[u_t]=\sum_{n=1}^\infty(-a)^{n-1} u_0 \frac{t^{n-1}}{(n-1)!}=u_0e^{-at}\,.
\end{equation*}
For what concerns the variance of $u_t$ the rule changes, since in this case we consider $f(u_t)=u_t^2$, so that the root can have up to $2$ incoming edges. This times the only trees that we can builds according to the rule have the form
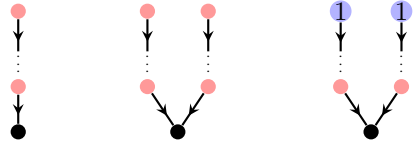
\begin{equation*}
	\begin{tikzpicture}[scale=0.2,baseline=0.1cm]
		\node at (0,0) [dot] (root) {};
		\node at (0,3)  [dotred] (center) {};
		\node at (0,8)  [dotred] (center2) {};
		\node at (0,4)  [scale=0.5] () {.};
		\node at (0,4.5)  [scale=0.5] () {.};
		\node at (0,5)  [scale=0.5] () {.};
		\node at (0,5)  [scale=0.5] (fine) {};
		\draw[kernel1] (center) to
		node [sloped,below] {\small }     (root);
		\draw[kernel1] (center2) to
		node [sloped,below] {\small }     (fine);
	\end{tikzpicture}
	\hspace{1.5cm}
	\begin{tikzpicture}[scale=0.2,baseline=0.1cm]
		\node at (0,0) [dot] (root) {};
		\node at (-2,3) [dotred] (centerl) {};
		\node at (+2,3) [dotred] (centerr) {};
		\node at (-2,4)  [scale=0.5] () {.};
		\node at (-2,4.5)  [scale=0.5] () {.};
		\node at (-2,5)  [scale=0.5] () {.};
		\node at (-2,5)  [scale=0.5] (fine1) {};
		\node at (2,4)  [scale=0.5] () {.};
		\node at (2,4.5)  [scale=0.5] () {.};
		\node at (2,5)  [scale=0.5] () {.};
		\node at (2,5)  [scale=0.5] (fine2) {};
		\node at (-2,8)  [dotred] (centerll) {};
		\node at (2,8)  [dotred] (centerrr) {};
		\draw[kernel1] (centerl) to
		node [sloped,below] {\small }     (root);
		\draw[kernel1] (centerr) to
		node [sloped,below] {\small }     (root);
		\draw[kernel1] (centerrr) to
		node [sloped,below] {\small }     (fine2);
		\draw[kernel1] (centerll) to
		node [sloped,below] {\small }     (fine1);
	\end{tikzpicture}
	\hspace{1.5cm}
	\begin{tikzpicture}[scale=0.2,baseline=0.1cm]
		\node at (0,0) [dot] (root) {};
		\node at (-2,3) [dotred] (centerl) {};
		\node at (+2,3) [dotred] (centerr) {};
		\node at (-2,4)  [scale=0.5] () {.};
		\node at (-2,4.5)  [scale=0.5] () {.};
		\node at (-2,5)  [scale=0.5] () {.};
		\node at (-2,5)  [scale=0.5] (fine1) {};
		\node at (2,4)  [scale=0.5] () {.};
		\node at (2,4.5)  [scale=0.5] () {.};
		\node at (2,5)  [scale=0.5] () {.};
		\node at (2,5)  [scale=0.5] (fine2) {};
		\node at (-2,8)  [dotblue] (centerll) {\mbox{\small $1$}};
		\node at (2,8)  [dotblue] (centerrr) {\mbox{\small $1$}};
		\draw[kernel1] (centerl) to
		node [sloped,below] {\small }     (root);
		\draw[kernel1] (centerr) to
		node [sloped,below] {\small }     (root);
		\draw[kernel1] (centerrr) to
		node [sloped,below] {\small }     (fine2);
		\draw[kernel1] (centerll) to
		node [sloped,below] {\small }     (fine1);
	\end{tikzpicture}
\end{equation*}
where the dots denote that we can concatenate an arbitrary number of $\alpha$-vertices. Note that, since $\beta$-vertices can only close the concatenation, we can effectively forget about them when computing the combinatorial part of the series. This time the elementary differential is different in that the derivative of $u^2$ produces a factor $2$. Then the contribution from the root and all the concatenations of $\alpha$-vertices yields
\begin{equation*}
	u_0^2+2u_0^2 (e^{-at}-1)\,,
\end{equation*}
Now consider all trees of the second kind above. This time, to account for the symmetry factor, we must distinguish the case where the number of vertices on both sides of the tree coincide or not. Denoting by $\cT^2$ the second family of trees,
\begin{align*}
	\sum_{\tau\in\cT^2}\frac{|\tau|!}{\sigma_\mathfrak{e}(\tau)\tau!}&\Upsilon_{u_0}^{\alpha,\beta,f}(\tau)\frac{t^{\mathfrak{l}(\tau)}}{\mathfrak{l}(\tau)!}=\sum_{n=2}^\infty\sum_{m=1}^{n-1}\frac{(n+m)!}{n!m!}2u_0^2(-a)^{n+m}\frac{t^{n+m}}{(n+m)!}+\sum_{n=1}^\infty\frac{2n!}{2(n!)^2}2u_0^2(-a)^{2n}\frac{t^{2n}}{2n!}\\
	&=2u_0^2\sum_{n=2}^\infty\sum_{m=1}^{n-1}\frac{(-at)^{n+m}}{n!m!}+u_0^2\sum_{n=1}^\infty\frac{(-at)^{2n}}{(n!)^2}
\end{align*}
By completing the sums with $n=0,1$ and $m=0$ and adding the contribution from linear trees computed above, one obtains
\begin{align*}
	u_0^2+2u_0^2 (e^{-at}-1)&+\sum_{\tau\in\cT^2}\frac{|\tau|!}{\sigma_\mathfrak{e}(\tau)\tau!}\Upsilon_{u_0}^{\alpha,\beta,f}(\tau)\frac{t^{\mathfrak{l}(\tau)}}{\mathfrak{l}(\tau)!}=2u_0^2\sum_{n=0}^\infty\sum_{m=0}^{n-1}\frac{(-at)^{n+m}}{n!m!}+u_0^2\sum_{n=0}^\infty\frac{(-at)^{2n}}{(n!)^2}=u_0^2 e^{-2at}.
\end{align*}
This contribution coincides with the square of the expectation value, which is subtracted when computing the variance. 
Moving to trees belonging to the third family, that we denote by $\cT^3$, an analogous reasoning applies. Once again we must break down the sum into symmetric and non-symmetric trees to obtain 
\begin{align*}
	\sum_{\tau\in\cT^3}&\frac{|\tau|!}{\sigma_\mathfrak{e}(\tau)\tau!}\Upsilon_{u_0}^{\alpha,\beta,f}(\tau)\frac{t^{\mathfrak{l}(\tau)}}{\mathfrak{l}(\tau)!}=\sum_{n=1}^\infty\sum_{m=0}^{n-1}\frac{(n+m)!}{n!m!}2\sigma^2(-a)^{n+m}\frac{t^{n+m+1}}{(n+m+1)!}+\sum_{n=0}^\infty\frac{(2n)!}{2n!^2}2\sigma^2(-a)^{2n}\frac{t^{2n+1}}{(2n+1)!}\\
	&=-\frac{2\sigma^2}{a}\sum_{n=1}^\infty\sum_{m=0}^{n-1}\frac{(-at)^{n+m+1}}{(n+m+1)n!m!}-\frac{\sigma^2}{a}\sum_{n=0}^\infty\frac{(-at)^{2n+1}}{(2n+1)n!^2}=-\frac{\sigma}{a}\sum_{n=0}^\infty\sum_{m=0}^\infty\frac{(-at)^{n+m+1}}{(n+m+1)n!m!}\\
	&=-\frac{\sigma}{a}\sum_{k=0}^\infty\frac{(-at)^{k+1}2^k}{(k+1)k!}=\sigma\sum_{k=0}^\infty\frac{(-2a)^k}{k!}\int_0^tx^k\,dx=\frac{\sigma}{2a}(1-e^{-2\sigma t})\,.
\end{align*}
Putting all contributions together we obtain the exact expression for the variance of the Ornstein-Uhlenbeck process 
\begin{equation*}
	\mathbb{E}^{u_0}[u^2_t]-(\mathbb{E}^{u_0}[u_t])^2=\frac{\sigma}{2a}(1-e^{-2\sigma t})\,.
\end{equation*}

\subsection*{Geometric Brownian motion}
Anther relevant example in which the measure of the solution can be determined explicitely is the \emph{geometric Brownian motion}
\begin{equation}\label{Eq: geometric BM}
	\begin{cases}
		du_t=au_tdt+\sigma u_tdW_t\,,\qquad \\
		u(0)=u_0
	\end{cases}\,,
\end{equation}
where $a,\sigma\geq0$ are again non-negative constants. The solution can be readily derived via a logarithmic change of variable and It\^o's formula. Instead of mimicking the line of reasoning adopted for the Ornstein-Uhlenbeck process, we focus on a natural feature of problems with a linear drift. When writing the solution in its weak form as a stochastic convolution, nothing prevents us from including the linear part into the differential operator. If we do so, the weak solution is obtained by convolving with the Green's function of the modified operator, see \cite[\S 2.1]{bonicelli2023algebraic} where this point of view is justified from the standpoint of Feynman diagrams via the introduction of effective edges. We can give a motivation of this choice by recalling that, as proved in Theorem \ref{Thm: main result S-series}, the contributions to the It\^o-Taylor series are in one-to-one correspondence with exotic coloured trees that abide by the rule imposed by the explicit form of $\alpha$ and $\beta$. For a linear drift, a vertex $\alpha$ cannot give rise to a bifurcation since it carries a single $\psi$-leg. Let us focus on the simpler tree with a single vertex $\alpha$, that represents a contribution of the form $\int_0^t\alpha(u_0)f^{(1)}(u_0)\,ds$. When expanding at higher orders the only effect of the linear part is to add internal vertices $\alpha$ to form linear trees as the ones in the previous example. In other words we can group the contributions
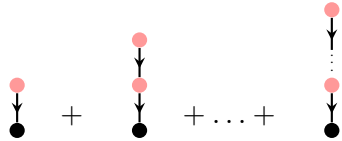
\begin{equation*}
	\begin{tikzpicture}[scale=0.2,baseline=0.1cm]
		\node at (0,0) [dot] (root) {};
		\node at (0,3)  [dotred] (center) {};
		\draw[kernel1] (center) to
		node [sloped,below] {\small }     (root);
	\end{tikzpicture}\quad+\quad
	\begin{tikzpicture}[scale=0.2,baseline=0.1cm]
		\node at (0,0) [dot] (root) {};
		\node at (0,3)  [dotred] (center) {};
		\node at (0,6)  [dotred] (center2) {};
		
		\draw[kernel1] (center) to
		node [sloped,below] {\small }     (root);
		\draw[kernel1] (center2) to
		node [sloped,below] {\small }     (center);
	\end{tikzpicture}\quad+\ldots+\quad
	\begin{tikzpicture}[scale=0.2,baseline=0.1cm]
		\node at (0,0) [dot] (root) {};
		\node at (0,3)  [dotred] (center) {};
		\node at (0,8)  [dotred] (center2) {};
		\node at (0,4)  [scale=0.5] () {.};
		\node at (0,4.5)  [scale=0.5] () {.};
		\node at (0,5)  [scale=0.5] () {.};
		\node at (0,5)  [scale=0.5] (fine) {};
		\draw[kernel1] (center) to
		node [sloped,below] {\small }     (root);
		\draw[kernel1] (center2) to
		node [sloped,below] {\small }     (fine);
	\end{tikzpicture}
\end{equation*} 
up to an arbitrary number of $\alpha$-vertices. The analytic contribution to this effective edge reads
\begin{equation*}
	a u_0f^{(1)}(u_0)\int_0^t\left(1+\sum_{n=1}^\infty a^n\int_{[0,s]^n}ds_n\ldots ds_1\right)\,ds=a u_0f^{(1)}(u_0)\int_0^te^{as}\,ds\,.
\end{equation*}
Hence, the net effect of adding up all trees that differ by the number of $\alpha$-vertices between two fixed consecutive vertices is to substitute the edge with an effective one that encodes integration against the exponential kernel $e^{as}$. This procedure can be carried over for all edges, hence leaving as with trees involving only effective edges and $\beta$-vertices. An important difference is that the elementary differential becomes time-dependent: each derivative of the coefficient represented by a vertex must be evaluated at $u_0 e^{as}$, where $s$ is the integration variable that it carries. We can justify this choice by observing that, after this change of paradigm, the expectation value of $u$ receives a single contribution, the one from the root. Hence we should incorporate the elementary differential inside the realization map, in such a way that $\tilde{\Pi}_t(\bullet):=u_0e^{at}$. The definition of $\tilde{\Pi}$
on trees $\tau\in\cT_\beta^\e$ with $|\tau|>1$
differs from the one in \eqref{Eq: realization map exotic trees correct} by the presence of derivatives of $\beta$ evaluated at $u_0e^{as}$ for the integrated internal variable $s$.
This simplification, that we could have exploited also in the previous example, becomes crucial here: this time the rule enforced by \eqref{Eq: geometric BM} allows for bifurcations stemming from $\beta$-vertices. Hence, without this reduction, we could have trees with an arbitrary number of closed paths and $\alpha$-vertices, making the explicit resummation of the B-series unpractical. 
For what concerns $f(u_t)=u_t^2$, the only family of exotic trees that we can construct with the sole $\beta$-vertices have the form
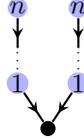
\begin{equation*}
	\begin{tikzpicture}[scale=0.2,baseline=0.1cm]
		\node at (0,0) [dot] (root) {};
		\node at (-2,3) [dotblue] (centerl) {\mbox{\small $1$}};
		\node at (+2,3) [dotblue] (centerr) {\mbox{\small $1$}};
		\node at (-2,4)  [scale=0.5] () {.};
		\node at (-2,4.5)  [scale=0.5] () {.};
		\node at (-2,5)  [scale=0.5] () {.};
		\node at (-2,5)  [scale=0.5] (fine1) {};
		\node at (2,4)  [scale=0.5] () {.};
		\node at (2,4.5)  [scale=0.5] () {.};
		\node at (2,5)  [scale=0.5] () {.};
		\node at (2,5)  [scale=0.5] (fine2) {};
		\node at (-2,8)  [dotblue] (centerll) {\mbox{\small $n$}};
		\node at (2,8)  [dotblue] (centerrr) {\mbox{\small $n$}};
		\draw[kernel1] (centerl) to
		node [sloped,below] {\small }     (root);
		\draw[kernel1] (centerr) to
		node [sloped,below] {\small }     (root);
		\draw[kernel1] (centerrr) to
		node [sloped,below] {\small }     (fine2);
		\draw[kernel1] (centerll) to
		node [sloped,below] {\small }     (fine1);
	\end{tikzpicture}
\end{equation*}
where the dots stand for an arbitrary number of pairs $\doublebullet$. For any $\tau\in\cT_\beta$ of this form such that $|\tau|=n$ we have
\begin{align*}
	\frac{1}{\sigma_\e(\tau)}\tilde{\Pi}_t^\e(\tau)&=\frac{1}{2}\int_{[0,t]^{n-1}}2\sigma^{2(n-1)}u_0^2e^{2as_n} e^{2a(t-s_1)}\ldots e^{2a(s_{n-1}-s_n)} \,ds_n\ldots ds_1=\sigma^{2(n-1)}u_0^2e^{2at}\frac{t^{n-1}}{(n-1)!}\,,
\end{align*}
where $2\sigma^{2(n-1)}u_0^2e^{2as_n}$ is the new, time-dependent contribution to $\tilde{\Pi}$. As a result
\begin{equation*}
	\mathbb{E}^{u_0}[u_t^2]-(\mathbb{E}^{u_0}[u_t])^2=u_0^2e^{2at}\sum_{n=0}^\infty\frac{(\sigma^2 t)}{n!}-u_0^2e^{2at}=e^{2at}(e^{\sigma^2t}-1)\,,
\end{equation*}
in accordance with the result that one would obtain via It\^o's lemma.

\section{Distribution-valued functionals and contraction maps}\label{Sec: deformation maps}
To establish the notation adopted in this section, the universal tensor algebras over $\mathcal{D}(\R)$ and $\mathcal{D}'(\R)$ are defined as
\begin{equation}
	\mathsf{T}(\R):=\R\oplus\bigoplus_{n=1}^\infty\mathcal{D}(\R)^{\otimes n}\,,\hspace{1.5cm}\mathsf{T}'(\R):=\R\oplus\bigoplus_{n=1}^\infty\mathcal{D}'(\R)^{\otimes n}\,.
\end{equation}
In this section we introduce the key notions that allow us to reformulate the MSR path integral in terms of suitable functionals and contraction maps. The interested reader can consult \cite{dappiaggi2022microlocal, bonicelli2023algebraic} for a more detailed account of these notions.
\begin{definition}\label{Def: distribution-valued functional}
	A distribution-valued functional is a map
	\begin{equation*}
		F:C^\infty(\R)\times C^\infty(\R)\to\mathsf{T}'(\R)\,.
	\end{equation*}
	We denote by $\mathcal{F}(\R,\mathsf{T}')$ the space of distribution-valued functionals and we endow it with the \emph{multilocal product}
	\begin{equation*}
		(F\cdot H)(\psi,\tilde{\psi}):=F(\psi,\tilde{\psi})\otimes H(\psi,\tilde{\psi})\,,\qquad F,H\in\mathcal{F}(\R,\mathsf{T}'),\psi,\tilde{\psi}\in C^\infty(\R)\,,
	\end{equation*}
	where $\otimes$ denotes the tensor product of distributions. 
\end{definition}
In the following, given a functional $F\in\mathcal{F}(\R,\mathsf{T}')$ and arbitrary field configurations $\psi,\tilde{\psi}\in C^\infty(\R)$, we denote by $F^{(1,0)}(\psi,\tilde{\psi}), F^{(0,1)}(\psi,\tilde{\psi}): C^\infty(\R)\to \R$ its Frechet derivatives\footnote{Henceforth we shall assume that all distribution-valued functionals of our interest admit Frechet derivatives of all orders. For a detailed discussion on the topologies on space of functionals considered in quantum field theory we refer to \cite{brouder2018properties}}
\begin{equation*}
	F^{(1,0)}(\psi,\tilde{\psi},\varphi):=\frac{d}{ds}F(\psi+s\varphi,\tilde{\psi})\Big\vert_{s=0}\,,\qquad F^{(0,1)}(\psi,\tilde{\psi},\varphi):=\frac{d}{ds}F(\psi,\tilde{\psi}+s\varphi)\Big\vert_{s=0}\,.
\end{equation*}
Higher order derivatives $F^{(n,m)}$ for $n,m\in\N$ are defined accordingly. For what concerns the adoption of this framework to the problem at hand, it suffices to consider the following sub-class of distribution-valued functionals. 
\begin{definition}\label{Def: multilocal functional}
	Let $\mathcal{F}(\R,\mathsf{T}')$ be the space of distribution-valued functionals as per Definition \ref{Def: distribution-valued functional}. We denote by $\mathcal{F}_{mloc}(\R,\mathsf{T}')$ the space of \emph{multi-local distribution-valued functionals} $F$ of the form
	\begin{align*}
		\langle F(\psi,\tilde{\psi}),f\rangle=\sum_{n=0}^\infty\int_{\R^n}f^n(t_1,\ldots, t_n) \,p^{n}(\psi(t_1),\tilde{\psi}(t_1),\ldots, \psi(t_n),\tilde{\psi}(t_n))\,dt_1\ldots dt_n\,,
	\end{align*}
	for all $f\in\mathsf{T}(\R)$, $\psi,\tilde{\psi}\in C^\infty(\R)$.  Here $p^n$ are fixed polynomials, while $f^n$ denotes the n-th component of $f$ in $\mathsf{T}(\R)$. In addition, we denote by $\mathcal{F}_{loc}(\R,\mathsf{T}')\subset \mathcal{F}_{mloc}(\R,\mathsf{T}')$ the subspace of functionals taking values in $\mathcal{D}'(\R)$.
\end{definition}
A special role is played by the \emph{field functionals}
\begin{equation}\label{Eq: field functionals}
	\langle\Psi(\psi,\tilde{\psi}),f\rangle:=\int_{\R}\psi(t)f(t)\,dt\,,\qquad\langle\tilde{\Psi}(\psi,\tilde{\psi}),f\rangle:=\int_{\R}\tilde{\psi}(t)f(t)\,dt\,.
\end{equation}
and by the \emph{constant functional}
\begin{equation}\label{Eq: constant functional}
	\langle\one(\psi,\tilde{\psi}),f\rangle:=\int_{\R}f(t)\,dt\,.
\end{equation}
These functionals have vanishing components in $\mathsf{T}'\setminus\CD(\R)$.
A direct computation shows that $\Psi,\tilde{\Psi}\in\mathcal{F}_{loc}(\R,\mathsf{T}')$. For any $h\in C^\infty(\R\times \R)$ we define the composition with the field functionals as
\begin{equation*}
	\langle h(\Psi,\tilde{\Psi})(\psi,\tilde{\psi}),f\rangle:=\int_{\R}h(\psi(s),\tilde{\psi}(s))f(s)\,ds\,.
\end{equation*}

\begin{notation}
	For notational ease, since we will work extensively with multilocal, polynomial distribution-valued functionals $F\in\mathcal{F}_{mloc}(\R,\mathsf{T}')$, we shall adopt a notation based on their integral kernel counterpart. Namely, we represent their $n$-th component as $F^{(n)}[\psi,\widetilde{\psi}](t_1,\ldots t_n)$, $t_1,\ldots, t_n\in\R$ and $\psi,\tilde{\psi}\in C^\infty(\R)$, formally defined via the relation
	\begin{equation}
		\langle F^{(n)}[\psi,\widetilde{\psi}],f\rangle=\int_{\R^{n}} F^{(n)}[\psi,\widetilde{\psi}](t_1,\ldots t_n)f(t_1\ldots t_n)\,dt_1\ldots dt_n,\hspace{1cm}f\in\mathcal{D}(\R^{n}).
	\end{equation}	
\end{notation}
Pursuing the analogy with Gaussian integration and Wick's theorem, the computation of expectation values with respect to the free dynamics boils down to maximally
contracting smooth fields $\psi$ with the auxiliary field $\tilde{\psi}$ in all possible ways. This operation can be formalized via a deformation of the algebraic structure of the spaces of distribution-valued functionals.
\begin{definition}\label{Def: deformation map}
	We define the \emph{contraction map} $\Gamma_{\Theta}:\mathcal{F}_{mloc}(\R,\mathsf{T}')\to\mathcal{F}(\R,\mathsf{T}')$ as the operator exponential
	\begin{equation*}
		\Gamma_\Theta:=\exp(D_\Theta)\,,
	\end{equation*}
	where $D_\Theta$ is a differential operator in the Frechet topology of $\mathcal{F}^{mloc}$ whose action reads
	\begin{equation}\label{Eq: D}
		\langle D_\Theta F(\psi,\tilde{\psi}),f\rangle:=\langle F^{(1,1)}(\psi,\tilde{\psi}),f\otimes \Theta\rangle\,,\hspace{1.5cm}f\in\mathsf{T}(\R),\psi,\tilde{\psi}\in C^\infty(\R)\,,
	\end{equation}
	or, in the kernel-based notation,
	\begin{equation*}
		D_\Theta(F)(\psi,\tilde{\psi}):=\int_{[0,T]^2}\Theta(t-t')F^{(1,1)}(\psi,\tilde{\psi})(t,t')\, dtdt'\,.
	\end{equation*} 
\end{definition}
Note that the expression in \eqref{Eq: D} is a convergent integral, since by Definition \ref{Def: multilocal functional} $supp\langle F^{(1,1)}(\psi,\tilde{\psi}),f\rangle\subset supp(f)$. The following standard result justifies out statement that $\Gamma_\Theta$ is a deformation of the algebraic structure induces by the multilocal product.  We refer the reader to \cite{dappiaggi2020algebra, dappiaggi2022microlocal} for the proof, that relies on the exponential form of the deformation. 
\begin{proposition}\label{Prop: homomorphism gamma}
	Let us introduce the space
	\begin{equation*}
		\mathcal{F}_\Theta(\R,\mathsf{T}'):=\{\Gamma_\Theta(F)\,|\; F\in \mathcal{F}(\R,\mathsf{T}')\}\,.
	\end{equation*}
	Then, $\mathcal{F}_\Theta(\R,\mathsf{T}')$ forms an associative algebra if endowed with the product $\star_\Theta:\mathcal{F}_\Theta(\R,\mathsf{T}')\times \mathcal{F}_\Theta(\R,\mathsf{T}')\to \mathcal{F}_\Theta(\R,\mathsf{T}')$ implicitely defined via the identity
	\begin{equation*}
		F\star_\Theta H:=\Gamma_\Theta\left[\Gamma_\Theta^{-1}(F)\cdot \Gamma_\Theta^{-1}(H)\right]\,,\hspace{1.5cm}\forall F,H\in \mathcal{F}_\Theta(\R,\mathsf{T}')\,.
	\end{equation*}
\end{proposition}

\begin{remark}\label{Rem: problem products}
	If $F\in\mathcal{F}_{mloc}(\R,\mathsf{T}')$, Definition \ref{Def: multilocal functional} entails that $F^{(1,1)}(\psi,\tilde{\psi})$ is a distribution supported on the diagonal. Thus, the action of $D_\Theta$ on $\mathcal{F}_{mloc}(\R,\mathsf{T}')$ is a priori hindered by the presence of the ill-defined value of $\Theta$ at the origin. Formally, this can be pictured as having to cope with the product of distributions $\theta(t-t')\delta(t-t')$, which we shall denote with a slight abuse of notation as $\Theta(0)$. The latter is evidently ill-defined since different regularizations $\Theta_\varepsilon\in C^\infty(\mathbb{R})$ of $\Theta$  obtained via the convolution with a mollifier $\rho_\varepsilon\in\mathcal{D}(\R)$ lead to different limiting values $\theta(0):=\lim_{\varepsilon\to 0}\theta_\varepsilon\delta$.
	This arbitrariness is rather natural and reflects an analogous freedom in adopting a specific finite-dimensional approximation of the SDE in the heuristic derivation of the MSR path integral, see \emph{e.g.} \cite{lau2007state}. In Definition \ref{Def: deformation map} we are implicitly considering an arbitrary but fixed choice of $\Theta(0)$. Yet, different extensions of the Heaviside theta lead to different contraction maps. We observe that any particular choice of $\Theta(0)$ does not constitute a loss of generality, we recall that Proposition \ref{Prop: homomorphism gamma} guarantees that different choices lead to isomorphic algebras. More specifically, for any but fixes $\Theta,\bar{\Theta}$ such that $\Theta(t)=\bar{\Theta}(t)$ for all $t\neq0$, the isomorphism between $(\mathcal{F}_\Theta(\R,\mathsf{T}'),\star_{\Theta})$ and $(\mathcal{F}_{\bar{\Theta}}(\R,\mathsf{T}'),\star_{\bar{\theta}})$ is realized by the map $\Gamma_{\bar{\Theta}-\Theta}$ in the sense that, for any $F_1,F_2\in\mathcal{F}_\Theta(\R,\mathsf{T}')$ 
	\begin{align*}
		\Gamma_{\bar{\Theta}-\Theta}\left[\Gamma_{\bar{\Theta}-\Theta}^{-1}(F_1)\star_{\Theta}\Gamma_{\bar{\Theta}-\Theta}^{-1}(F_2)\right]=F_1\star_{\bar{\Theta}} F_2\,.
	\end{align*}
	This sheds some light on an interesting interpretation of the arbitrariness in defining a stochastic integration theory as an example of renormalization involving finite counterterms. We postpone further investigations on the connection between stochastic integral prescriptions and renormalization to a future work.
\end{remark}
We introduce a nonlocal version of the contraction map of Definition \eqref{Def: deformation map}
\begin{equation*}
	\Gamma_{\Theta}^{mloc}:\mathcal{F}_{mloc}(\R,\mathsf{T}')\,\hat{\otimes}\,\mathcal{F}_{mloc}(\R,\mathsf{T}') \rightarrow\mathcal{F}_{mloc}(\R,\mathsf{T}')\,,
\end{equation*} 
whose action on $F_1,F_2\in\mathcal{F}_{mloc}(\R,\mathsf{T}')$ reads
\begin{align}\label{Eq: nonlocal gamma}
	&\Gamma_{\Theta}^{mloc}(F_1\hat{\otimes} F_2):=m\circ e^{D^{mloc}_{\Theta}}(F_1\hat{\otimes} F_2)\,,\\
	&D^{mloc}_{\Theta}(F_1\hat{\otimes} F_2)((\psi,\tilde{\psi})\otimes (\psi',\tilde{\psi}'))\nonumber\\
	&\hspace{1.8cm}:=\int_{\R^2}\Theta(t,t')\left(F_1^{(1,0)}(t)\hat{\otimes} F_2^{(0,1)}(t')+ F_1^{(0,1)}(t)\,\hat{\otimes}\,  F_2^{(1,0)}(t')\right)((\psi,\tilde{\psi})\otimes (\psi',\tilde{\psi}'))\,dtdt'\,,\nonumber
\end{align}
for $\psi,\tilde{\psi}, \psi',\tilde{\psi}'\in C^\infty(\R)$. Here $\hat{\otimes}$ denotes the tensor product of functionals
\begin{equation*}
	\langle(F_1\hat{\otimes}F_2)((\psi,\tilde{\psi})\otimes (\psi',\tilde{\psi}')),f\otimes g\rangle:=\langle F_1(\psi,\tilde{\psi}),f\rangle\langle F_2(\psi',\tilde{\psi}'),g\rangle\,,
\end{equation*}
while
$m:\mathcal{F}_{mloc}(\R,\mathsf{T}')\hat{\otimes}\mathcal{F}_{mloc}(\R,\mathsf{T}') \rightarrow\mathcal{F}_{mloc}(\R,\mathsf{T}')$ is the pull-back on the diagonal of $C^\infty(\R)\times C^\infty(\R)$.
The contraction maps $\Gamma_{\Theta}$ and $\Gamma_{\Theta}^{mloc}$ are linked by the identity
\begin{equation}\label{Eq: interplay local and nonlocal Gamma}
	\Gamma_{\Theta}(F_1\cdot F_2)=\Gamma_{\Theta}^{mloc}\left(\Gamma_{\Theta}(F_1)\hat{\otimes}\Gamma_{\Theta}(F_2)\right),
\end{equation}
for all $F_1,F_2\in\mathcal{F}_{loc}(\R,\mathsf{T})$.
At the practical level, the relation in \eqref{Eq: interplay local and nonlocal Gamma} states that we can separate \emph{local contractions} of fields $\psi$ and $\widetilde{\psi}$ belonging to the same local distribution-valued functional $F_i$ from \emph{nonlocal} contractions between $F_1$ and $F_2$.

\endappendix

\bibliographystyle{alpha}

\bibliography{bibliography}

\end{document}